\documentclass[11pt ]{article}
\usepackage[numbers]{natbib}
\usepackage[english, french]{babel}
\usepackage[latin1]{inputenc}
\usepackage{geometry}
\usepackage{graphicx}		
\usepackage{amsmath}
\usepackage{amsfonts} 
\usepackage{amssymb}
\usepackage{amsthm}
\usepackage[T1]{fontenc}
\usepackage{hyperref}
\usepackage{stmaryrd}
\usepackage{dsfont}
\usepackage[all]{xy}
\usepackage{sectsty}
\usepackage{multicol}
\usepackage{eucal}
\usepackage{tikz}
\usepackage{yhmath}
\usepackage{makerobust}
\usepackage{fancyhdr}
\usepackage{eso-pic}
\usepackage{footnote}

\sectionfont{\centering}

\def\R{\mathbb{R}}
\def\N{\mathbb{N}}
\def\Z{\mathbb{Z}}
\def\Q{\mathbb{Q}}

\DeclareMathOperator{\Frac}{Frac}
\DeclareMathOperator{\Fil}{Fil}
\DeclareMathOperator{\fil}{Fil}
\DeclareMathOperator{\id}{id}

\DeclareMathOperator{\Car}{Car}

\DeclareMathOperator{\gr}{gr}

\DeclareMathOperator{\supp}{Supp}
\DeclareMathOperator{\suppi}{Supp_\infty}

\newtheorem{theorem}{Théoreme}[section]
\newtheorem{lemma}[theorem]{Lemme}
\newtheorem{prop}[theorem]{Proposition}
\newtheorem{cor}[theorem]{Corollaire}
\newtheorem{definition}[theorem]{Définition}
\newtheorem{remark}[theorem]{Remarque}
\theoremstyle{definition}
\newtheorem{example}[theorem]{Exemple}

\def\V{\mathcal{V}}
\def\I{\mathcal{I}}

\def\B{\mathcal{B}}
\def\a{\mathcal{A}}

\def\D{\mathcal{D}}
\def\Dk{\widehat{\mathcal{D}}^{(0)}_{\X, k}}
\DeclareRobustCommand{\Dks}{\mathcal{D}_{X, k}}

\DeclareRobustCommand{\Dkq}{\widehat{\mathcal{D}}^{(0)}_{\X, k, \Q}}
\def\X{\mathfrak{X}}
\def\m{\mathfrak{m}}
\def\Spf{\mathrm{Spf}\,}

\def\Spec{\mathrm{Spec}\,}
\def\Sp{\mathrm{Sp}\,}

\def\O{\mathcal{O}}

\def\d{\mathrm{d}}

\def\Nb{\overline{N}_k}

\def\E{\mathcal{E}}
\def\F{\mathcal{F}}
\def\Eo{\mathcal{E}^\circ}
\def\M{\mathcal{M}}
\def\Nn{\mathcal{N}}
\def\L{\mathcal{L}}
\def\Ext{\mathcal{E}xt}
\def\Hom{\mathcal{H}om}
\def\Irr{\mathrm{Irr}}
\def\CC{\mathrm{CC}}

\def\vp{\V\langle \partial \rangle^{(k , r)}}
\def\vqp{K\langle \partial \rangle^{(k , r)}}

\def\Di{\mathcal{D}_{\X, \infty}}

\def\Dku{\widehat{\mathcal{D}}^{(0)}_{\X, k , \Q}(U)}
\def\Nm{\overline{N}_\mu}
\def\nb{\overline{n}_k}
\def\nm{\overline{n}_\mu}

\DeclareRobustCommand{\Ek}{\mathcal{E}_k(U)}

\DeclareRobustCommand{\Fkr}{\mathcal{F}_{k , r}(U)}
\def\Fi{\F_\infty}
\DeclareRobustCommand{\Fir}{\F_{\infty , r}}

\def\Mkr{\tilde{\mathcal{M}}_{k,r}}
\def\Mr{\tilde{\mathcal{M}}_{\infty , r}}

\DeclareRobustCommand{\Fkrs}{\mathcal{F}_{k , r}^*}

\DeclareRobustCommand{\Firs}{\F_{\infty , r}^*}

\def\Mkrs{\mathcal{M}_{k,r}^*}
\def\Mrs{\mathcal{M}_r^*}

\title{Cycle caractéristique pour les $\mathcal{D}$-modules coadmissibles sur une courbe formelle}
\author{Raoul Hallopeau}
\date{}
\begin{document}

\maketitle

\selectlanguage{english}
\begin{abstract}
Let $\mathfrak{X}$ be a formal smooth quasi-compact curve over a complete discrete valuation ring of mixed characteristic.
We consider over $\mathfrak{X}$ the sheaves of differential operators $\widehat{\mathcal{D}}^{(0)}_{\mathfrak{X}, k , \mathbb{Q}}$ with a congruence level $k \in \mathbb{N}$ and their projective limit $\mathcal{D}_{\mathfrak{X}, \infty} = \varprojlim_k \widehat{\mathcal{D}}^{(0)}_{\mathfrak{X}, k , \mathbb{Q}}$.
In this article, we define a characteristic variety for coadmissible $\mathcal{D}_{\mathfrak{X}, \infty}$-modules as a closed subset of the cotangent space $T^*\mathfrak{X}$.
For this purpose, we introduce a microlocalization sheaf of $\mathcal{D}_{\mathfrak{X}, \infty}$ in which the derivation is locally invertible.
We deduce a notion of "sub-holonomicity" for coadmissible $\mathcal{D}_{\mathfrak{X}, \infty}$-modules which is equivalent to being generically an integrable connection.
Finally, we associate characteristic cycles to sub-holonomic modules proving that the latter are of finite length.
\end{abstract}

\selectlanguage{french}
\tableofcontents

\section{Introduction}\label{section1}

Soit $\X$ un schéma formel lisse sur un anneau complet $\V$ de valuation discrète et de caractéristique mixte.
Huyghe-Schmidt-Strauch ont introduit dans l'article \cite{huyghe} un faisceau $\Di = \varprojlim_k \Dkq$ d'opérateurs différentiels obtenu en rajoutant des niveaux de congruence $k \in \N$ aux faisceaux d'opérateurs arithmétiques de Berthelot.
Les modules à considérer, analogues des modules cohérents sur un sch\'ema classique, sont les $\Di$-modules coadmissibles.
Plus précisément, si $U$ est un ouvert affine de $\X$ muni de coordonnées locales dont les dérivations sont notées $\partial_1 , \dots , \partial_d$ et si $\varpi$ une uniformisante de $\V$, alors
\[ \Dk(U) = \left\{ \sum_{\alpha \in \N^d} a_\alpha \cdot \partial_1^{\alpha_1} \dots \partial_d^{\alpha_d},~~ a_\alpha \in \O_\X(U) ~,  ~~|a_\alpha| \cdot \varpi^{-k | \alpha|} \underset{ | \alpha| \to \infty}{\longrightarrow} 0 \right\} \]
avec $| \alpha| := \alpha_1 + \dots + \alpha_d \in \N$.
Notons $\O_{\X , \Q} := \O_\X \otimes_\V K$ et $\Dkq := \Dk \otimes_\V K$ o\`u $K := \text{Frac}(\V)$.
Les inclusions d'algèbres $\widehat{\mathcal{D}}^{(0)}_{\X , k +1 , \Q}(U) \subset \Dkq(U)$ induisent des morphismes de transition $\widehat{\mathcal{D}}^{(0)}_{\X , k+1 , \Q} \to \Dkq$ pour tout niveau de congruence $k$.
Le faisceau $\Di := \varprojlim_k \Dkq$ est localement donn\'e par la $K$-algèbre de Fréchet-Stein
\[ \Di(U) =\displaystyle \left\{ \sum_{\alpha \in \N^d} a_\alpha \cdot \partial_1^{\alpha_1} \dots \partial_d^{\alpha_d} :  \,  a_\alpha \in \O_{\X, \Q} (U)~~ \mathrm{tq} ~~ \forall \eta>0,~  \lim_{|\alpha| \to \infty} |a_\alpha| \cdot \eta^{|\alpha|} =0 \right\}.\]
Par d\'efinition, un $\Di$-module est coadmissible s'il est isomorphe \`a une limite projective de $\Dkq$-modules cohérents $\M_k$ tels que $\Dkq \otimes_{\widehat{\mathcal{D}}^{(0)}_{\X, k+1 , \Q}} \M_{k+1} \simeq \M_k$.
La catégorie abélienne des $\Di$-modules coadmissibles ainsi obtenue est équivalente \`a la catégorie des modules coadmissibles sur l'espace de Zariski-Riemann associ\'e \`a $\X$.
Mentionnons qu'Ardakov-Wadsley ont développé dans \cite{ardakov} une théorie de $\wideparen{\D}$-modules coadmissibles sur un espace analytique rigide lisse.
Leur construction coïncide avec celle de Huyghe-Schmidt-Strauch pour la fibre générique $\X_K$ du sch\'ema formel lisse $\X$.
Nous introduisons dans cet article une variété caractéristique pour les $\Di$-modules coadmissibles dans le cas d'une courbe formelle lisse $\X$. 
Cette variété est un outil tr\`es important dans l'\'etude des $\D$-modules d'une variété complexe lisse $X$.
En effet, les $\D$-modules holonomes sont classiquement caractérisés par le fait d'avoir une variété caractéristique de dimension inférieure ou égale \`a celle de $X$.
Néanmoins, contrairement \`a la situation complexe, la variété caractéristique n'est a priori pas suffisante pour obtenir une bonne notion d'holonomie en géométrie analytique rigide.
Il est par exemple nécessaire d'ajouter une structure de Frobenius aux $\D$-modules arithmétiques de Berthelot afin de démontrer des énoncés de finitude.
Soulignons que l'article \cite{ABW2} de Ardakov-Bode-Wadsley  introduit une catégorie de modules faiblement holonomes de $\wideparen{\D}$-modules coadmissibles en utilisant la caractérisation cohomologique classique des modules holonomes.
Cependant, cette catégorie demeure trop vaste, car les $\wideparen{\D}$-modules faiblement holonomes ne sont par exemple pas tous de longueur finie.
Par ailleurs, Bode a introduit dans la continuité de l'article \cite{Bode} une notion d'holonomie pour les $\wideparen{\D}$-modules coadmissibles suivant la construction de Caro pour les $\D$-modules arithmétiques, sans utiliser de variété caractéristique.
Une telle variété pour les $\D$-modules coadmissibles demeure néanmoins un invariant important \`a définir.
Sur le schéma formel lisse $\X$, les opérateurs différentiels considérés ne sont pas d'ordre fini.
Par conséquent, la méthode classique des variétés complexes pour définir la variété caractéristique ne s'applique plus.
Une autre approche naturelle, adoptée dans ce papier, r\'eside dans l'utilisation de techniques de microlocalision appliquées au faisceau $\Di$ des opérateurs différentiels.
Nous nous limitons dans ce papier \`a la dimension une, unique cas o\`u nous obtenons un cycle caractéristique pour les $\Di$-modules coadmissibles.
Mentionnons qu'une construction assez similaire a déjà été faite pour les opérateurs différentiels arithmétiques par Abe dans \cite{abe} et par Abe-Marmora dans \cite{adriano}.
L'obtention des multiplicit\'es verticales pour les $\Di$-modules coadmissibles est une adaptation de leurs arguments \`a notre cas.
Par ailleurs, ce travail devrait permettre dans un futur article de définir une vari\'et\'e caract\'eristique dans le cas des courbes analytiques rigides lisses admettant un mod\`ele formel lisse en passant \`a l'espace de Zariski-Riemann.

\vspace{0.4cm}

D\'etaillons maintenant le contenu de ce papier.
Soit $\X$ une $\V$-courbe formelle lisse connexe et quasi-compacte de fibre sp\'eciale $X$.
Des variétés caractéristiques $\Car(\M_k) \subset T^* X$ ont été introduites dans l'article \cite{hallopeau1} pour les $\Dkq$-modules cohérents $\M_k$ en copiant la d\'efinition classique complexe après r\'eduction  modulo $\varpi$ des objects (ce qui permet de se ramener au cas d'opérateurs différentiels finis).
La section \ref{section2} de cet article consiste en des rappels sur la construction de ces variétés caractéristiques, ainsi que sur les $\Di$-modules coadmissibles et les modules faiblement holonomes.
Considérons maintenant un $\Di$-module coadmissible $\M = \varprojlim_k \M_k$  non nul.
Il est important de remarquer que les morphismes de transition $\M_{k+1} \to \M_k$ n'induisent pas de morphismes naturels intéressants entre les différentes variétés caractéristiques $\Car(\M_k)$.
En effet, nous perdons les informations sur les morphismes de transition en les réduisant modulo $\varpi$.
Ainsi, nous ne pouvons obtenir une variété caractéristique pour le module coadmissible $\M = \varprojlim_k \M_k$ \`a partir des variétés caractéristiques $\Car(\M_k)$.
Par ailleurs, les $\Di$-modules \og holonomes \fg \, doivent vérifier la condition suivante afin d'être de longueur finie.
Pour un opérateur différentiel non nul $P$ de $\Di$, le $\Di$-module coadmissible $\Di / P$ doit être \og holonome \fg\, seulement si $P$ est un opérateur différentiel fini.
En effet, si $P$ n'est pas un opérateur fini, alors le $\Dkq$-module cohérent $\Dkq/ P$ est un module de longueur finie augmentant avec $k$ tout en \'etant localement un $\O_{\X , \Q}$-module libre de dimension de plus en plus grande.
Nous ne pouvons donc pas espérer avoir de bonnes propriétés de finitude pour le module coadmissible $\Di / P = \varprojlim_k \Dkq /P$.
Notons $(x , \xi)$ un système de coordonnées locales sur $T^*U$ associée à la coordonnée étale de $U$.
Si $P = \sum_{n = 0}^d a_n \cdot \partial^n$ est un opérateur différentiel fini d'ordre $d$, alors pour $k$ suffisamment grand
\[ \Car \left( \Dkq / P \right) \cap T^*U = \{ (x, \xi) \in T^*U : \sigma(\bar{P})(x,\xi) = a_d(t) \cdot \xi^d = 0 \}. \]
Ces variétés caractéristiques ne dépendent donc plus de $k$ à partir d'un certain niveau de congruence $k$.
La figure \ref{figure1} repr\'esente la variété caractéristique des $\Dkq$-modules cohérents $\M_k$, o\`u $x_1 , \dots , x_s$ désignent les zéros du coefficient dominant $a_d$ de $P$ dans l'ouvert $U$.
Soulignons que dans la théorie complexe des $\D$-modules (et pour un niveau de congruence $k$ fix\'e), tout quotient de $\D$ par un opérateur différentiel $P$ est holonome en dimension un.
Les arguments classiques ne peuvent donc s'appliquer au faisceau $\Di$ et cela rajoute une difficulté supplémentaire \`a l'introduction d'une variété caractéristique pour les $\Di$-modules coadmissibles.

\begin{figure}[h!]
\begin{center}
\caption{$\Car (\Dkq / P) \cap T^* U$}
\begin{tikzpicture}
\label{figure1}
\draw[thick][->] (-2,0) -- (8,0);
\draw (8.2,0) node[right] {$x$};
\draw [thick][->] (0,-1) -- (0,3.5);
\draw (0,3.7) node[above] {$\xi$};
\draw (0,0) node[below right] {$0$};

\draw[red][thick] (-1.5,0) -- (7,0);

\draw[red][thick] (-1, -1) -- (-1,3) ;
\draw (-1,0) node[below right]{$x_1$} ;

\draw[red][thick] (1.8, -1) -- (1.8,3) ;
\draw (1.8,0) node[below right]{$x_2$} ;

\draw[red][thick] (3.5 , -1) -- (3.5,3) ;
\draw (3.5,0) node[below right]{$x_i$} ;

\draw[red][thick] (6, -1) -- (6,3) ;
\draw (6,0) node[below right]{$x_s$} ;

\draw (7,3) node{$T^*U$} ;

\end{tikzpicture}
\end{center}
\end{figure}

Afin de r\'esoudre les deux points expliqu\'es ci-dessus, nous construisons dans les sections \ref{section3} et \ref{section4} de cet article un microlocalisé de l'alg\`ebre $\Di(U)$ des opérateurs différentiels \`a convergence rapide dont les principales propriétés sont données dans le théorème suivant.

\begin{theorem}
Il existe une $K$-algèbre microlocalisé $\Fi(U)$ de $\Di(U)$, obtenue comme union croissante de $K$-algèbres de Fr\'echet-Stein, vérifiant la condition d'inversibilit\'e suivante : un \'el\'ement $P$ de $\Di(U)$ est inversible dans $\Fi(U)$ si et seulement si $P = \sum_{n=0}^d a_n \cdot \partial^n$ est un opérateur fini d'ordre $d$ de $\Di(U)$ dont le coefficient dominant $a_d$ est inversible dans $\O_{\X , \Q}(U)$.
\end{theorem}

Le support du module $\Fi/ P$ permet de distinguer les opérateurs infinis des opérateurs finis.
En effet, si $P$ est un opérateur infini, alors $(\supp \Fi / P ) = \X$ puisque $P$ n'est pas inversible dans l'algèbre $\Fi(\X)$. Dans ce cas, $\dim (\supp \Fi / P) = 1$.
Si maintenant $P$ est un opérateur fini de coefficient dominant $a_d$, alors le support $(\supp \Fi / P)$ est l'union finie des zéros $x_1, \dots , x_s$ de la fonction $a_d$ dans $\X$.
Ainsi, $\dim (\supp \Fi / P) = 0$ et les points de ce support sont les abscisses des composantes irréductibles verticales des variétés caractéristiques des modules $\Dkq / P$ données dans la figure \ref{figure1} pour $k$ suffisamment grand.
Notons $T^*\X$ le fibré cotangent de $\X$.
Le microlocalis\'e $\Fi$ permet d'établir dans la section \ref{section5} le resultat suivant.

\begin{theorem}
Nous pouvons associer \`a tout $\Di$-module coadmissible $\M = \varprojlim_k \Dkq$ une variété caractéristique $\Car(\M)$, partie fermée de $T^*\X$, vérifiant l'inégalité de Bernstein : si $\M \neq 0$, alors $\Car(\M)$ est égal \`a $T^*\X$ ou bien équidimensionnel de dimension 1.
\end{theorem}

Par exemple, la variété caractéristique $\Car(\Di / P)$ coincide avec les vari\'et\'es de la figure \ref{figure1} (pour $k$ assez grand) lorsque $P$ un op\'erateur différentiel fini non nul.
Dans le cas contraire, $\Car(\Di / P) = T^*\X$.
Nous étudions enfin dans la section \ref{section6} la notion d'holonomie induite par cette variété caractéristique.

\begin{definition}
Un $\Di$-module coadmissible $\M$ est sous-holonome si $\dim(\Car(\M)) \leq 1$.
\end{definition}

D'après ce qui vient d'être dit, le module coadmissible $\Di / P$ est sous-holonone si et seulement si $P$ est un op\'erateur différentiel fini non nul.
Nous démontrons dans la partie \ref{section6} les propriétés suivantes des $\Di$-modules sous-holonomes.

\begin{theorem}\,
\begin{enumerate}
\item
Un $\Di$-module coadmissible est sous-holonome si et seulement si il est g\'en\'eriquement une connexion intégrable.
\item
Il est possible d'associer \`a tout $\Di$-module sous-holonome un cycle caractéristique induit par sa variété caractéristique.
\item
Les $\Di$-modules sous-holonones définissent une sous-catégorie abélienne et artinienne des modules faiblement holonomes de Ardakov-Bode-Wadsley.
\end{enumerate}
\end{theorem}

Cette catégorie de $\Di$-modules sous-holonomes n'est néanmoins pas stable par certaines des six opérations usuelles, d'o\`u l'utilisation du terme \og sous-holonomes \fg.
Par exemple, l'article \cite{bitoun} de Bitoun-Bode présente des exemples explicites d'images directes de modules à connexions intégrables qui ne sont pas même coadmissibles.
Soit $X_K = \Sp (K \langle x \rangle)$ le disque analytique rigide sur $K$, $U_K = X_K \backslash \{0\}$, $j : U_K \hookrightarrow X_K$ l'injection naturelle et $\partial = \frac{d}{d x}$.
Notons $P_\lambda = x \cdot \partial - \lambda$ pour $\lambda \in K$.
Le $\wideparen{\D}_{U_K}$-module coadmissible $\M_\lambda = \wideparen{\D}_{U_K}  / P_\lambda \simeq \O_{U_K} \cdot x^\lambda$ est une connexion intégrable puisque $x$ est inversible dans l'ouvert $U_K$.
Le théorème 1.1 de \cite{bitoun} nous dit que le module $j_*\M_\lambda$ est coadmissible si et seulement si $\lambda$ est de type positif.
Il existe cependant des scalaires de type zéro.

\section*{Notations}

\begin{enumerate}
\item[$\bullet$]
$\V$ est un anneau complet de valuation discrète de caractéristique mixte $(0,p)$, d'idéal maximal $\m$ et de corps résiduel $\kappa$ supposé parfait.
On note $\varpi$ une uniformisante de $\V$, $| \cdot |$ la valeur absolue normalisée (autrement dit, $|\varpi| = p^{-1}$) et $K =\Frac(\V)$ son corps des fractions.
\item[$\bullet$]
$X$ est une $\kappa$-courbe lisse connexe quasi-compacte et $x \in X$ est un point fermé.
\item[$\bullet$]
$\X$ est un $\V$-schéma formel lisse relevant $X$, d'idéal de définition engendré par l'uniformisante $\varpi$ et $\X_K$ est l'espace analytique rigide associé à $\X$.
\item[$\bullet$]
$t$ est un relèvement local sur $\O_\X$ d'une uniformisante en $x$ ($\O_{X , x}$ est un anneau de valuation discrète puisque $X$ est une courbe). Alors $\d t$ est une base de $\Omega^1 _{\X,x}$. On note $\partial$ la dérivation associée.
\item[$\bullet$]
$U$ est un ouvert affine de $\X$ contenant $x$ sur lequel on dispose d'une coordonnée étale $t$ associée à $x$.
\item[$\bullet$]
Sauf mention contraire, les idéaux et les modules considérés sont tous à gauche.
\end{enumerate}

\section{$\Di$-modules faiblement holonomes}\label{section2}

La partie \ref{section2.1} regroupe quelques rappels sur les $\Dkq$-modules holonomes introduits dans \cite{hallopeau1} pour un niveau de congruence $k \in \N$ fixé.
La section \ref{section2.2} rappelle les principales propriétés du faisceau $\Di = \varprojlim_k \Dkq$ et des $\Di$-modules coadmissibles.
Enfin, on introduit dans la partie \ref{section2.3} la notion de $\Di$-modules coadmissibles faiblement holonomes au sens de Ardakov-Bode-Wadsley dans l'article \cite{ABW2}.

\subsection{$\Dkq$-modules holonomes}\label{section2.1}

Le faisceau $\D^{(0)}_{\X , k}$ est défini comme un sous-faisceau dépendant d'un paramètre $k \in \N$ appelé \og niveau de congruence \fg\, du faisceau usuel $\mathcal{D}^{(0)}_\X$ des opérateurs différentiels sur la courbe formelle lisse $\X$.
Localement, $\D^{(0)}_{\X , k}(U)$ est la $\V$-algèbre engendrée par $\O_\X(U)$ et par la dérivation $\varpi^k\partial$.
Plus précisément,
\[ \D^{(0)}_{\X , k}(U) = \left\{ \sum_{n \in \N}  a_n \cdot (\varpi^{k}\partial)^n , ~~ a_n \in \O_\X(U), ~~ a_n = 0 ~~ \mathrm{pour} ~~ n >>0 \right\}. \]
Autrement dit, $\D^{(0)}_{\X , k}(U)$ est le $\O_\X(U)$-module libre de base les puissances de $\varpi^{k} \partial$, ie $\D^{(0)}_{\X , k}(U) = \bigoplus_{n\in \N} \O_\X(U) \cdot (\varpi^k \partial)^n$.
La dérivation $\varpi^k \partial$ est un endomorphisme $\V$-linéaire de $\O_\X(U)$.
La loi multiplicative de la $\V$-algèbre des opérateurs différentiels $\D^{(0)}_{\X , k}(U)$ est simplement la composition d'applications linéaires.
En pratique, on a $(\varpi ^k \partial) ^n \cdot (\varpi ^k \partial) ^m = (\varpi ^k \partial) ^{n+m}$ et le crochet
\[ \forall a \in \O_\X(U), ~~ [\varpi^k \partial , a ] = \varpi^k \partial \cdot a - a \cdot \varpi^k \partial = \varpi^k \partial(a) . \]
On retrouve $\mathcal{D}^{(0)}_\X$ lorsque $k=0$.
L'ouvert affine $U$ de $\X$ est le spectre formel d'une $\V$-algèbre $\O_\X(U)$ telle que $\O_{\X , \Q}(U) := \O_\X(U)\otimes _V K$ soit une $K$-algèbre affinoïde.
Ces algèbres sont intègres puisque $\X$ est connexe et lisse.
Ainsi, la norme spectrale de l'algèbre $\O_\X(U)$, notée dans la suite $| \cdot |$ et définie par $|f| = p^{-k}$ si $f \in \varpi^k \cdot \O_\X(U) \backslash \varpi^{k+1} \cdot \O_\X(U)$, est complète et non-archimédienne.
Soit $\Dk := \varprojlim_i \left(\D^{(0)}_{\X , k}/ \varpi^{i+1} \D^{(0)}_{\X , k}\right)$ le complété $\varpi$-adique de $\D^{(0)}_{\X , k}$. Localement,
\[ \Dk(U) = \left\{ \sum_{n = 0 }^{\infty}  a_n \cdot (\varpi^k\partial)^n , ~~ a_n \in \O_\X(U) , ~~| a_n |\underset{n \to \infty}{\longrightarrow} 0 \right\} . \]
On note $\O_{\X , \Q} := \O_\X \otimes_\V K$ et $\Dkq := \Dk \otimes_\V K$.
Les éléments de $\Dkq(U)$ sont les opérateurs différentiels convergents pour la topologie $p$-adique à coefficients dans la $K$-algèbre affinoïde $\O_{\X , \Q}(U)$.
Il est démontré dans l'article \cite{huyghe} que les algèbres $\Dk(U)$ et $\Dkq(U)$ sont noetheriennes et que $\Dk$ et $\Dkq$ sont des faisceaux d'anneaux cohérents.
Pour tous entiers $k'>k$, il est clair que $\widehat{\mathcal{D}}^{(0)}_{\X, k' }(U) \subset \Dk(U)$ et $\Dk(U)$ est une sous-algèbre de l'algèbre des opérateurs différentiels cristallins $\widehat{\mathcal{D}}^{(0)}_\X(U) = \widehat{\mathcal{D}}^{(0)}_{\X, 0}(U)$.

\begin{definition}
Soit $P = \sum_{n =0}^\infty a_n \cdot (\varpi^k \partial)^n$ un opérateur différentiel de $\Dkq(U)$. On note $| P |_k := \max_{n \geq 0} \{ |a_n| \}$ et $\Nb(P) := \max\{ n \in \N : |a_n| = |P|_k \}$.
\end{definition}

On v\'erifie facilement que $\Dk(U) = \{ P \in \Dkq(U), ~ |P|_k \leq 1 \}$.
De plus, l'entier $\Nb(P)$ coïncide avec l'ordre de $ \bar{P} = (\alpha P \mod \varpi)$ dans $\Dks(U)$ pour tout scalaire $\alpha \in K$ vérifiant $|\alpha P|_k = 1$, où $\Dks := \Dk\otimes_\V \kappa$ est un faisceau d'opérateurs différentiels sur la fibre spéciale $X$ de $\X$.
On appelle donc coefficient dominant de $P$ son coefficient d'indice $\Nb(P)$.
Par ailleurs, il est montré dans \cite[lemme 2.5]{hallopeau1} que la norme $| \cdot |_k$ et la fonction $\Nb$ ne dépendent pas du choix de la coordonnée locale de $U$, ni de l'ouvert $U$.
Les résultats suivants correspondent à \cite[proposition 2.6, proposition 2.7]{hallopeau1}.

\begin{prop}\label{prop2.2.3}~
\begin{enumerate}
\item
Les algèbres $\Dk(U)$ et $\Dkq(U)$ sont complètes pour $| \cdot |_k$ et la norme induite sur tout $\Dkq$-module cohérent est complète.
\item
Pour tous opérateurs différentiels $P$ et $Q$ de $\Dkq(U)$, on a $| P Q |_k = | P |_k \cdot |Q|_k$ et $\Nb(PQ) = \Nb(P) + \Nb(Q)$.
\item
Un opérateur différentiel $P = \sum_{n=0}^\infty a_n \cdot (\varpi^k\partial)^n$ de $\Dkq(U)$ est inversible dans $\Dkq(U)$ si et seulement si $\Nb(P) = 0$ et si $a_0 \in \O_{\X , \Q}(U)^\times$.
\end{enumerate}
\end{prop}

On rappelle que $\X$ et $X$ ont le même espace topologique ; on identifie $U$ à un ouvert affine de $X$.
Soit $\partial_k$ l'image de la dérivation $\varpi^k \partial$ dans la réduction $\D_{X,k}(U)$ de $\Dk(U)$ modulo $\varpi$. Alors $\D_{X,k}(U) = \bigoplus_{n \in \N} \O_X(U) \cdot \partial_k^n$.
Lorsque $k \geq 1$, l'algèbre $\Dks(U)$ est commutative puisque l'algèbre $(\Dkq(U) , | \cdot |_k)$ est quasi-abélienne d'après la proposition \ref{propqa} ci-dessous :
\[ \exists \gamma\in [0 , 1[~ \text{tel que}~\forall P, Q \in \Dkq(U), ~ |PQ-QP|_k \leq \gamma \cdot |P|_k \cdot |Q|_k . \]
Ces propriétés sont fausses lorsque $k = 0$. Par exemple, $[\partial , t] = 1$ dans $\D_\X^{(0)}(U)$ et dans $\D_X(U)$.
On munit le faisceau $\Dks$ de la filtration croissante donnée localement par l'ordre des opérateurs différentiels : $\forall m \in \N, ~~ \Fil^m (\D_{U , k}) = \bigoplus_{n = 0}^m \O_U\cdot \partial_k^n$.
On note $\gr \Dks = \bigoplus_{m\in \N} \gr_m \Dks$ le gradué associé et $\xi_k$ l'image de la dérivation $\partial_k$ dans $\gr_1 (\D_{U , k})$.
Localement, $\gr (\D_{U , k}) \simeq \O_U[\xi_k]$ est un anneau de polynômes en une variable sur $\O_U$.
En particulier, le fibré cotangent $T^* X$ de $X$ est isomorphe à $\Spec (\gr (\Dks))$ en tant que $\kappa$-schéma.
On identifie ces deux schémas dans la suite. On note $\pi_X : T^* X \to X$ la projection canonique.
Soit $P = \sum_{n=0}^d a_n \cdot \partial_k^n$ un opérateur différentiel de $\D_{X, k}(U)$ d'ordre $d$. On lui associe un élément du gradué $\gr \Dks(U)$ appelé \textit{symbole principal} de $P$ par
\[ \sigma(P) := a_d \cdot \xi_k^d \in \gr_d\D_{X, k}(U) .\]
Une filtration $(\fil^\ell E)_{\ell\in\N}$ d'un $\Dks$-module quasi-cohérent à gauche $E$ est une suite croissante $(\fil^\ell E)_\ell$ de sous-$\O_X$-modules quasi-cohérents de $E$ telle que
\begin{enumerate}
\item
$ E = \bigcup_{\ell\geq 0} \fil^\ell E$ ;
\item
$\forall n , \ell \in \N$, $(\fil^n \Dks) \cdot (\fil^\ell E) \subset \fil^{\ell + n} E$.
\end{enumerate}
Le gradué $\gr E$ pour une telle filtration est un $\gr \Dks$-module. La filtration est appelée \textit{bonne filtration } si le gradué $\gr E$ est un $\gr\Dks$-module cohérent.
Puisque la courbe $X$ est quasi-compacte, tout $\Dks$-module cohérent $E$ admet une bonne filtration globale.
Etant donn\'e une telle bonne filtration, on associe à $E$ le $\O_{T^* X}$-module cohérent
\[ \tilde{E} = \O_{T^*X} \otimes_{\pi_X^{-1} (\gr \Dks)} \pi_X^{-1} (\gr E) .\]
La variété caractéristique de $E$ est par définition le support du $\O_{T^* X}$-module cohérent $\tilde{E}$ : $\Car E := \supp \tilde{E}$.
C'est une sous-variété fermée du fibré cotangent $T^*X$ puisque le module $\tilde{E}$ est cohérent.
Elle est de plus indépendante du choix de la bonne filtration choisie.

Soit maintenant $\E$ un $\Dkq$-module cohérent à gauche. Un \textit{modèle entier} de $\E$ est un $\Dk$-module cohérent $\Eo$ sans $\varpi$-torsion tel que $\E \simeq \Eo \otimes_\V K$. Puisque $\E$ est cohérent, il existe un modèle entier $\Eo$ d'après \cite[proposition 3.3.4]{berthelot1}.
La réduction $\Eo  \otimes_\V \kappa$ modulo $\varpi$ de $\Eo$ est un $\Dks$-module cohérent.

\begin{definition}
La variété caractéristique de $\E$ est la variété caractéristique du $\Dks$-module cohérent $\Eo  \otimes_\V \kappa$, ie $\Car (\E) := \Car (\Eo  \otimes_\V \kappa)$.
\end{definition}

On peut vérifier, voir par exemple \cite[lemme 5.2.6]{berthelotintro}, qu'il s'agit d'un sous-schéma fermé du fibré cotangent $T^*X$ de $X$ indépendant du choix du modèle entier.

\begin{example}~
On suppose la courbe formelle $\X$ affine munie d'une coordonnée étale. On note toujours $\xi_k = \sigma(\partial_k)$ l'image de la dérivation $\partial_k$ dans le gradué $\gr_1 (\Dks)$.
\begin{enumerate}
\item
Puisque le support du faisceau $\Dks$ est $X$ tout entier, on a $\Car (\Dkq) = T^* X$.
\item
La variété caractéristique du module nul $\E = 0$ est vide.
\item
Soit $\E = \Dkq / P$ avec $P \in \Dkq(\X)$ un opérateur différentiel de norme un. D'après \cite[corollaire 4.2.2]{garnier}, $\Eo = \Dk / P$ est sans $\varpi$-torsion puisque $P$ est une base de division\footnote{Voir \cite[partie 4]{garnier} ou \cite[partie 2.4]{hallopeau1} pour la d\'efinition d'une base de division.}  de l'ideal $\Dkq \cdot P$.
Ainsi, $\Eo$ est un modèle entier de $\E$.
On note $d = \Nb(P)$ et $b$ le coefficient d'indice $d$ de $P$. La réduction $\bar{P}$ de $P$ modulo $\varpi$ est un opérateur de $\Dks(X)$ d'ordre $d$, de coefficient dominant $\bar{b} = (b \mod \varpi) \in \O_{X}(X)$.
Alors $\gr (\Eo \otimes_\V \kappa )= \gr (\Dks) / (\sigma(\bar{P}))$, où $\sigma(\bar{P}) = \bar{b} \cdot \xi_k^d $ est le symbole principal de $\bar{P}$. 
La variété caractéristique de $\E$ est donc donnée par l'équation $\Car(\E) = \{ (x, \xi) \in T^*X : \sigma(\bar{P})(x,\xi) = \bar{b}(x) \cdot \xi^d = 0 \}$.
\end{enumerate}
\end{example}

Tout $\Dkq$-module cohérent $\M_k$ vérifie l'inégalité de Bernstein démontrée dans \cite[proposition 3.13]{hallopeau1} : si le module $\M_k$ est non nul, alors les composantes irréductibles de sa variété caractéristique $\Car(\M_k)$ sont toutes de dimension au moins un.

\begin{definition}\label{def2.5}
Un $\Dkq$-module cohérent $\M_k$ est dit holonome si $\dim(\Car(\M_k)) \leq 1$.
\end{definition}

La proposition suivante rassemblant \cite[proposition 3.22, lemme 3.23, corollaire 3.24, proposition 3.28]{hallopeau1} caractérise les $\Dkq$-modules holonomes.

\begin{prop}\label{prop3.22}
Un $\Dkq$-module cohérent $\E$ est holonome si et seulement $\E$ vérifie l'une des assertions équivalentes suivantes :
\begin{enumerate}

\item
$\E$ est localement de la forme $\Dkq / \I$ pour un idéal cohérent $\I \neq 0$  ;
\item
$\E$ est de longueur finie ;
\item
$\E$ est localement de torsion ;
\item
$\Ext_{\Dkq}^d(\E , \Dkq) = 0$ pour tout entier $d \neq 1$ ;
\item
il existe un ouvert non vide $U$ de $\X$ tel que $\E_{|U}$ soit un $\O_{\X , \Q |U}$-module libre de rang fini. Autrement dit, $\E_{|U}$ est une connexion intégrable.
\end{enumerate}
De plus, un $\Dkq$-module coh\'erent est une connexion intégrable si et seulement si sa vari\'et\'e caract\'eristique est vide ou la section nulle de $T^*X$.
\end{prop}

\subsection{$\Di$-modules coadmissibles}\label{section2.2}

Pour tout niveau de congruence $k$ et pour tout ouvert affine $U$ sur lequel on dispose d'une coordonnée locale, on rappelle que $\widehat{\mathcal{D}}^{(0)}_{\X, k + 1 , \Q}(U)$ est une sous-algèbre de $\Dkq(U)$.
On considère les morphismes de transition $\widehat{\mathcal{D}}^{(0)}_{\X, k+1 , \Q} \to \Dkq$ induits par ces inclusions locales.
On définit $\Di$ comme la limite projective sur $k$ des faisceaux $\Dkq$ :
\[ \Di := \varprojlim_k \Dkq . \]
C'est un faisceau de $K$-algèbres sur $\X$ tel que $\Di(U)$ soit une $K$-algèbre de Fréchet-Stein dont la topologie est induite par les normes $| \cdot |_k$ des algèbres de Banach $\Dkq(U)$. De plus,
\[\Di(U)  = \bigcap_{k \geq 0} \Dkq (U) = \displaystyle \left\{ \sum_{n = 0}^\infty a_n \cdot \partial^n :  \,  a_n \in \O_{\X, \Q} (U)~~\text{tq} ~~ \forall \eta>0,~  \lim_{n \to \infty} |a_n| \cdot \eta^n =0 \right\}. \]
Le résultat suivant, démontré dans \cite[lemme 4.2]{hallopeau1}, caractérise les opérateurs différentiels finis de l'algèbre $\Di(U)$ à l'aide des fonctions $\Nb$ pour tout niveau de congruence $k$.
On en déduit les éléments inversibles de l'algèbre $\Di(U)$.

\begin{lemma}\label{lemmedeg}
Soit $P$ un opérateur différentiel de $\Di(U)$. La suite $(\Nb(P))_{k \geq 0}$ est croissante.
De plus, $P$ est un opérateur fini de degré $d \in \N$ si et seulement si la suite $(\Nb(P))_k$ est stationnaire de valeur limite $d$.
\end{lemma}

\begin{cor}
Les éléments inversibles de la $K$-algèbre $\Di(U)$ sont exactement les fonctions inversibles : $\Di(U)^\times = \O_{\X,\Q}(U)^\times$.
\end{cor}
\begin{proof}
Soit $P \in \Di(U)^\times = \bigcap_{k \geq 0} \Dkq(U)^\times$.
Alors $\Nb(P) =0$ pour tout nouveau de congruence $k$ et le coefficient constant de $P$ inversible d'après la proposition \ref{prop2.2.3}.
Le lemme précédent implique que $P$ est un opérateur fini d'ordre 0.
Autrement dit, $P$ est un élément inversible de $\O_{\X,\Q}(U)$.
\end{proof}

On termine cette partie par la définition des $\Di$-modules coadmissibles suivie d'un exemple intéressant motivant la construction du microlocalisé de $\Di$ et la définition de la variété caractéristique pour les $\Di$-modules coadmissibles faite dans la section \ref{section5}.

\begin{definition}
Un $\Di$-module $\M$ est dit coadmissible s'il est isomorphe à une limite projective $\varprojlim_k \M_k$ de $\Dkq$-modules cohérents $\M_k$ telle que les applications de transitions $\M_{k+1} \to \M_k$ soient $\widehat{\mathcal{D}}^{(0)}_{\X, k+1 , \Q}$-linéaires et induisent des isomorphismes de $\Dkq$-modules $ \Dkq \otimes_{\widehat{\mathcal{D}}^{(0)}_{\X, k+1 , \Q}} \M_{k+1} \simeq \M_k$ pour chaque niveau de congruence $k$.
\end{definition}

La catégorie des $\Di$-modules coadmissibles est abélienne et contient les $\Di$-modules cohérents.
La notion de coadmissibilité remplace la cohérence pour les schémas classiques.

\begin{example}\label{ex4.2}
Soit $P = \prod_{n \geq 1} (1 - \varpi^n \partial) \in \Di(U)$.
Le coefficient de $\partial^n$ est de la forme $\varpi^\frac{n(n+1)}{2} \cdot a_n$ avec $a_n$ un élément de $\V$ de valeur absolue un.
Ainsi, le coefficient d'ordre $n$ de $P$ dans $\Dkq(U)$ est $\varpi^{n \left(\frac{n+1}{2}-k \right)} \cdot a_n$.
Par définition, $\Nb(P)$ est le plus grand entier $n$ maximisant la valeur absolue $|\varpi|^{n \left(\frac{n+1}{2}-k \right)}$. On obtient $\Nb(P) = k$.
Dans $\Dkq(U)$, $P$ s'écrit $P = P_k \cdot Q_k$ avec $P_k = \prod_{ 1 \leq n \leq k} (1 - \varpi^n \partial)$ un opérateur fini d'ordre $\Nb(P_k) = k$ et $Q_k = \prod_{n > k} (1 - \varpi^n \partial)$ inversible dans $\Dkq(U)$ (puisque $\Nb(Q_k) = 0$ et son coefficient constant est inversible).
On en déduit que $\Dkq / P \simeq \Dkq / P_k$ est une connexion intégrable de rang $k$ et que $\Di / P \simeq \varprojlim_k \Dkq / P_k$ en tant que $\Di$-module coadmissible.
Par ailleurs, il est possible de retrouver $P$ à partir des opérateurs $P_k$. En effet, la suite $(P_k)_k$ converge vers $P$ dans la $K$-algèbre de Fréchet-Stein $\Di(U)$.
Le $\Di$-module coadmissible $\Di /P$ n'a pas de bonnes propriétés de finitude puisque pour tout niveau de congruence $k$, le $\Dkq$-module cohérent $\Dkq / P_k$ est une connection intégrable de rang $k$.
On souhaite donc définir une notion d'holonomie pour les $\Di$-modules coadmissibles telle que le module $\Di / P$ soit holonome si et seulement $P$ est un opérateur différentiel fini.
Cela motive la construction du microlocalisé $\Fi$ de $\Di$ faite dans les sections \ref{section3} et \ref{section4} ; ce dernier ne rend inversibles que les opérateurs différentiels finis.
\end{example}

\subsection{$\Di$-modules faiblement holonomes}\label{section2.3}

On définit dans cette partie les $\Di$-modules coadmissibles faiblement holonomes au sens de Ardakov-Bode-Wadsley dans l'article \cite{ABW2}.
Pour cela, on introduit dans le lemme suivant les $\Di$-modules \`a droite $\Ext^d_{\Di}(\M , \Di)$ d'un $\Di$-module coadmissible $\M$.

\begin{lemma}\label{lemme4.3.1}
Soit $\M = \varprojlim_k \M_k$ un $\Di$-module coadmissible.
Pour tout niveau de congruence $k \in \N$, on dispose d'un isomorphisme de $\Dkq$-modules à droite :
\[ \Hom_{\widehat{\mathcal{D}}^{(0)}_{\X, k+1, \Q}}(\M_{k+1} , \widehat{\mathcal{D}}^{(0)}_{\X, k+1, \Q}) \otimes_{\widehat{\mathcal{D}}^{(0)}_{\X, k+1, \Q}} \Dkq \simeq \Hom_{\Dkq}(\M_k , \Dkq) . \]
Par ailleurs, $\Ext^d_{\Di}(\M , \Di) := \varprojlim_k \Ext^d_{\Dkq}(\M_k , \Dkq)$ définit pour tout $d \in \N$ un $\Di$-module coadmissible a droite.
De plus, $\Ext^d_{\Di}(\M , \Di) = 0$ d\'es que $d \geq 2$.
\end{lemma}
\begin{proof}
On commence par \'etablir le premier point pour tout niveau de congruence $k$.
Puisque $\M_k$ est un $\Dkq$-module cohérent à gauche, on peut munir le faisceau de groupes abéliens $\Hom_{\Dkq}(\M_k , \Dkq)$ d'une structure de $\Dkq$-module cohérent à droite.
En effet, la question étant locale, on peut supposer $\X$ affine et le module $\M_k$ de présentation finie : $ (\Dkq)^m \to (\Dkq)^n \to \M_k \to 0$.
En appliquant le foncteur $\Hom_{\Dkq}(\bullet , \Dkq)$, on observe que $\Hom_{\Dkq}(\M_k , \Dkq)$ est un sous-module de type fini du $\Dkq$-module cohérent à droite $\Hom_{\Dkq}((\Dkq)^n , \Dkq) \simeq (\Dkq)^n$.
Il en découle que le $\Dkq$-module à droite $\Hom_{\Dkq}(\M_k , \Dkq)$ est cohérent.
Soit $\varphi \in \mathrm{Hom}_{\widehat{\mathcal{D}}^{(0)}_{\X, k+1, \Q}}(\M_{k+1} , \widehat{\mathcal{D}}^{(0)}_{\X, k+1, \Q})$.
On lui associe une application $\Dkq(\X)$-linéaire
\[ \tilde{\varphi} : \M_k \simeq \Dkq \otimes_{\widehat{\mathcal{D}}^{(0)}_{\X, k+1, \Q}} \M_{k+1}  \longrightarrow \Dkq, ~~ P \otimes m  \mapsto P \cdot \varphi(m) . \]
On en déduit une application entre les $\Dkq(\X)$-modules à droite
\begin{align*}
 \mathrm{Hom}_{\widehat{\mathcal{D}}^{(0)}_{\X, k+1, \Q}}(\M_{k+1} , \widehat{\mathcal{D}}^{(0)}_{\X, k+1, \Q}) \otimes_{\widehat{\mathcal{D}}^{(0)}_{\X, k+1, \Q}(\X)} \Dkq(\X) & \to \mathrm{Hom}_{\Dkq}(\M_k , \Dkq) .\\
 \varphi \otimes Q & \mapsto \tilde{\varphi} \cdot Q
\end{align*}
On dispose d'une telle application sur chaque ouvert $U$ de $\X$. On obtient ainsi un morphisme entre $\Dkq$-modules cohérents à droite
\[ \Hom_{\widehat{\mathcal{D}}^{(0)}_{\X, k+1, \Q}}(\M_{k+1} , \widehat{\mathcal{D}}^{(0)}_{\X, k+1, \Q}) \otimes_{\widehat{\mathcal{D}}^{(0)}_{\X, k+1, \Q}} \Dkq \to \Hom_{\Dkq}(\M_k , \Dkq) .\]
On montre maintenant que ce morphisme est un isomorphisme. Il suffit de le vérifier localement.
On suppose donc que $\X$ est affine afin que $\M_{k+1}$ soit de présentation finie, ie $(\widehat{\mathcal{D}}^{(0)}_{\X, k+1, \Q})^{m} \to (\widehat{\mathcal{D}}^{(0)}_{\X, k+1, \Q})^n \to \M_{k+1} \to 0$.
En tensorisant par $\Dkq$, on obtient une présentation $(\Dkq)^m \to (\Dkq)^n \to \M_k \to 0$ de $\M_k$.
En appliquant les foncteurs $\Hom_{\widehat{\mathcal{D}}^{(0)}_{\X, k+1, \Q}}( \bullet , \widehat{\mathcal{D}}^{(0)}_{\X, k+1, \Q})$ et $\Hom_{\Dkq}( \bullet , \Dkq)$, on obtient le diagramme commutatif suivant entre $\Dkq$-modules à droites :
\[ \xymatrix@C=35pt{
0  \ar[d] & 0 \ar[d] & \\
\Hom_{\widehat{\mathcal{D}}^{(0)}_{\X, k+1, \Q}} \left(\M_{k+1} , \widehat{\mathcal{D}}^{(0)}_{\X, k+1, \Q} \right) \otimes \Dkq \ar[r] \ar[d] &  \Hom_{\Dkq} \left( \M_k , \Dkq \right) \ar[d] & \\
\Hom_{\widehat{\mathcal{D}}^{(0)}_{\X, k+1, \Q}} \left( (\widehat{\mathcal{D}}^{(0)}_{\X, k+1, \Q})^{n} , \widehat{\mathcal{D}}^{(0)}_{\X, k+1, \Q} \right) \otimes \Dkq \ar[r]^-{\simeq} \ar[d] &  \Hom_{\Dkq} \left( (\Dkq)^{n} , \Dkq \right) \ar[d] & \\
\Hom_{\widehat{\mathcal{D}}^{(0)}_{\X, k+1, \Q}} \left( (\widehat{\mathcal{D}}^{(0)}_{\X, k+1, \Q})^{m} , \widehat{\mathcal{D}}^{(0)}_{\X, k+1, \Q} \right) \otimes \Dkq \ar[r]^-{\simeq}  &   \Hom_{\Dkq} \left( (\Dkq)^{m} , \Dkq \right) 
} \]
On déduit du fait que les deux dernières flèches horizontales soient des isomorphismes que la première flèche en est aussi un.
On définit alors le $\Di$-module coadmissible à droite
\[ \Ext^0_{\Di}(\M , \Di) := \varprojlim_k \Hom_{\Dkq}(\M_k , \Dkq) . \]
Les faisceaux dérivés $\Ext^d_{\Dkq}(\M_k , \Dkq)$ de $\Hom_{\Dkq}(\M_k , \Dkq)$ admettent une structure de $\Dkq$-module à droite pour laquelle ils sont cohérents.
Pour tout entier $d \geq 2$ et pour tout niveau de congruence $k$, il a été vu dans la preuve de \cite[proposition 3.28]{hallopeau1} que $\Ext^d_{\Dkq}(\M_k , \Dkq) = 0$.
Ainsi, pour $d \geq 2$, on pose simplement
\[ \Ext^d_{\Di}(\M , \Di) := \varprojlim_k \Ext^d_{\Dkq}(\M_k , \Dkq) = 0 . \]
Il reste \`a traiter le cas $d =  1$.
Pour tout niveau de congruence $k \in \N$, on peut démontrer que le morphisme de $\Dkq$-modules à droites
\[ \Ext^1_{\widehat{\mathcal{D}}^{(0)}_{\X, k+1, \Q}}(\M_{k+1} , \widehat{\mathcal{D}}^{(0)}_{\X, k+1, \Q}) \otimes_{\widehat{\mathcal{D}}^{(0)}_{\X, k+1, \Q}} \Dkq \to \Ext^1_{\Dkq}(\M_k , \Dkq) \]
induit par l'application $\Hom_{\widehat{\mathcal{D}}^{(0)}_{\X, k+1, \Q}}(\M_{k+1} , \widehat{\mathcal{D}}^{(0)}_{\X, k+1, \Q}) \to \Hom_{\Dkq}(\M_k , \Dkq)$ est un isomorphisme.
En effet, on peut le vérifier localement où l'on dispose d'une résolution projective de $\M_{k+1}$ dont les premiers termes sont des $\widehat{\mathcal{D}}^{(0)}_{\X, k+1, \Q}$-modules libres de rang fini.
Cela fournit (par tensorisation par $\Dkq$) une résolution de $\M_k$ dont les premiers termes sont des $\Dkq$-modules libres de mêmes rangs.
Le $\Di$-module à droite $\Ext^1_{\Di}(\M , \Di) := \varprojlim_k \Ext^1_{\Dkq}(\M_k , \Dkq)$ est donc coadmissible.
\end{proof}

\begin{definition}\label{def4.2.3}
Un $\Di$-module coadmissible $\M = \varprojlim_k \M_k$ est appelé faiblement holonome si $\Ext^0_{\Di}(\M , \Di) = 0$.
De manière équivalente, un $\Di$-module coadmissible $\M$ est faiblement holonome si et seulement si $\Ext^d_{\Di}(\M , \Di) = 0$ pour tout entier $d \neq 1$.
\end{definition}

La catégorie des modules faiblement holonomes est non triviale. Elle contient par exemple les $\Di$-modules cohérents de la forme $\Di / \I$ pour tout idéal cohérent $\I$ non nul de $\Di$.
Cette définition coincide avec celle de Ardakov-Bode-Wadsley dans \cite{ABW2}.
En effet, ils définissent les modules faiblement holonomes de la manière suivante.
Un $\wideparen{\D}$-module coadmissible $\M$ sur un $K$-espace analytique rigide lisse $X_K$ est faiblement holonome si les groupes de cohomologie $\mathrm{Ext}^d_{\wideparen{\D}}(\M , \wideparen{\D})$ sont nuls pour tout entier $ d \neq \dim X_K$.

\begin{prop}\label{prop4.2.4}
Un $\Di$-module coadmissible $\M = \varprojlim_k \M_k$ est faiblement holonome si et seulement si les $\Dkq$-modules cohérents $\M_k$ sont tous holonomes\footnote{Voir la d\'efinition \ref{def2.5} pour la notion de $\Dkq$-module holonone.}.
\end{prop}
\begin{proof}
Si les $\M_k$ sont tous holonomes, alors $\Ext^0_{\Dkq}(\M_k , \Dkq) = 0$ d'après la proposition \ref{prop3.22}.
Il en découle que $\Ext^0_{\Di}(\M , \Di) = \varprojlim_k \Ext^0_{\Dkq}(\M_k , \Dkq) = 0$.
Réciproquement, on suppose que $\Ext^0_{\Di}(\M , \Di) = 0$. Puisque ce module est coadmissible, on a $\Ext^0_{\Dkq}(\M_k , \Dkq) = \Ext^0_{\Di}(\M , \Di) \otimes_{\Di} \Dkq = 0$.
On en déduit que $\M_k$ est holonome toujours d'après la proposition \ref{prop3.22}.
\end{proof}

\begin{cor}
La catégorie des modules faiblement holonomes est abélienne.
\end{cor}
\begin{proof}
Soit $0 \to \M \to \Nn \to \L \to 0$ une suite exacte de $\Di$-modules coadmissibles avec $\M = \varprojlim_k \M_k$, $\Nn = \varprojlim_k \Nn_k$ et $\L = \varprojlim_k \L_k$.
Pour tout niveau de congruence $k$, on dispose d'une suite exacte de $\Dkq$-modules cohérents $0 \to \M_k \to \Nn_k \to \L_k \to 0$.
D'après \cite[proposition 3.16]{hallopeau1}, le module $\Nn_k$ est holonome si et seulement si $\M_k$ et $\L_k$ le sont.
La proposition \ref{prop4.2.4} implique alors que $\Nn$ est faiblement holonome si et seulement $\M$ et $\L$ sont faiblement holonomes.
\end{proof}

\begin{remark}
Lorsque $\M = \varprojlim_k \M_k$ est faiblement holonome, on dispose localement d'une résolution libre finie des modules $\M_k$ donnée par \cite[proposition 2.20]{hallopeau1}.
En effet, le module $\M_k$ est holonome d'après \ref{prop4.2.4}. Il existe donc un idéal cohérent non nul $\I_k$ de $\Dkq$ tel que $\M_k \simeq \Dkq / \I_k$.
Soit $x \in \X$. Une base de division\footnote{Voir \cite[partie 2.4]{hallopeau1} pour la d\'efinition d'une base de division.} relativement au point $x$ fournit une présentation explicite de $\I_k$ sur un voisinage $U$ de $x$ :
\[ \xymatrix{ 0 \ar[r] & (\widehat{\mathcal{D}}^{(0)}_{U, k , \Q})^{r-1} \ar[r]^{\cdot R} & (\widehat{\mathcal{D}}^{(0)}_{U, k , \Q})^r \ar[r] & \I_{k|U} \ar[r] & 0 } . \]
La matrice $R$ s'exprime entièrement en fonction de la base de division.
Combinée avec la suite exacte courte $ 0 \to \I \to \Dkq \to \M_k \to 0$, on obtient la présentation suivante du module $\M_k$ sur l'ouvert $U$ :
\[ \xymatrix{ 0 \ar[r] & (\widehat{\mathcal{D}}^{(0)}_{U, k , \Q})^{r-1} \ar[r]^{\cdot R} & (\widehat{\mathcal{D}}^{(0)}_{U, k , \Q})^r \ar[r] & \widehat{\mathcal{D}}^{(0)}_{U, k , \Q} \ar[r] & \M_{k |U} \ar[r] & 0 } . \]
On retrouve bien que $\Ext^d_{\Dkq}(\M_k , \Dkq) = 0$ pour tout entier $d \geq 2$.
Par ailleurs, cette résolution permet de calculer explicitement le groupe $\Ext^1_{\Dkq}(\M_k , \Dkq)$ au voisinage de $x$ en fonction d'une base de division relativement au point $x$.
\end{remark}

Soit $\M  = \varprojlim_k \M_k$ un $\Di$-module coadmissible et $d$ un entier naturel.
Les $\Dkq$-modules cohérents à gauche $\Ext^d_{\Dkq}(\M_k , \Dkq) \otimes_{\O_{\X , \Q}} \omega_{\X , \Q}^{-1}$ forment un $\Di$-module coadmissible  à gauche.
En effet, cela résulte de l'article \cite[lemme 2.2, partie I]{virrion} de Virrion :
\begin{align*}
\Dkq \otimes_{\widehat{\mathcal{D}}^{(0)}_{\X, k+1, \Q}} & \left( \Ext^1_{\widehat{\mathcal{D}}^{(0)}_{\X, k+1, \Q}}(\M_{k+1} , \widehat{\mathcal{D}}^{(0)}_{\X, k+1, \Q}) \otimes_{\O_{\X , \Q}} \omega_{\X , \Q}^{-1} \right)  \\
& \simeq \Ext^1_{\widehat{\mathcal{D}}^{(0)}_{\X, k+1, \Q}}(\M_{k+1} , \widehat{\mathcal{D}}^{(0)}_{\X, k+1, \Q}) \otimes_{ \widehat{\mathcal{D}}^{(0)}_{\X, k+1, \Q}} \left( \Dkq  \otimes_{\O_{\X , \Q}} \omega_{\X , \Q}^{-1} \right) \\
& \simeq \left( \Ext^1_{\widehat{\mathcal{D}}^{(0)}_{\X, k+1, \Q}}(\M_{k+1} , \widehat{\mathcal{D}}^{(0)}_{\X, k+1, \Q}) \otimes_{ \widehat{\mathcal{D}}^{(0)}_{\X, k+1, \Q}} \Dkq \right) \otimes_{\O_{\X , \Q}} \omega_{\X , \Q}^{-1}  \\
& \simeq \Ext^1_{\Dkq}(\M_k , \Dkq) \otimes_{\O_{\X , \Q}} \omega_{\X , \Q}^{-1} .
\end{align*}
Le dual du module cohérent $\M_k$ est défini par $\M_k^\vee := \Ext^1_{\Dkq}(\M_k , \Dkq) \otimes_{\O_{\X , \Q}} \omega_{\X , \Q}^{-1}$.
Il a été vu dans \cite[partie 3.5]{hallopeau1} que si le module $\M_k$ est holonome, alors son dual $\M_k^\vee$ reste holonome et nous disposons d'un isomorphisme naturel de bidualité $(\M_k^\vee)^\vee \simeq \M_k$.

\begin{definition}
Le dual d'un $\Di$-module coadmissible $\M  = \varprojlim_k \M_k$ est le module coadmissible $\M^\vee := \varprojlim_k \M_k^\vee = \varprojlim_k \left( \Ext^1_{\Dkq}(\M_k , \Dkq) \otimes_{\O_{\X , \Q}} \omega_{\X , \Q}^{-1} \right)$.
\end{definition}

\begin{prop}
Soit $\M  = \varprojlim_k \M_k$ un $\Di$-module coadmissible faiblement holonome.
Alors son dual $\M^\vee$ est aussi faiblement holonome.
De plus, on dispose d'un isomorphisme de bidualité : $(\M^\vee)^\vee \simeq \M$.
\end{prop}
\begin{proof}
D'après la proposition \ref{prop4.2.4}, les $\Dkq$-modules cohérents $\M_k$ sont tous holonomes.
Alors $\M_k^\vee = \Ext^1_{\Dkq}(\M_k , \Dkq) \otimes_{\O_{\X , \Q}} \omega_{\X , \Q}^{-1}$ est holonome d'après la proposition \ref{prop3.22}.
Ainsi, le $\Di$-module coadmissible $\M^\vee = \varprojlim_k \M_k^\vee$ est faiblement holonome toujours d'après la proposition \ref{prop4.2.4}.
Enfin, les isomorphismes naturels de bidualité $(\M_k^\vee)^\vee \simeq \M_k$ pour chaque niveau de congruence induisent un isormorphisme $(\M^\vee)^\vee \simeq \M$ en prenant la limite projective.
\end{proof}

\section{Microlocalisations de $\Dkq$}\label{section3}

On introduit dans cette section des microlocalisés $\F_{k , r}$ du faisceau $\Dkq$ en utilisant une localisation non commutative de l'annexe de l'article \cite{zab} de Zábrádi.
Pour cela, il faut partir d'une algèbre quasi-abélienne, ce qui est le cas des algèbres de Banach $(\Dkq(U) , | \cdot |_k)$ dès que $k \geq 1$.
On localise ensuite la partie multiplicative des puissances de la dérivation $\partial$ dans $\Dkq(U)$.
Une difficulté consiste alors a trouver des microlocalisés des algèbres $\Dkq(U)$ admettant des morphismes de transition commutant avec les morphismes injectifs $\widehat{\mathcal{D}}^{(0)}_{\X , k+1 , \Q}(U) \hookrightarrow \Dkq(U)$.
Intervient ainsi un paramètre $r \in \{ 1 , \dots , k \}$ permettant d'obtenir des microlocalisés $\F_{k , r}$ poss\'edant des morphismes de transition $\F_{k+1 , r} \to \F_{k , r}$ \'etendant $\widehat{\mathcal{D}}^{(0)}_{\X , k+1 , \Q} \to \Dkq$.
En passant à la limite projective sur $k$, on obtiendra dans la section \ref{section4} des faisceaux $\Fir = \varprojlim_{k \geq r} \F_{k , r}$ microlocalisés de $\Di$. Localement, $\Fir(U)$ sera une $K$-algèbre de Fréchet-Stein et l'union des alg\`ebres $\Fir(U)$ pour $r \geq 1$ aura la condition d'inversibilit\'e souhait\'ee discut\'ee en introduction.

\subsection{Microlocalisation d'une algèbre de Banach quasi-abélienne}\label{section3.1}

Cette partie résume les principales propriétés du microlocalisé d'une algèbre de Banach quasi-abélienne construit dans l'annexe écrite par Peter Schneider de l'article \cite{zab}.
On démontre ensuite que l'algèbre $(\Dkq(U) , | \cdot |_k)$ est quasi-abélienne dès que $U$ est un ouvert affine sur lequel on dispose d'une coordonnée étale et dès que $k \geq 1$.
On se donne une $K$-algèbre de Banach $A$ munie d'une norme multiplicative $| \cdot |$ non archimédienne.

\begin{definition}
On dit que l'algèbre $A$ est quasi-abélienne pour la norme $|\cdot|$ s'il existe un réel $\gamma \in [0,1[$ tel que pour tout couple $(a, b) \in A^2$, $| ab -ba | = | [a,b] | \leq \gamma \cdot |ab| = \gamma \cdot  |a| |b|$.
\end{definition}

On fixe $m$ normes $| \cdot |_1 , \dots,  | \cdot |_m$ quasi-abéliennes sur $A$.
Soit $S$ une partie multiplicative de $A$. Il existe alors une unique $K$-algèbre de Banach $B = A \langle | \cdot |_1 , \dots,  | \cdot |_m, S \rangle $ munie d'une norme sous-multiplicative $ \| \cdot \|$ non archimédienne et d'un morphisme isométrique de $K$-algèbres $\phi : (A , \max\{| \cdot |_1 , \dots,  | \cdot |_m \}) \to (B , \| \cdot \|)$
tels que $\phi(S) \subset B^\times$ et tels que les éléments de B de la forme $s^{-1} a$ pour $(s,a) \in S \times A$ soient denses dans $B$.
On identifie $A$ à une sous-algèbre de $B$ via $\phi$: on note encore $a$ l'image d'un élément $a$ de $A$ et $s^{-1}$ l'inverse de $s\in S$ dans $B$.
Par ailleurs, le morphisme $\phi : A \to B$ vérifie la propriété universelle suivante.

\begin{prop}\cite[proposition A.18]{zab}\label{prop1.2}
Soit $(D , \|\cdot \|_D)$ une $K$-algèbre de Banach munie d'un morphisme de $K$-algèbres $f : A \to D$ tel que $f(S) \subset D^\times$. On suppose qu'il existe une constante $c > 0$ pour laquelle
\[ \forall (s,a) \in S \times A, ~~ \| f(s)^{-1} \cdot f(a) \|_D \leq c \cdot \max_{1 \leq i \leq m} \left\{ |s|_i^{-1} \cdot |a|_i \right\} . \]
Il existe alors un unique morphisme continu de $K$-algèbres $\tilde{f} : B \to D$ tel que $f = \tilde{f} \circ \phi$.
De plus, si $c = 1$ et si la norme $\| \cdot \|_D$ de $D$ est sous-multiplicative, alors l'application $\tilde{f}$ est 1-lipschitzienne : $\forall b \in B, ~~ \| \tilde{f} (b) \|_D \leq \| b\|$.
\end{prop}

On se donne une constante $\gamma \in [0 , 1[$ telle que pour tous $a, b \in A$ et pour tout $i \in \{ 1 ,\dots , m\}$, $| ab -ba |_i  \leq \gamma \cdot |ab|_i$.

\begin{prop}\cite[proposition A.21]{zab}\label{prop1.4}
Pour tout $(e_1 , \dots , e_n) \in B^n$ et pour toute permutation $\sigma \in S_n$, on a $\| e_1\dots e_n - e_{\sigma(1)} \dots e_{\sigma(n)} \| \leq \gamma \| e_1 \dots e_n \|$.
En particulier, l'algèbre $B$ est quasi-abélienne de constante $\gamma$ et $\|ab\| = \|ba\|$ pour tout $a , b \in B$.
Par ailleurs, dans le cas ou $B$ est le localisé de $A$ pour une unique norme, la norme de $B$ est multiplicative.
\end{prop}

Le lemme suivant correspondant \`a \cite[lemme A.7]{zab} relie la norme des éléments de $B$ de la forme $s^{-1} a$ avec les normes $|\cdot |_i$ de $s$ et de $a$.

\begin{lemma}\label{lemme1.5}
Pour $a \in A$ et $s \in S$, on a $ \displaystyle \| s^{-1} a \| = \max_{ 1 \leq i \leq m} \{ |s|_i^{-1} \cdot |a|_i \} $.
\end{lemma}

On désigne toujours par $U$ un ouvert affine sur lequel on dispose d'une coordonnée locale.
On rappelle que la norme $| \cdot |_k$ et la fonction $\Nb$ ne dépendent pas du choix de la coordonnée locale sur $U$.

\begin{prop}\label{propqa}
Pour tout entier $k\geq 1$, l'algèbre de Banach $(\Dkq(U) , |\cdot|_k)$ est quasi-abélienne de constante optimale $|\varpi|^k = p^{-k}$.
\end{prop}
\begin{proof}
On doit démontrer que pour tous les opérateurs différentiels $P$ et $Q$ de $\Dkq(U)$, $|PQ -QP|_k \leq \frac{1}{p^k} \cdot |P|_k \cdot |Q|_k$.
On peut se ramener au cas o\`u $P = a \cdot (\varpi^k \partial)^n$.
En effet, si $P = \sum_{n \in \N} P_n$ avec $|P|_k = \max_n \{ |P_n| \}$, alors $[ P , Q] = \sum_n [ P_n , Q]$ et $| [ P , Q] |_k \leq \max \{ |[ P_n , Q]|_k \}$.
On peut aussi supposer que $Q = b \cdot (\varpi^k \partial)^m$.
On a
\[ a\cdot (\varpi^k\partial)^n \cdot b \cdot (\varpi^k\partial)^m= ab\cdot (\varpi^k \cdot \partial)^{n+m} + \varpi^k \cdot \underbrace{\sum_{\ell = 0}^n \varpi^{k (n- \ell - 1)} {n\choose \ell} a \partial^{n - \ell}(b) \cdot (\varpi^k\partial)^{\ell+m}}_\alpha \]
avec $|\alpha|_k \leq |a|_k \cdot |b|_k$.
De m\^eme, $ b \cdot (\varpi^k\partial)^m \cdot a\cdot (\varpi^k\partial)^n = ab\cdot (\varpi^k \cdot \partial)^{n+m} + \varpi^k \cdot \beta$ avec $|\beta|_k \leq |a|_k \cdot |b|_k$.
On a donc $[P , Q] = \varpi^k\cdot (\alpha - \beta)$ et $|[P , Q]|_k \leq |\varpi^k| \cdot |a|_k \cdot |b|_k$.
\end{proof}

\begin{remark}
L'algèbre $\widehat{\mathcal{D}}^{(0)}_{\X, \Q}(U) = \widehat{\mathcal{D}}^{(0)}_{\X, 0, \Q}(U)$ n'est pas quasi-abélienne puisque par exemple $[\partial , t] = 1$.
Pour $k>0$, on a $[\varpi^k \partial , t] = \varpi^k$ avec $|\varpi^k\partial |_k = 1 = |t|$ et $|\varpi^k| < 1$.
\end{remark}

\subsection{Un premier microlocalisé $\E_k$}\label{section3.2}

On suppose à partir de maintenant, et dans la suite de cet article, que le niveau de congruence $k$ est supérieur ou égal à un.
Ainsi, la $K$-algèbre $(\Dku , | \cdot |_k)$ est toujours quasi-abélienne. On peut donc la localiser pour n'importe quelle partie multiplicative $S$.
On choisit concrètement les puissance de la dérivation : $S := \{ \partial^n, ~~ n\in\N\}$.
Il revient au même de prendre comme partie multiplicative les puissances de $\varpi^k\partial$ puisque $\varpi$ est inversible dans $\Dku$.
Dans cette partie, on localise l'algèbre $\Dku$ pour la seule norme $| \cdot |_k$.
Le microlocalisé $\Ek$ ainsi obtenu possède de très bonnes propriétés.
En particulier, on peut déterminer facilement ses éléments inversibles.
Cependant, il n'existe pas de morphismes de transition pertinents $\E_{k+1}(U) \to \E_k(U)$ commutant avec $\widehat{\mathcal{D}}^{(0)}_{\X, k+1}(U) \hookrightarrow \Dk(U)$.
On ne peut donc pas obtenir à partir des localisés $\E_k(U)$ un microlocalisé au niveau de l'algèbre $\Di(U) = \varprojlim_k \Dkq(U)$.
On introduira dans la partie suivante les microlocalisés plus compliqués $\Fkr$ (coïncidant avec $\Ek$ lorsque $k =r$) admettant des morphismes de transition.
Néanmoins, on obtient facilement une injection naturelle $\Fkr \hookrightarrow \Ek$ continue.
Ainsi, bien connaitre les algèbres $\Ek$ permet de déduire des propriétés des microlocalisés $\Fkr$.

\begin{definition}
L'algèbre $\Ek := \Dku \langle |\cdot|_k , S \rangle$ est définie comme le microlocalisé de $\Dku$ donné par la partie multiplicative $S$ et la seule norme $| \cdot |_k$.
On note $\phi_k : \Dku \hookrightarrow \Ek$ le morphisme de $K$-algèbres associé.
\end{definition}

Le microlocalisé $\Ek$ est une $K$-algèbre de Banach et le morphisme $\phi_k : \Dku \hookrightarrow \Ek$ est une isométrie. De plus, puisqu'on a localisé par rapport à la seule norme $| \cdot |_k$ de $\Dkq(U)$, la norme de $\Ek$ est multiplicative.
On identifie donc $\Dku$ à une sous-algèbre de $\Ek$ via le morphisme $\phi_k$ et l'on peut noter encore $| \cdot |_k$ la norme de $\Ek$ sans ambiguïté.
On a $|\partial|_k = |\varpi|^{-k} = p^{-k}$. Le lemme \ref{lemme1.5} nous dit que $|\partial^{-1}|_k = |\varpi|^k = p^k$.
De manière équivalente, $|\varpi^k\partial|_k = |(\varpi^k\partial)^{-1}|_k =1$.

\begin{prop}\label{prop2.7}
Tout élément $S$ de $\Ek$ s'écrit uniquement sous la forme d'une série $S = \sum_{n \in \Z} a_n \cdot (\varpi^k\partial)^n$ avec $a_n \in \O_{\X , \Q}(U)$ tels que $a_n \to 0$ lorsque $n \to \pm \infty$. De plus, $|S|_k = \max_{n\in\Z} \{ |a_n| \}$.
\end{prop}
\begin{proof}
Commençons par démontrer que l'algèbre $\Ek$ contient les séries de cette forme. Soit $(a_n)_{n \in \Z}$ une suite d'éléments de $\O_{\X , \Q}(U)$ vérifiant $a_n \to 0$ lorsque $n \to \pm \infty$.
On pose $S_m = \sum_{n\geq -m}  a_n \cdot (\varpi^k\partial)^n$ pour $m\in \N$.
Puisque l'algèbre $\Ek$ contient $\Dku$ et les puissances de $\partial^{-1}$, les séries $S_m$ appartiennent au localisé $\Ek$.
Pour tout entier $\ell \geq 0$, on a $S_{m+\ell} - S_m = \sum_{-(m+\ell) \leq n < -m} a_n \cdot (\varpi^k\partial)^n$ et

\[ |S_{m+\ell} - S_m|_k \leq  \max_{-(m+\ell) \leq n < -m} \left\{ |a_n \cdot (\varpi^k\partial)^n|_k \right\} = \max_{-(m+\ell) \leq n < -m} \left\{ |a_n| \right\} .\]
Puisque la suite $a_n$ converge vers 0 lorsque $n \to -\infty$, la suite $(S_m)_m$ est de Cauchy dans l'algèbre de Banach $\Ek$.
Elle converge donc vers un élément de $\Ek$ que l'on note $S = \sum_{n \in \Z} a_n \cdot (\varpi^k\partial)^n$.
Par ailleurs, on a $|S_m|_k = \max_{n\geq -m} |a_n|$. En effet, $S_m \cdot (\varpi^k \partial)^m = \sum_{n=0}^\infty a_{n-m} (\varpi^k\partial)^n$.
Puisque la norme $|\cdot|_k$ est multiplicative et comme $|\varpi^k\partial|_k =1$, on obtient $|S_m|_k = |S_m \cdot (\varpi^k \partial)^m |_k = \max_{n\in \N} |a_{n-m}| = \max_{n\geq -m} |a_n|$.
On en déduit en passant à la limite $m \to \infty$ que $|S|_k = \max_{n\in\Z} |a_n|$.
En particulier, $S = 0$ si et seulement si tous ses coefficients $a_n$ sont nuls. Ainsi, une telle écriture est unique.

On démontre maintenant que $\Ek$ coïncide avec l'ensemble de ces séries. On note $F$ le sous-ensemble de $\Ek$ constitué des éléments de la forme $S = \sum_{n \in \Z} a_n \cdot (\varpi^k\partial)^n$ avec $a_n \to 0$ pour $n \to \pm \infty$.
On va vérifier que $F$ contient les éléments $(\varpi^k \partial)^n \cdot P$ pour $n\in \Z$ et $P \in \Dku$ et que $F$ est fermé. Puisque l'ensemble des termes de la forme $(\varpi^k \partial)^n \cdot P$ est dense dans $\Ek$, on aura bien $\Ek = F$.
Pour le premier point, il suffit de prouver que si $f$ est un élément de $\O_{\X,\Q}(U)$, alors $\partial^{-1} f \in F$. On a $[\partial , f ] = \partial f - f \partial  = \partial(f)$, soit $\partial f = \partial(f) + f \partial$.
Il vient $ \partial^{-1} \partial f = f = \partial^{-1} \partial(f) + \partial^{-1} f \partial$.
Ainsi, $\partial^{-1} f = f \partial^{-1} - \partial^{-1} \partial(f) \partial^{-1}$. En réitérant ce calcule pour $\partial^{-1} \partial(f)$ et en modifiant ce terme dans l'expression de $\partial^{-1} f$, on obtient $\partial^{-1} \partial^n(f) = \partial^n(f) \partial^{-1} - \partial^{-1} \partial^{n+1}(f) \partial^{-1}$.
Une récurrence montre alors que pour tout entier $n \in \N^*$,
\[ \partial^{-1} f = \sum_{j=0}^{n-1} (-1)^j \partial^j(f) \partial^{-(j+1)} + (-1)^n \partial^{-1} \partial^n(f) \partial^{-n} .\]
On a
\[ | \partial^j(f) \partial^{-(j+1)} |_k = |\partial^j(f)| \cdot |\partial^{-(j+1)}|_k \leq |f| \cdot p^{-k(j+1)} \underset{j\to \infty}{\longrightarrow} 0 ,\]
\[ |\partial^{-1} \partial^n(f) \partial^{-n}|_k = |\partial^{-1}|_k \cdot |\partial^n(f)| \cdot |\partial^{-n}|_k \leq |f| \cdot p^{-k(n+1)} \underset{n\to \infty}{\longrightarrow} 0 .\]
La somme partielle définissant $\partial^{-1} f$ converge donc lorsque $n \to \infty$. On en déduit que
\[ \partial^{-1} f = \sum_{n=0}^{+\infty} (-1)^n \partial^n(f) \cdot \partial^{-(n+1)} \in F .\]
Enfin, puisque la norme d'un élément $S = \sum_{n \in \Z} a_n \cdot (\varpi^k\partial)^n$ de $F$ est $|S|_k = \max_{n\in \Z} |a_n|$, toute suite de Cauchy d'éléments $S_m = \sum_{n \in \Z} a_{n,m} \cdot (\varpi^k\partial)^n$ de $F$ converge aussi vers un élément de $F$.
En effet, le fait que la suite $(S_m)$ soit de Cauchy implique que pour tout entier $n\in\Z$, la suite $(a_{n,m})_m$ est de Cauchy. Elle converge donc vers un élément $\alpha_n \in \O_{\X , \Q}(U)$.
La suite $(S_m)_m$ converge alors vers $S =  \sum_{n \in \Z} \alpha_n \cdot (\varpi^k\partial)^n$. Cela prouve que $F$ est fermé.
Il en résulte que $\Ek =F$.
\end{proof}

\begin{remark}
La preuve montre que pour toute fonction $f \in \O_{\X,\Q}(U)$,
\[ \partial^{-1} f = \sum_{n=0}^{+\infty} (-1)^n \partial^n(f) \cdot \partial^{-(n+1)} \in \Ek . \]
\end{remark}

\begin{prop}\label{propnoeth}
La $K$-algèbre $\Ek$ est noetherienne à gauche et à droite.
De plus, le morphisme d'algèbres $\phi_k : \Dku \hookrightarrow \Ek$ est plat à gauche et à droite.
\end{prop}
\begin{proof}
Un point donn\'e $x$ de $\X$ est $\kappa'$-rationnel pour une extension finie $\kappa'$ de $\kappa$ de degré $s \in \N^*$.
La norme spectrale de l'algèbre $O_{\X , \Q}(U)$ relative au point $x$ est à valeur dans $p^{\frac{1}{s} \cdot\Z}$ (corollaire 1.5.4 de l'article \cite{garnier} de Garnier).
Les valeurs prises par cette derni\`ere étant indexées par $\Z$, nous pouvons supposer que $x$ est $\kappa$-rationnel et que $|\O_{\X , \Q}(U)| = p^\Z$.
On munit l'algèbre $\O_{\X , \Q}(U)$ de la filtration croissante exhaustive fournie par la norme spectrale : $O_{\X , Q}(U)(m) = \left\{ f \in \O_{\X , \Q}(U) : |f| \leq p^m \right\} $ pour $m \in \Z$.
On note $\gr(\O_{\X , \Q}(U))$ le gradué associé.
Pour $f$ une section de $\O_{\X , \Q}(U)$ telle que $|f| = p^m$, soit $\gamma(f) = \gamma_m(f)$ l'image de $f$ dans le gradué $\gr_m(\O_{\X , Q}(U)) := \O_{\X , Q}(U)(m) / \O_{\X , Q}(U)(m-1)$.
On note $\gamma : \O_{\X , \Q}(U) \to \gr(\O_{\X , \Q}(U))$ l'application induite et $\gamma : K \to \gr K$ la restriction à $K$.
On dote ainsi le gradué $\gr \O_{\X , Q}(U)$ d'une structure de $\gr K$-algèbre.
Puisque l'algèbre $O_{\X , \Q}(U)$ est noetherienne, le gradu\'e $\gr(\O_{\X , \Q}(U))$ est noetherien.
Soit $P = \sum_{n \in \Z} a_n \cdot ( \varpi^k \partial)^n$ un élément non nul de $\Ek$. Sa norme $|P|_k = \max_{n \in \Z} \{ |a_n| \}$ appartient à $p^\Z$.
On munit l'algèbre $\Ek$ de la filtration donnée par la norme : $\Ek(m) = \{ P \in \E_k(U) : |P|_k \leq p^m \}$ pour $ m \in \Z$.
C'est une filtration croissante indexée par $\Z$ vérifiant
\[ \forall m , m' \in \Z, ~~ \Ek(m) \cdot \Ek(m') \subset \Ek(m + m') . \]
On note $\gr_m \Ek := \Ek(m) / \Ek(m-1)$ et $\gr \Ek := \bigoplus_{m\in \Z} \gr_m \Ek$ le gradué associé.
Soit $\zeta_k$ l'image de $\varpi^k \partial$ dans $\gr_0 \Ek$. Pour tout $\lambda \in \Z$, $\Ek(m)$ est isomorphe à $\Ek(m + \lambda)$ via la multiplication par $\varpi^\lambda$.
On en déduit que $gr_m \Ek \simeq \gr_{m+ \lambda} \Ek$.
On peut donc identifier $\gamma(\varpi^\lambda) \cdot \zeta_k$ avec l'image de $\varpi^\lambda \cdot (\varpi^k \partial)$ dans le gradué $\gr_\lambda \Ek$.
On a $|P|_k = p^m$ pour un certain entier $m \in \Z$. L'image de $P$ dans le gradué $\gr_m \Ek$ est donc l'élément
\[ \gamma(P) = \sum_{ \nb(P) \leq n \leq \Nb(P)} \gamma_m(a_n) \cdot \zeta_k^n\]
où $\nb(P)$ le plus petit entier naturel $\ell$ pour lequel $|P|_k = |a_\ell |$.
On définit ainsi une application $\gamma : \Ek \to \gr \Ek$ dont l'image engendr\'e le gradué en tant que $\gr K$-algèbre.
Puisque l'algèbre $\Ek$ est quasi-abélienne, le gradué est commutatif.
On en déduit que
\[ \gr \Ek \simeq \gr(\O_{\X , \Q}(U)) [ \zeta_k , \zeta_k^{-1}] \]
en tant que $\gr K$-algèbres. En particulier, $\gr \Ek$ est une $\gr K$-algèbre noetherienne.
La proposition 1.1 de l'article \cite{schneider} implique alors que l'algèbre $\Ek$ est noetherienne à gauche et à droite.
De même, on munit l'algèbre $\Dku$ de la filtration donnée par la norme. L'inclusion $\Dku \hookrightarrow \Ek$ induit une inclusion entre les gradués $\gr \Dku \hookrightarrow \gr \Ek$.
On a $\gr \Dku \simeq \gr(\O_{\X , \Q}(U)) [ \zeta_k ]$ en tant que $\gr K$-algèbres. Le morphisme de $\gr K$-algèbres $\gr \Dku \hookrightarrow \gr \Ek$ est donc plat.
D'après l'article \cite[proposition 1.2]{schneider}, on obtient que le morphisme $\Dku \hookrightarrow \Ek$ est plat à gauche et à droite.
\end{proof}

\subsubsection*{Elements inversibles de $\Ek$}

On rappelle que l'algèbre $\Dk(U)$ est l'ensemble des \'el\'ements de $\Dku$ de norme $| \cdot |_k$ inférieure ou égale à un et que le faisceau de $\kappa$-algèbres $\Dks$ sur la fibre spéciale $X$ de $\X$ est la réduction modulo $\varpi$ de $ \Dk$.
Ce dernier est engendré localement sur $U$ par ${\O_X}_{|U}$ et par la dérivation $\partial_k$ image de $\varpi^k \partial$ après réduction modulo $\varpi$.
De plus, $\Dks(U)$ est une $\kappa$-algèbre commutative puisque $[\varpi^k \partial , t ] = \varpi^k \in \Dk(U)$ avec $k \geq 1$.
Ainsi, $\Dks(U)= \O_X(U)[ \partial_k]$ est un anneau de polynôme en la variable $\partial_k$.
En particulier, $\Dks(U)^\times = \O_X(U)^\times$.
Soit $\E_k^\circ(U)$ l'ensemble des éléments de $\Ek$ de norme $| \cdot |_k$ inférieure on égale à un. C'est une $\V$-algèbre de Banach contenant $\Dk(U)$.
On note $E_k(U) = \E_k^\circ(U)  \otimes_\V \kappa $ la réduction modulo $\varpi$ de $\E_k^\circ(U)$.
Il s'agit d'une $\kappa$-algèbre contenant $\Dks(U)$ dans laquelle la dérivation $\partial_k$ est inversible.
Puisque l'algèbre de Banach $(\Ek , | \cdot |_k)$ est quasi-abélienne, $E_k(U)$ est une $\kappa$-algèbre commutative. Il en découle que $E_k(U) = \O_X(U)[ \partial_k , \partial_k^{-1}]$.
Autrement dit, $E_k(U)$ est le localisé classique de l'algèbre de polynômes $\O_X(U)[ \partial_k]$ pour la partie multiplicative des puissances de la dérivation $\partial_k$.
On en déduit l'égalité
\[ E_k(U)^\times = \left\{ a \cdot \partial_k^n, ~~ a \in  \O_X(U)^\times ~ \mathrm{et} ~ n \in \Z \right\} . \]
Pour tout $P = \sum_{n \in \Z} a_n \cdot (\omega^k \partial)^n \in \Ek$, on pose $\Nb(P) := \max \{ n \in \Z : |a_n| = |P|_k \}$ et $\nb(P) := \min \{ n \in \Z : |a_n| = |P|_k \}$.
Ces d\'efinitions généralisent celles de $\Dku$ donn\'ees pr\'ec\'edemment.

\begin{prop}\label{prop2.11}
Un élément $P$ de $\Ek$ est inversible si et seulement si $\Nb(P) = \nb(P)$ et si son coefficient d'indice $\Nb(P)$ est inversible dans $\O_{\X , \Q}(U)$.
\end{prop}
\begin{proof}
Quitte à normaliser $P$, on peut supposer que $|P|_k = 1$. Alors $P \in \E_k^\circ(U)$. On note $\bar{P} = (P \mod \varpi)$ l'image de $P$ dans $E_k(U) = \E_k^\circ(U) \otimes_\V \kappa$.
On commence par démontrer que $P$ est inversible dans $\Ek$ si et seulement si $\bar{P}$ est inversible dans $E_k(U)$. Le sens direct est évident.
On suppose que $\bar{P}$ est inversible dans $E_k(U)$ : il existe $Q \in \E_k^\circ(U)$ tel que $\bar{P} \bar{Q} = 1$.
Ainsi, $P Q$ est de la forme $1 + \varpi R$ pour un certain élément $R \in \E_k^\circ(U)$.
Puisque $| \varpi R |_k \leq |\varpi| <1$, l'élément $1 + \varpi R$ est inversible dans $\E_k(U)$. En effet, son inverse est donné par la série classique $(1+ \varpi R)^{-1} = \sum_{i = 0}^\infty (-\varpi R)^i$.
Cette somme est convergente dans l'algèbre de Banach $\Ek$ puisque $| \varpi R |_k <1$.
On en déduit que $P \cdot Q (1 + \varpi R)^{-1} = 1$. Autrement dit, $P$ est inversible à droite dans $\E_k(U)$.
De même, $P$ est inversible à gauche puisque $\bar{Q} \bar{P} = 1$. Ainsi, $P$ est inversible dans l'algèbre $\E_k(U)$.
On rappelle que  $E_k(U)^\times = \left\{ a \cdot \partial_k^n, ~~ a \in  \O_X(U)^\times ~ \mathrm{et} ~ n \in \Z \right\}$.
L'élément $P$ est donc inversible dans le microlocalis\'e $\Ek$ si et seulement si $\bar{P}$ est de la forme $\bar{a} \cdot \partial_k^n$ avec $\bar{a} \in \O_X(U)^\times$.
La réduction modulo $\varpi$ de $P$ est un monôme si et seulement si $\Nb(P) = \nb(P)$.
Enfin, le terme $\bar{a}$ est inversible dans $\O_X(U)$ si et seulement si $a$ est inversible dans $\O_\X(U)$.
Il s'agit exactement des conditions données dans l'énoncé de la proposition.
\end{proof}

\begin{remark}
Ce critère est analogue à la condition d'inversibilité des séries de Laurent convergentes $K\langle T \rangle $ à coefficients dans un corps ultramétrique complet $K$.
Par ailleurs, il est possible de démontrer ce résultat sans réduire modulo $\varpi$ les faisceaux $\Dk(U)$ et $\E_k^\circ(U)$.
En effet, on peut raisonner de manière analogue en utilisant le gradué $\gr \Ek \simeq \gr(\O_{\X , \Q}(U)) [ \zeta_k , \zeta_k^{-1}]$ de $\E_k(U)$ pour la filtration induite par la norme $| \cdot |_k$.
\end{remark}

Soit $P = \sum_{n \in \Z} a_n \cdot (\varpi^k \partial)^n$ un élément inversible de $\E_k(U)$.
On peut exprimer explicitement son inverse en fonction des coefficients $a_n$.
On note $ d = \Nb(P)$ l'ordre de $P$. Puisque $P$ est inversible, on sait que $\Nb(P) = \nb(P)$ et que le coefficient $a_d$ de $P$ est inversible d'après la proposition \ref{prop2.11}. On a
\[ P = a_d \cdot  \underbrace{\left( \sum_{n \in \Z} \frac{a_{n+d}}{a_d} \cdot (\varpi^k \partial)^n \right) }_{Q} \cdot (\varpi^k \partial)^d .\]
Le coefficient constant de $Q$ est égal à 1. Ses autres coefficients sont de normes strictement inférieures à 1. Son inverse est comme précédemment donné par $Q^{-1} = \sum_{i \geq 0} (-(Q-1))^i$.
On obtient
\[ P^{-1} = (\varpi^k\partial)^{-d} \cdot  \sum_{i \in \N} \left(-\sum_{n \in \Z\backslash \{0\}} \frac{a_{n+d}}{a_d} \cdot (\varpi^k \partial)^n \right)^i  \cdot a_d^{-1} .\]

\subsubsection*{Pas de morphismes de transition $\E_{k+1}(U) \to \E_k(U)$}

Nous ne disposons pas de morphismes de transition intéressants $ \E_{k+1}(U) \to \E_k(U)$ prolongeant les inclusions $\widehat{\mathcal{D}}^{(0)}_{\X , k+1 , \Q} (U) \hookrightarrow \Dk(U)$.
La propriété universelle du microlocalisé ne fournit pas de tel morphisme : l'inclusion $f_k := \phi_{k  | \widehat{\mathcal{D}}^{(0)}_{\X , k+1 , \Q} (U) } : \widehat{\mathcal{D}}^{(0)}_{\X , k+1 , \Q} (U)\hookrightarrow \Ek$ ne vérifie pas la condition d'uniforme continuité de la propriété universelle \ref{prop1.2} appliquée à l'algèbre $\E_{k+1}(U)$.
En effet, cette condition s'écrit
\[ \exists c>0 ~~ \mathrm{tq} ~~  \forall n\in\N ,~~  \forall P \in \widehat{\mathcal{D}}^{(0)}_{\X , k+1 , \Q} (U),  ~~ | \partial^{-n} \cdot P |_k \leq c \cdot |\partial|_{k+1}^{-n} \cdot |P|_{k+1} . \]
On a $|\partial^{-1} |_k = |\varpi|^k = p^{-k}$ et $|\partial^{-1}|_{k+1} = p^{-k+1}$. Ainsi, pour $P=1$, il est impossible de trouver une telle constante $c$.
Le morphisme $f_k$ ne se factorise donc pas naturellement en un morphisme de $\E_{k+1}(U)$ dans $\Ek$ étendant l'inclusion de $\widehat{\mathcal{D}}^{(0)}_{\X , k+1 , \Q} (U)$ dans $\Dku$.
Pour tout $P \in \widehat{\mathcal{D}}^{(0)}_{\X , k+1 , \Q} (U)$, on a $ | P |_k \leq |P|_{k+1}$ avec égalité stricte dés que $\Nb(P) >0$.
En effet, si $P = \sum_{n \geq 0} a_n \cdot \partial^n$, alors $|P|_k = \max_{n\in \N} \left \{ |a_n| \cdot p^{kn} \right \} \leq |P|_{k+1} = \max_{n\in \N} \left\{ |a_n| \cdot p^{(k+1)n} \right \} $.
Puisque la norme de $\Ek$ est multiplicative, cette inégalité est inversée dès que l'on inverse $\partial$.
L'inclusion de $\widehat{\mathcal{D}}^{(0)}_{\X , k+1 , \Q} (U) \hookrightarrow \Dk(U)$ est inversée dès que l'on considère des opérateurs dont les puissances de $\partial$ sont toutes négatives.
Par exemple, $P = \sum_{n \in \Z_{<0}} \varpi^{-n} \cdot (\varpi^{k+1} \partial)^n$ est un élément de $\E_{k+1}(U) \backslash \Ek$.
Il n'existe donc pas de morphisme injectif continu de $\E_{k+1}(U)$ dans $\E_k(U)$ prolongeant l'inclusion $\widehat{\mathcal{D}}^{(0)}_{\X , k+1 , \Q} (U) \hookrightarrow \Dkq(U)$.

\subsubsection*{Les faisceaux de $K$-algèbres $\E_k$}

Dans ce qui précède, le microlocalisé $\Ek$ a \'et\'e défini sur un ouvert affine $U$ de $\X$ sur lequel on dispose d'une coordonnée locale.
On rappelle que la $K$-algèbre de Banach $(\Dku , | \cdot |_k)$ ne dépend pas du choix de la coordonnée locale.
Puisque $\Ek$ est défini par une propriété universelle au dessus de $\Dku$, il est indépendant du choix de la coordonnée locale.
On note $\mathcal{U}$ l'ensemble des ouverts affines de $\X$ sur lesquels on dispose d'une coordonnée locale.
C'est une base de voisinages ouverts de $\X$.

\begin{prop}
Il existe un unique faisceau $\E_k$ de $K$-algèbres sur $\X$ tel que pour tout ouvert $U$ de $\mathcal{U}$, $\Ek = \Dkq(U) \left\langle | \cdot |_k , \{\partial^n\}_{n\in\N} \right\rangle$.
\end{prop}
\begin{proof}
Soit $V \subset U$ deux ouverts de $\mathcal{U}$.
Puisque $\E_k(V)$ ne dépend pas du choix de la coordonnée locale, le morphisme composé $\Dku \to \Dkq(V) \to \E_k(V)$ vérifie les hypothèses de la propriété universelle \ref{prop1.2} appliquée à $\E_k(V)$.
Il existe donc un unique morphisme $\E_k(U) \to \E_k(V)$ induit par le morphisme $\Dku \to \Dkq(V)$.
On en déduit que $\E_k$ est un préfaisceau sur l'ensemble des ouverts de $\mathcal{U}$.
Puisque $\mathcal{U}$ est une base de voisinages ouverts de $\X$, il existe un unique préfaisceau $\E_k$ sur $\X$ coïncidant avec les $\E_k(U)$ sur les ouverts $U$ de $\mathcal{U}$.
Enfin, l'unicité de l'écriture des éléments de $\Ek$ comme séries en les puissances de $(\varpi^k \partial)$ implique que $\E_k$ est un faisceau sur les ouverts $U$ de $\mathcal{U}$.
On en déduit que $\E_k$ est un faisceau de $K$-algèbres sur $\X$.
\end{proof}

Pour r\'esumer, on a construit un faisceau de $K$-algèbres $\E_k$ admettant $\Dkq$ comme sous-faisceau pour lequel la dérivation $\partial$ est localement inversible.
Cependant, on ne dispose pas de morphismes de transition naturels $\E_{k+1} \to \E_k$ induits par les morphismes $\widehat{\mathcal{D}}^{(0)}_{\X, k+1, \Q} \to \Dkq$.
La proposition \ref{prop2.11} se traduit de la manière suivante.

\begin{prop}\label{propekinv}
Un élément $P \in \E_k(U)$ est inversible dans $\E_k(V)$ pour un ouvert $V \subset U$ de $\X$ si et seulement si $\Nb(P) = \nb(P)$ et si le coefficient d'indice $\Nb(P)$ de $P$ est inversible dans $\O_{\X , \Q}(V)$.
\end{prop}

Soit $P = \sum_{n \geq 0} a_n \cdot (\varpi^k\partial)^n$ un opérateur de $\Dkq(U)$. On note $d = \Nb(P)$.
On rappelle que si $P$ est inversible dans $\E_k(V)$, alors son inverse est donné par
\[ P^{-1} = (\varpi^k\partial)^{-d} \cdot  \sum_{i \in \N} \left(-\sum_{n \geq 1} \frac{a_{n+d}}{a_d} \cdot (\varpi^k \partial)^n \right)^i  \cdot a_d^{-1} . \]
L'inverse $P^{-1}$ est dans $\Dkq(V)$ si et seulement si $d = 0$. On retrouve ainsi les éléments localement inversibles du faisceau $\Dkq$, déjà donnés plus tôt dans la proposition \ref{prop2.2.3}.

\subsection{Des microlocalisés $\F_{k , r}$ avec des morphismes de transition}\label{section3.3}

On fixe tout d'abord l'ouvert affine $U$ sur lequel on dispose d'une cordonnée locale.
En ne prenant plus la seule norme $|\cdot|_k$ pour définir le microlocalisé de $\Dku$, mais un certain nombre des normes précédentes $|\cdot|_r , \dots , |\cdot|_k$ pour un entier naturel $r$ supérieur ou égal à un, la condition d'uniforme continuité pour obtenir des morphismes de transition devient évidente.
Dans cette partie, l'entier $r \geq 1$ est fixé et $k$ est toujours suppos\'e sup\'erieur ou \'égal \`a $r$.
La partie multiplicative $S$ est encore constituée des puissances positives de la dérivation $\partial$.

\begin{definition}
Pour tout entier $k \geq r$, on note $\Fkr := \Dku \langle |\cdot|_r , \dots ,  |\cdot|_k ; S \rangle$ et $\varphi_{k , r} : \Dku \hookrightarrow \Fkr$ le morphisme de $K$-algèbres associé.
\end{definition}

Soit $\|\cdot \|_{k,r}$ la norme de $\Fkr$ ; $(\Fkr , \|\cdot \|_{k,r})$ est une $K$-algèbre de Banach.
Par construction, $\varphi_{k , r}$ est une isométrie de $(\Dku , \max\{|\cdot|_r , \dots ,  |\cdot|_k\})$ dans $(\Fkr , \|\cdot \|_{k,r})$.
Pour tout opérateur différentiel $P$ de $\Dku$, on vérifie que $|P|_r \leq \dots \leq  |P|_k$. On en déduit l'égalité $\max\{|P|_r , \dots ,  |P|_k\} = |P|_k$.
Autrement dit, $\varphi_{k , r}$ est une isométrie de $(\Dku , |\cdot|_k)$ dans $(\Fkr , \|\cdot \|_{k,r})$.
On identifie donc $\Dku$ à une sous-algèbre de $\Fkr$ via le morphisme $\varphi_{k , r}$. Cette identification est compatible avec les normes.
Il est important de remarquer que la norme de $\Fkr$ n'est cette fois ci plus multiplicative mais seulement sous-multiplicative dés que $k>r$. En effet, on sait d'après le lemme \ref{lemme1.5} que
\[ \| \partial^{-1} \|_{k,r} = \max_{ r \leq i \leq k} \{ |\partial|_i^{-1} \} = \max_{ r \leq i \leq k} \{ |\varpi|^i \} = |\varpi|^r = p^{-r} . \]
Cependant, on a toujours $\| \partial \|_{k,r} = |\partial |_k = | \varpi |^{-k} = p^k$.
Par ailleurs, $\F_{k , k}(U) = \Ek$. En effet, on localise alors l'algèbre $\Dku$ seulement pour la norme $|\cdot |_k$.
Par ailleurs, le morphisme $f_k = \phi_{k  | \widehat{\mathcal{D}}^{(0)}_{\X , k+1 , \Q} (U) } : \widehat{\mathcal{D}}^{(0)}_{\X , k+1 , \Q} (U) \hookrightarrow \Fkr$ vérifie de manière évidente la condition d'uniforme continuité pour $\F_{k+1 , r}(U)$ avec pour constante $c = 1$.
En effet, pour tout entier $n \in \N$ et pour tout opérateur $P$ de $\widehat{\mathcal{D}}^{(0)}_{\X , k+1 , \Q} (U)$, on a
\[\max_{r\leq i \leq k} \{ |\partial^{-n} \cdot P |_i \} \leq  \max_{r\leq i \leq k+1} \{ |\partial^{-n} \cdot P |_i \} \leq  \max_{r\leq i \leq k+1} \{ |\partial|_i^{-n} \cdot |P |_i \} . \]
Par propriété universelle du localisé $\F_{k+1 , r}(U)$, le morphisme $f_k$ se factorise uniquement par un morphisme continu de $K$-algèbres $\varepsilon_{k , r} : \F_{k+1 ,r}(U) \to \Fkr$.
Ce morphisme est 1-lipschitzien, tout comme l'est l'inclusion de l'algèbre $\widehat{\mathcal{D}}^{(0)}_{\X , k+1 , \Q} (U)$ dans $\Dku$.
Ainsi, $\|\varepsilon_{k , r}(S)\|_{k , r} \leq \|S\|_{k+1 , r} $ pour tout élément $S$ de $\F_{k+1,r}(U)$.

\begin{prop}\label{prop4.1.23}
Tout élément $S$ de $\Fkr$ s'écrit uniquement sous la forme
\[ S = \sum_{n = 0}^\infty a_n \cdot (\varpi^k\partial)^n + \sum_{ n =1}^\infty  a_{-n} \cdot (\varpi^r\partial)^{-n} \]
avec $a_n \to 0$ lorsque $n \to \pm \infty$. De plus, $\| S \|_{k , r} = \max_{n\in \Z} \{ |a_n| \}$.
\end{prop}
\begin{proof}
On vérifie toujours que pour toute fonction $f \in \O_{\X,\Q}(U)$,
\[ \partial^{-1} f = \sum_{n=0}^{+\infty} (-1)^n \partial^n(f) \cdot \partial^{-(n+1)} \in \F_{k , r}(U) .\]
Cette série converge dans $\F_{k , r}(U)$ puisque $\| \partial^{-(n+1)} \|_{k , r} \leq \| \partial^{-1} \|_{k , r}^{n+1} \leq \left(\frac{1}{p^r}\right)^{n+1} \underset{n\to \infty}{\longrightarrow} 0$.
On d\'emontre ensuite que la suite $S_m = \sum_{n = 0}^ {\infty} a_n \cdot (\varpi^k\partial)^n + \sum_{n=-1}^{-m} a_n \cdot (\varpi^r\partial)^n$
est de Cauchy dans l'algèbre de Banach $(\F_{k , r}(U) , \| \cdot \|_{k , r})$.
Cependant, contrairement \`a la preuve de la proposition \ref{prop2.7}, la norme $\| \cdot \|_{k , r}$ est maintenant sous-multiplicative : il n'est donc pas clair que $\| S_m \|_{k , r} = \max_{n \geq -m} \{ |a_n| \}$.
Cette égalité nécessite plus de travail que pour l'algèbre $\Ek$. L'élément
\[ S'_m := S_m \cdot (\varpi^r \partial)^m = \sum_{n = 0}^\infty a_n \cdot \varpi^{rm} \varpi^{kn} \cdot \partial^{n+m} + \sum_{n=0}^{m-1} a_{n-m} \cdot (\varpi^r\partial)^n \]
appartient à $\Dku$. D'après la proposition \ref{prop1.4} et le lemme \ref{lemme1.5}, on sait que
\[ \|S_m\|_{k , r} = \|S'_m \cdot (\varpi^r \partial)^{-m} \|_{k , r} =  \| (\varpi^r \partial)^{-m} \cdot S'_m \|_{k , r} = \max_{r \leq \ell \leq k} \left\{ |\varpi^r \partial|_\ell^{-m} \cdot | S'_m|_\ell \right\} . \]
Pour $r \leq \ell \leq k$, on a
\[ S'_m = \sum_{n = 0}^\infty a_n \cdot \varpi^{rm+kn-\ell(n+m)} \cdot (\varpi^\ell \partial)^{n+m} + \sum_{n=0}^{m-1} a_{n-m} \cdot \varpi^{(r-\ell)n} \cdot (\varpi^\ell\partial)^n . \]
Comme $|\varpi^r\partial|_\ell^{-m} = |\varpi|^{(\ell - r)m}$, on obtient
\begin{align*}
|\varpi^r \partial|_\ell^{-m} \cdot | S'_m|_\ell  & = \max \left\{ 
 \max_{n\in\N} |a_n| \cdot |\varpi|^{rm+kn-\ell(n+m)+ (\ell-r)m} ~ ,  \right. \\
& \hspace{4.55cm} ~ \left.\max_{0 \leq n \leq m-1} |a_{n-m}| \cdot |\varpi|^{(r-\ell)n + (l-r)m} \right\} \\
& = \max \left\{ \max_{n\in\N}  |a_n| \cdot |\varpi|^{(k-\ell)n} ~ , ~\max_{0 \leq n \leq m-1}  |a_{n-m}| \cdot |\varpi|^{((\ell-r)(m-n)}  \right\} .
\end{align*}
Le terme $\max_{n\in\N}\{  |a_n| \cdot |\varpi|^{(k-\ell)n} \}$ est maximum pour $\ell = k$.
En effet, lorsque $\ell \in \{1 , \dots , k \}$, la puissance $(k-l)n$ de $\varpi$ est positive et $ |\varpi|^{(k-\ell)n} \leq 1$.
De manière analogue, le terme $\max_{0 \leq n \leq m-1} \{ |a_{n-m}| \cdot |\varpi|^{((\ell-r)(m-n)} \}$ est maximum pour $\ell = r$, auquel cas on obtient $\max_{0 \leq n \leq m-1} |a_{n-m}| = \max_{-m \leq n <0 } |a_n|$. On en déduit finalement que
\[ \|S_m\|_{k , r} = \max \left\{ \max_{n\in\N} |a_n| , \max_{-m \leq n <0 } |a_n| \right\} = \max_{n \geq -m} |a_n|  . \]
Il en découle immédiatement que la suite $(S_m)_{m\in\N}$ est de Cauchy.
Elle converge donc dans la $K$-algèbre de Banach $\Fkr$ vers un élément que l'on note $S$. 
Enfin, le passage à la limite $m \to \infty$ donne $\| S \|_{k , r} = \max_{n\in \Z} \{ |a_n| \}$.
Le reste de la preuve est analogue à celle de la proposition \ref{prop2.7}.
\end{proof}

Ainsi, tout élément $S$ du microlocalisé $\Fkr$ s'écrit uniquement sous la forme
\[ S = \underbrace{\sum_{n = 0}^\infty a_n \cdot (\varpi^k\partial)^n}_{P \in \Dku} + \underbrace{\sum_{n = 1}^{\infty} a_{-n} \cdot (\varpi^r \partial)^{-n}}_Q \]
avec $a_n \to 0$ lorsque $n \to \pm \infty$ et $\| S \|_{k,r} = \max_{n\in \Z} \{ |a_n| \} = \max\{|P|_k , \|Q\|_{k,r} \}$.
Soient $Q$ et $Q'$ deux éléments du microlocalisé $\Fkr$ formés seulement de puissances strictement négatives de la dérivation $\varpi^r\partial$.
On observe que $\| Q \cdot Q' \|_{k,r} = \| Q \|_{k,r} \cdot \| Q' \|_{k,r}$.
Autrement dit, la norme $\| \cdot \|_{k,r}$ de $\Fkr$ est multiplicative pour le produit de tels éléments, comme elle l'est pour le produit d'opérateurs différentiels de $\Dku$.
On note donc $|Q|_r$ pour $\|Q\|_{k,r}$ et $\| S \|_{k,r} = \max\{|P|_k , |Q|_r \} $.
Par ailleurs, l'action de la dérivation $\partial^{-1}$ sur une fonction $f$ de $\O_{\X,\Q}(U)$ est encore donnée par $\partial^{-1} f = \sum_{n=0}^{+\infty} (-1)^n \partial^n(f) \cdot \partial^{-(n+1)} \in \F_{k , r}(U)$.
On rappelle que le morphisme de transition $\varepsilon_{k , r} : \F_{k+1 ,r}(U) \to \Fkr$ est une factorisation de l'inclusion de $\widehat{\mathcal{D}}^{(0)}_{\X , k+1 , \Q} (U)$ dans $\Fkr$.
Comme $\varepsilon_{k,r}$ envoie la dérivation $\partial$ sur elle-même, $\varepsilon_{k,r}(\partial^n) = \partial^n$ pour tout entier relatif $n \in \Z$.
Puisque ce morphisme est continu, on a $\varepsilon_{k,r}(S) = S$ pour tout élément $S$ de $\F_{k+1 , r}(U)$.
Ainsi, $\varepsilon_{k,r}$ injecte l'algèbre $\F_{k+1 , r}(U)$ dans $\Fkr$.
Pour récapituler, on obtient le diagramme commutatif
\[ \xymatrix@R=3pc @C=3pc{ \widehat{\mathcal{D}}^{(0)}_{\X , k+1 , \Q} (U) \ar@{^{(}->}[r]^{i_k} \ar@{^{(}->}[d]_{\varphi_{k+1 , r}} & \Dku \ar@{^{(}->}[d]^{\varphi_{k, r}} \\ \F_{k+1 , r}(U) \ar@{^{(}->}[r]_{\varepsilon_{k,r}} & \Fkr } \]
où $i_k$ et $\varepsilon_{k,r}$ sont des inclusions 1-lipschitziennes et $\varphi_{k , r}$ et $\varphi_{k+1 , r}$ des isométries.
La proposition suivante relie les $K$-algèbres $\Fkr$ et $\Ek$ : le microlocalisé $\Fkr$ est une sous-algèbre de $\Ek$.

\vspace{0.4cm}

On rappelle que $\mathcal{U}$ est l'ensemble des ouverts affines de $\X$ sur lesquels on dispose d'une coordonnée locale.
C'est une base de voisinages ouverts de $\X$. On démontre comme pour le faisceau localisé $\E_k$ le résultat suivant.

\begin{prop}
Il existe un unique faisceau $\F_{k,r}$ de $K$-algèbres sur $\X$ tel que pour tout ouvert $U$ de $\mathcal{U}$, $\Fkr = \Dkq(U) \left\langle |\cdot|_r , \dots ,  |\cdot|_k, \{\partial^n\}_{n\in\N} \right\rangle$.
\end{prop}

Les inclusions locales $\F_{k+1 , r} (U) \hookrightarrow \Fkr$ pour tout ouvert $U$ de $\mathcal{U}$ induisent des morphismes de faisceaux $\F_{k+1 ,r} \to \F_{k , r}$ commutant avec les morphismes de transition $\widehat{\mathcal{D}}^{(0)}_{\X, k+1, \Q} \to \Dkq$.
On dispose ainsi du diagramme commutatif suivant :
\[ \xymatrix@R=3pc @C=3pc{ \widehat{\mathcal{D}}^{(0)}_{\X , k+1 , \Q} \ar[r] \ar[d] & \Dkq \ar[d] \\ \F_{k+1 , r} \ar[r] & \F_{k , r} \,.} \]

\begin{prop}
Le morphisme $\phi_k : \Dku \hookrightarrow \Ek$ se relève en un morphisme injectif de $K$-algèbres $\gamma_{k , r} : \Fkr \hookrightarrow \Ek$.
De plus, l'application $\gamma_{k , r}$ est 1-lipschitzienne.
\end{prop}
\begin{proof}
Le morphisme $\phi_k : \Dku \hookrightarrow \Ek$ vérifie clairement la condition d'uniforme continuité dans la propriété universelle du localisé $\Fkr$.
Il induit donc un unique morphisme de $K$-algèbres 1-lipschitzien $\gamma_{k,r} : \Fkr \to \Ek$.
Ce morphisme envoie les éléments de $\Dku$ sur eux mêmes et $\partial^{-1}$ sur $\partial^{-1}$.
Par continuité, il envoie tout élément $S$ de $\Fkr$ sur l'élément $S$ de $\Ek$.
Ainsi, le morphisme $\gamma_{k,r}$ est injectif.
\end{proof}

Comme pour $\Ek$, on munit l'algèbre $\Fkr$ de la filtration croissante donnée par la norme sous-multiplicative $\| \cdot \|_{k,r}$.
On note $\gr(\Fkr)$ le gradu\'e associ\'e ainsi que $\zeta_k$ et $\zeta_r$ les images respectives des dérivations $(\varpi^k \partial)$ et de $(\varpi^r\partial)^{-1}$ dans le gradué $\gr_0 (\Fkr)$.
On d\'eduit de la description de $\Fkr$ donn\'ee dans la proposition \ref{prop4.1.23} que le groupe ab\'elien $(\gr\Fkr , +)$ est engendr\'e librement par les \'el\'ements $\gamma(a)\cdot \xi_\ell^m$\footnote{L'application $\gamma : \O_{\X , \Q}(U) \to \gr \O_{\X , \Q}(U)$ est explicit\'e dans preuve de la proposition \ref{propnoeth}.} pour $a \in \O_{\X , \Q}(U)$, $m \in \N$ et $\ell \in \{ k , r \}$ :
\begin{equation}\label{eqgr}
\gr (\Fkr) := \bigoplus_{m \in \Z} \gr_m(\Fkr) = \bigoplus_{\ell, m \in \Z} \left(\gr_m(\O_{\X , \Q}(U))\cdot \xi_k^\ell  \oplus \gr_m(\O_{\X , \Q}(U))\cdot \xi_r^\ell \right)
\end{equation}
Pour $k = r$, on rappelle que $\F_{k,k}(U) = \Ek$ et que $\gr (\Ek) \simeq \gr(\O_{\X , \Q}(U)) [ \zeta_k , \zeta_k^{-1}]$.
On munit le gradu\'e  $\gr(\Fkr)$ du produit induit par $\gamma : \Fkr \to \gr \Fkr$ de manière analogue \`a la preuve de la proposition \ref{propnoeth}.
Plus précisément, soit $P$ et $Q$ deux \'el\'ements de $\Fkr$.
Si $\|P\|_{k , r} = p^m$ et $\|Q\|_{k , r} = p^n$, alors $\gamma(P)$ est l'image de $P$ dans $\gr_m(\Fkr)$ et $\gamma(P)\cdot \gamma(Q)$ est défini comme l'image $\gamma(P \cdot Q)$ de $P \cdot Q$ dans $\gr_{m + n}(\Fkr)$.

\begin{lemma}
On suppose que $k > r \geq 1$. Alors $\gr\Fkr \simeq \gr(\O_{\X , \Q}(U)) [ \zeta_k , \zeta_r] / (\zeta_k \cdot \zeta_r)$ en tant que $\gr K$-algèbres.
\end{lemma}
\begin{proof}
La norme de $\Fkr$ étant sous-abélienne, le gradué $\gr \Fkr$ est une $\gr K$-algèbre commutative.
Puisque $(\varpi^k \partial) \cdot ( \varpi^r \partial)^{-1} = \varpi^{k-r}$, on a $\| (\varpi^k \partial) \cdot ( \varpi^r \partial)^{-1} \| _{k,r} = p^{r-k} < \| \varpi^k \partial\|_{k , r} \cdot \| ( \varpi \partial)^{-1} \|_{k , r} = 1$.
Ainsi, $\zeta_k \cdot \zeta_r = 0$ dans $\gr(\Fkr)$.
Ce fait et la description \ref{eqgr} de $\gr(\Fkr)$ fournissent une surjection de $\gr K$-algèbres $\varphi : \gr(\O_{\X , \Q}(U)) [ \zeta_k , \zeta_r] \twoheadrightarrow \gr \Fkr$ telle que $(\zeta_k \cdot \zeta_r ) \subset I = \ker(\varphi)$ et envoyant les polynômes respectivement en $\xi_k$ et en $\xi_r$ sur eux-m\^eme.
Autrement dit, $\varphi$ est injectif sur les \'el\'ements de la forme $\sum_n \alpha_n \cdot \xi_k^n + \sum_m \beta_m \cdot \xi_r^m$.
Ainsi, le noyau $I$ de $\varphi$ est exactement l'id\'eal engendr\'e par $\zeta_k \cdot \zeta_r$.
On en déduit que $\gr (\Fkr) \simeq \gr(\O_{\X , \Q}(U)) [ \zeta_k , \zeta_r] / (\zeta_k \cdot \zeta_r)$ en tant que $\gr K$-algèbres.
\end{proof}

Le gradu\'e $\gr \Fkr$ admettant de la torsion, il ne permet plus d'obtenir directement des résultats de platitude pour le microlocalis\'e $\F_{k , r}(U)$.
La preuve de la proposition suivante est donc bien plus technique et complexe que celle de la proposition \ref{propnoeth}.

\begin{prop}\label{prop2.28}
La $K$-algèbre $\Fkr$ est noetherienne et les morphismes de faisceaux $\varepsilon_{k , r} : \F_{k+1 , r} \to \F_{k  , r}$ et $\varphi_{k , r} : \Dkq \to \F_{k , r}$ sont plats, le tout à gauche et à droite.
\end{prop}
\begin{proof}
Puisque le gradué $\gr \Fkr \simeq \gr(\O_{\X , \Q}(U)) [ \zeta_k , \zeta_r] / (\zeta_k \cdot \zeta_r)$ est noetherien, l'algèbre $\Fkr$ l'est aussi toujours d'après \cite[proposition 1.1]{schneider}.
On adapte ensuite la preuve de \cite[lemme 1.3.4]{adriano} au cas du morphisme $\Dkq \to \F_{k , r}$ ; le lecteur peut s'y référer pour plus de détails.
Il suffit de montrer que $\mathrm{Tor}_1^{\Dkq}(\F_{k , r} , \bullet) = 0$.
Comme $\Dkq \to \widehat{\D}^{(0)}_{\X , r , \Q}$ est plat par \cite[proposition 2.2.16]{huyghe}, le faisceau $\E_r = \F_{r , r}$ est plat sur $\Dkq$ d'après la proposition \ref{propnoeth}.
En conséquence, $\mathrm{Tor}_1^{\Dkq}(\F_{k , r} , \bullet) = 0$ dès que $\mathrm{Tor}_1^{\Dkq}(\F_{r , r} / \F_{k , r} , \bullet) = 0$.
Puisque $\F_{ r , r} / \F_{k , r} \simeq \widehat{\D}^{(0)}_{\X , r , \Q} / \Dkq$ et $\Dkq \to \widehat{\D}^{(0)}_{\X , r , \Q}$ est plat, on obtient la platitude de $\F_{k , r}$ sur $\Dkq$.
Enfin, pour démontrer que le morphisme $\varepsilon_{k , r} : \F_{k+1 , r} \to \F_{k  , r}$ est plat à gauche et à droite, on adapte la preuve de la proposition 2.2.16 du m\^eme article \cite{huyghe} aux cas des faisceaux $\F_{k , r}$ pour $r \geq 1$ fixé.
On fixe un ouvert affine $U$ de $\X$ sur lequel on dispose d'une  coordonnée locale.
Il suffit de montrer que le morphisme de $K$-algèbres $\F_{k+1 , r}(U) \hookrightarrow \F_{k  , r}(U)$ est plat à gauche et à droite.
On note $A = \O_\X(U)$ et $\F_{k , r}^\circ(U)$ l'ensemble des éléments de $\F_{k , r}(U)$ de norme $\| \cdot \|_{k , r}$ inférieure ou égale à un :
\[ \F_{k,r}^\circ(U) =\left\{ \sum_{n =0}^\infty a_n \cdot (\varpi^k \partial)^n + \sum_{n = 1}^{\infty} a_{-n} \cdot (\varpi^r \partial)^{-n} \in \Fkr , ~~ a_n \in A \right\} . \]
C'est une $A$-algèbre vérifiant $\F_{k , r}(U) = \F_{k , r}^\circ(U) \otimes_\V K$.
On introduit le $A$-module
\[ H := \F_{k+1 , r}^\circ(U) + \mathcal{D}^{(0)}_{\X, k}(U) . \]
La dérivation $(\varpi^k \partial)$ n'appartient pas à la $\V$-algèbre $\F_{k+1,r}^\circ(U)$.
Cependant, elle appartient à la $K$-algèbre $\F_{k+1 , r}(U) = \F_{k+1,r}^\circ(U) \otimes_\V K$.
La démonstration de la platitude du morphisme $\F_{k+1 , r}(U) \hookrightarrow \F_{k  , r}(U)$ se décompose en trois étapes.
\begin{enumerate}
\item
Vérifions tout d'abord que le module $H$ est une $A$-algèbre.
Soit $Q \in \F_{k+1 , r}^\circ(U)$ et $P \in \mathcal{D}^{(0)}_{\X, k}(U)$.
Comme $k \geq r$, on observe facilement que $P \cdot Q$ appartient au $A$-module $H$.
En effet, on considère pour cela le $A$-module suivant :
\[ F_{k , r}(U) := \left\{ \sum_{n = 0}^\infty a_n \cdot (\varpi^k \partial)^n+ \sum_{n = 1}^{\infty} a_{-n} \cdot (\varpi^r \partial)^{-n} , ~~ a_n \in A , ~~ a_n = 0 ~\mathrm{pour}~ n << 0 \right\} . \]
Il est clair que l'on dispose d'inclusions $F_{k+1 , r}(U) \subset F_{k , r}(U) \subset H$.
On peut décomposer l'élément $Q$ sous la forme $Q = Q_1 + \varpi \cdot Q_2$ avec $Q_1 \in F_{k+1 , r}(U)$ et $Q_2 \in \F_{k+1 , r}^\circ(U)$.
On a $(\varpi^k \partial) \cdot Q_1 \in F_{k,r}(U) \subset H$.
Par ailleurs, $(\varpi^k \partial) \cdot (\varpi Q_2) = (\varpi^{k+1} \partial) \cdot Q_2 \in \F_{k+1 ,r}^\circ(U)$. Autrement dit, $(\varpi^k \partial) \cdot Q \in H$.
On en déduit que $P \cdot Q \in H$.
Pour montrer que $H$ est une $A$-algèbre, il suffit maintenant d'observer que $Q \cdot P \in H$.
On rappelle que pour toute fonction $a$ de $A$, on a
\[ (\varpi^r\partial)^{-1} \cdot a = \sum_{n=0}^{+\infty} (-1)^n \partial^n(a)\varpi^{rn} \cdot (\varpi^r\partial)^{-(n+1)} \in \F_{k+1 , r}^\circ(U) \]
avec $(-1)^n \partial^n(a)\varpi^{rn} \underset{n\to \infty}{\longrightarrow} 0$.
En particulier, comme $k \geq r$,
\[ (\varpi^r \partial)^{-1} \cdot a (\varpi^k \partial) = \varpi^{k-r} \cdot\sum_{n=0}^{+\infty} (-1)^n \partial^n(a) \cdot \partial^{-n} \in \F_{k+1 , r}^\circ(U) . \]
On en déduit que pour tout entier $\ell \in \N$,
\[ (\varpi^r \partial)^{-\ell} \cdot P \in H = \F_{k+1 , r}^\circ(U) + \mathcal{D}^{(0)}_{\X, k}(U) . \]
Ainsi, pour tout élément fini $\tilde{Q}$ de $\F_{k+1 , r}^\circ(U)$, il en découle que $\tilde{Q} \cdot P \in H$.
Soit $d$ le degré de l'opérateur différentiel fini $P$. L'élément $Q$ s'écrit uniquement sous la forme
\[ Q = \underbrace{\sum_{n =0}^\infty a_n \cdot (\varpi^{k+1} \partial)^n}_{Q_1} + \underbrace{\sum_{n = 1}^d a_{-n} \cdot (\varpi^r \partial)^{-n} }_{Q_2}+ \underbrace{\sum_{n = d+1}^{\infty} a_{-n} \cdot (\varpi^r \partial)^{-n}}_{Q_3} \]
Puisque $Q_1 \in \widehat{\mathcal{D}}^{(0)}_{\X, k+1}(U)$, il est clair que $Q_1 \cdot P \in  \widehat{\mathcal{D}}^{(0)}_{\X, k+1}(U) + \mathcal{D}^{(0)}_{\X, k}(U) \subset H$.
D'après ce que l'on vient d'établir ci-dessus, $Q_2 \cdot P \in H$.
Soit $m > d$ et $Q_3^{(m)}$ la troncation à l'ordre $-m$ de $Q_3$ :
\[ Q_3 :=  \sum_{n = d}^m a_{-n} \cdot (\varpi^r \partial)^{-n} .\]
Comme $(\varpi^r \partial)^{-1}$ diminue strictement le degré et comme $d$ est le degré de l'opérateur $P$, on vérifie que $Q_3^{(m)} \cdot P \in \F_{k+1 , r}^\circ(U)$.
De plus, on vérifie que la suite $(Q_3^{(m)} \cdot P)_{m > d}$ est de Cauchy dans l'algèbre $(\F_{k+1 , r}^\circ(U) , \| \cdot \|_{k , r})$.
Enfin, puisque la $\V$-algèbre $\F_{k+1 , r}^\circ(U)$ est complète, on en déduit que $Q_3 \cdot P \in \F_{k+1 , r}^\circ(U)$.
Pour conclure, on a bien montré que $Q \cdot P \in H$.
On a ainsi démontré que $H$ est une $A$-algèbre.

\item
Démontrons ensuite que la $A$-algèbre $H$ est noetherienne à gauche et à droite.
L'algèbre $\F_{k+1 , r}^\circ(U)$ est noetherienne d'après l'article \cite[proposition 1.1]{schneider}.
En effet, son gradué est isomorphe à l'algèbre noetherienne $\gr(A)[ \zeta_k , \zeta_r] / (\zeta_k \cdot \zeta_r)$ et $\F_{k+1 , r}^\circ(U)$ est complet pour la topologie $\varpi$-adique.
L'algèbre $H$ est engendrée sur $A$ par $\F_{k+1 , r}^\circ(U)$ et par la dérivation  $(\varpi^k \partial)$.
Pour toute fonction $a$ de $A$, on a
\[ [\varpi^k \partial , a ] = \varpi^k \partial(a) \in A \subset \F_{k+1 , r}^\circ(U). \]
Comme $ [\varpi^k \partial , \varpi^{k+1} \partial ] = [\varpi^k \partial , (\varpi^r \partial)^{-1} ] = 0$, $[\varpi^k \partial , P ] \in \F_{k+1 , r}^\circ(U)$ pour tout opérateur différentiel fini $P$ de $\F_{k+1 , r}^\circ(U)$.
Enfin, puisque la $\V$-algèbre $\F_{k+1 , r}^\circ(U)$ est complète pour la topologie $\varpi$-adique, on en déduit que pour tout $P \in \F_{k+1 , r}^\circ(U)$, $[\varpi^k \partial , P ] \in \F_{k+1 , r}^\circ(U)$.
Une récurrence sur $\ell \in \N^*$ implique alors que
\begin{equation}\label{eqcommutateur}
\forall P \in \F_{k+1 , r}^\circ(U), ~~ [(\varpi^k \partial)^\ell , P ] \in \sum_{i=0}^{\ell-1} \F_{k+1 , r}^\circ(U) \cdot (\varpi^k \partial)^i
\end{equation}
Le même argument que celui de la fin de la preuve de la proposition 2.2.16 de l'article \cite{huyghe} montre ensuite que $H$ est noetherien.
On le rappelle brièvement ci-dessous.
Soit $I$ un idéal à gauche de $H$. On note $J$ l'ensemble des éléments $R$ de $\F_{k+1 , r}^\circ(U)$ pour lesquels il existe $\ell \in \N$, $P \in I$ et $R_0 , \dots , R_{\ell-1} \in \F_{k+1 , r}^\circ(U)$ tels que
\[ P = R \cdot (\varpi^k \partial)^\ell + \sum_{i=0}^{\ell-1} R_i \cdot (\varpi^k \partial)^i . \]
Soit $R'$ un autre élément de $J$, où $P' = R' \cdot (\varpi^k \partial)^{\ell'} + \sum_{i=0}^{\ell'-1} R'_i \cdot (\varpi^k \partial)^i \in I$.
On suppose par exemple que $\ell \geq \ell'$.
La relation \ref{eqcommutateur} implique que
\[ P + (\varpi^k \partial)^{\ell - \ell'} \cdot P' = (R + R') \cdot (\varpi^k \partial)^\ell + \sum_{i=0}^{\ell - 1} R_i^{''} \cdot (\varpi^k \partial)^i \]
pour certains $R_i^{''}  \in \F_{k+1 , r}^\circ(U)$.
Autrement dit, $R + R' \in J$ et $J$ est un idéal à gauche de $\F_{k+1 , r}^\circ(U)$.
Ainsi, $J$ est engendré sur $\F_{k+1 , r}^\circ(U)$ par des éléments $R_1 , \dots , R_s$.
Par ailleurs, $I \cap \F_{k+1 , r}^\circ(U)$ est un idéal à gauche de $\F_{k+1 , r}^\circ(U)$. Il est donc engendré par des éléments $P_1, \dots , P_t$.
Enfin, on observe que $I$ est engendré en tant qu'idéal à gauche de $H$ par les éléments $R_1 , \dots , R_s, P_1, \dots , P_t$.
Ceci prouve que $H$ est noetherien à gauche. De même, on démontre que $H$ est noetherien à droite.

\item
Montrons enfin que le morphisme d'algèbres $\F_{k+1 , r}(U) \hookrightarrow \F_{k  , r}(U)$ est plat à gauche et à droite.
Comme $\mathcal{D}^{(0)}_{\X, k+1 , \Q}(U) = \mathcal{D}^{(0)}_{\X, k , \Q}(U)$, on a $H_\Q := H \otimes_\V K = \F_{k+1 , r}(U)$.
Par ailleurs, il est clair que $\widehat{H} = \F_{k , r}^\circ(U)$.
Puisque $\F_{k+1 ,r}^\circ(U) \subset \F_{k,r}^\circ(U)$ et $\mathcal{D}^{(0)}_{\X, k}(U) \subset \F_{k,r}^\circ(U)$, on a $H \subset \F_{k,r}^\circ(H)$.
De plus, l'algèbre $H$ contient $F_{k,r}(U)$ et l'algèbre $\F_{k,r}^\circ(U)$ est le complété $\varpi$-adique de $F_{k,r}(U)$.
Le complété $\varpi$-adique $\widehat{H}$ de $H$ est donc égal à la $\V$-algèbre $\F_{k,r}^\circ(U)$.
Le résultat de platitude découle alors du fait que l'algèbre $H$ est noetherienne.
En effet, dans ce cas le complété $\widehat{H}$ est plat sur $H$ et donc l'algèbre $\widehat{H}_\Q = \F_{k , r}(U)$ est plate sur $H_\Q = \F_{k +1 , r}(U)$.
\end{enumerate}
\end{proof}

\subsubsection*{Opérateurs de $\Dkq$ inversibles dans le microlocalisé $\F_{k , r}$}

Soit toujours $U$ un ouvert affine sur lequel on dispose d'une coordonnée locale.
Il est plus délicat de déterminer les éléments inversibles de l'algèbre $\F_{k,r}(U)$ que ceux de $\F_{k , k}(U) = \Ek$ lorsque $k>r$.
En effet, tout élément $S$ de $\Fkr$ s'écrit uniquement comme une série
\[ S = \sum_{n = 0}^\infty a_n \cdot (\varpi^k\partial )^n + \sum_{n = 1}^{\infty} a_{-n} \cdot (\varpi^r \partial)^{-n} . \]
On note $\partial_k$ et $(\partial_r)^{-1}$ les réductions modulo $\varpi$ respectivement de $\varpi^k \partial$ et $(\varpi^r\partial)^{-1}$. La réduction modulo $\varpi$ de $S$ est une somme finie de la forme $\bar{S} = \sum_{n < 0} \bar{a}_n \cdot \partial_r^{n} + \sum_{n \geq 0} \bar{a}_n \cdot \partial_k ^n$.
Pour $k >r$, on a $\partial_k \cdot \partial_r^{-1} = \partial_r^{-1} \cdot \partial_k = 0$ puisque $(\varpi^k \partial) \cdot (\varpi^r \partial)^{-1} = \varpi^{k-r}$.
Autrement dit, la dérivation $\partial_k$ n'est plus inversible dans le quotient $\F_{k,r}^\circ(U) / \varpi$.
On ne peut donc pas déterminer les éléments inversibles de l'algèbre $\Fkr$ après réduction modulo $\varpi$.
Le même phénomène apparait en considérant le gradué $\gr \F_{k, r}(U)$.
Bien que la dérivation $\partial$ soit inversible dans $\F_{k, r}(U)$, l'image de $\partial$ dans le gradué n'est plus inversible.
Cependant, on dispose d'un morphisme continu et injectif $\varepsilon_{k,r} : \Fkr \hookrightarrow \Ek$.
Puisque les éléments inversibles de l'algèbre $\Ek$ sont connus, on en déduit facilement les opérateurs différentiels de $\Dku$ inversibles dans le microlocalisé $\Fkr$.
Ce cas est suffisant pour définir la variété caractéristique à partir du microlocalisé.
Soit $P = \sum_{n \geq 0} a_n \cdot (\varpi^k \partial)^n$ un opérateur différentiel de $\Dku$.
Les normes de $P$ dans les algèbres $\Ek$ et dans $\Fkr$ coïncident avec la norme usuelle $|P|_k$ de $P$ dans $\Dku$.
Si $P$ est inversible dans l'algèbre $\Fkr$, alors $P$ est inversible dans $\Ek$.
Il doit donc vérifier l'égalité $\Nb(P) = \nb(P)$, o\`u $\Nb(P) = \max \{ n \in \N : |a_n| = |P|_k \}$ et $\nb(P)= \min \{ n \in \N : |a_n| = |P|_k \}$.
On note $d = \Nb(P) = \nb(P)$ cet entier. Pour que $P$ soit inversible dans le localisé $\Ek$, il faut de plus que son coefficient dominant $a_d$ soit inversible dans la $K$-algèbre affinoïde $\O_{\X , \Q}(U)$.
L'inverse de $P$ dans $\Ek$ est alors
\[ P^{-1} = (\varpi^k\partial)^{-d} \cdot  \sum_{i \in \N} \left(-\underbrace{\sum_{n = 0}^{d-1} \frac{a_n}{a_d} \cdot (\varpi^k \partial)^{n-d}}_\alpha - \underbrace{\sum_{n \geq 1} \frac{a_{n+d}}{a_d} \cdot (\varpi^k \partial)^n}_\beta \right)^i  \cdot a_d^{-1} .\]
L'opérateur $P^{-1}$ appartient à l'algèbre $\Fkr$ si et seulement si la somme sur $i$ converge pour la norme $\| \cdot \|_{k,r}$ de $\Fkr$.
Une condition suffisante est d'avoir $\| \alpha + \beta \|_{k , r} < 1$.
On rappelle l'égalité des normes $\| \alpha + \beta \|_{k , r} = \max\{ \| \alpha \|_{k , r} , \| \beta \|_{k , r}\}$ démontrée dans la proposition \ref{prop4.1.23}.
On a toujours
\[ \| \beta \|_{k,r} = | \beta |_k = \left| \sum_{n \geq 1} \frac{a_{n+d}}{a_d} \cdot  (\varpi^k \partial)^n \right|_k < 1 . \]
En effet, l'opérateur $\beta = \sum_{n \geq 1} \frac{a_{n+d}}{a_d} \cdot (\varpi^k \partial)^n$ appartient à l'algèbre $\Dkq(U)$ et par définition de $d$, $| a_{n+d} | < |a_d|$ dès que $n \geq 1$.
Par ailleurs,
\[ \alpha = \sum_{n = 0}^{d-1} \frac{a_n}{a_d} \cdot (\varpi^k \partial)^{n-d} = \sum_{n=0}^{d-1} \frac{a_n}{a_d} \cdot \varpi^{(k-r)(n-d)} \cdot (\varpi^r \partial)^{n-d} . \]
On obtient
\[ \| \alpha \|_{k,r} = \max_{0 \leq n < d} \left( \frac{|a_n|}{|a_d|} \cdot |\varpi|^{(k-r)(n-d)} \right) .\]
On en déduit que l'opérateur différentiel $P$ est inversible dans le microlocalisé $\F_{k,r}(U)$ dès que $|a_n| \cdot p^{(k-r)(d-n)} < |a_d|$ pour tout $n \in \{0, \dots, d-1\}$.
La réciproque est vraie. Cependant, elle n'est pas immédiate puisque la norme $\| \cdot \|_{k , r}$ de la $K$-algèbre $\F_{k , r}(U)$ est seulement sous-multiplicative.

\begin{prop}\label{prop2.22}
Soit $P = \sum_{n = 0}^\infty a_n \cdot (\varpi^k \partial)^n$ un opérateur différentiel de $\Dkq(U)$ et $V \subset U$ un ouvert.
Alors $P$ est inversible dans $\F_{k , r}(V)$ si et seulement si
\begin{enumerate}
\item
$\Nb(P) = \nb(P) = d$,
\item
$a_d$ est inversible dans $\O_{\X , \Q}(V)$,
\item
$ \forall n \in \{ 0 , \dots , d-1 \}$, $|a_n| \cdot p^{(k-r)(d-n)} < |a_d|$.
\end{enumerate}
\end{prop}
\begin{proof}
Si $P$ vérifie ces trois conditions, alors $P$ est inversible dans le microlocalisé $\F_{k ,r}(V)$.
Réciproquement, si $P$ est inversible dans $\F_{k , r}(V)$, alors $P$ vérifie les deux premières conditions.
En effet, $P$ est à fortiori inversible dans le localisé $\E_k(V)$.
Il reste à établir que $P$ vérifie le point 3. On rappelle que l'inverse $P^{-1}$ de $P$ dans $\E_k(V)$ s'écrit
\[ P^{-1} = (\varpi^k\partial)^{-d} \cdot  \sum_{i \in \N} \left(-\underbrace{\sum_{n = 0}^{d-1} \frac{a_n}{a_d} \cdot (\varpi^k \partial)^{n-d}}_\alpha - \underbrace{\sum_{n \geq 1} \frac{a_{n+d}}{a_d} \cdot (\varpi^k \partial)^n}_\beta \right)^i  \cdot a_d^{-1} .\]
Il a été vu que $\| \beta \|_{k , r} < 1$ et que $\| \alpha + \beta \|_{k , r} = \max\{ \| \alpha \|_{k , r} , \| \beta \|_{k , r}\}$.
On doit démontrer que $\| \alpha \|_{k ,r} < 1$.
Comme $\alpha$ ne contient que des puissances strictement négatives de la d\'erivation $\partial$, la norme $\| \cdot \|_{k ,r}$ est multiplicative pour les puissances de $\alpha$, ie $\|\alpha^i \|_{k , r} = |\alpha^i|_r = |\alpha|_r ^i$.
De m\^eme, $\| \beta^ i \|_{k , r} = | \beta |_k ^ i < 1$.
On suppose que $\| \alpha \|_{k , r} = 1$.
On montre alors que la somme définissant $P^{-1}$ ne converge pas dans $\F_{k , r}(U)$.
Comme $\| \alpha \|_{k , r} = |\alpha|_r = 1$,  $\| \alpha^i \|_{k , r} =|\alpha|_r ^i = 1$.
Puisque la norme $\| \cdot \|_{k , r}$ de $\F_{k , r}(U)$ est sous-multiplicative, on a $(\alpha + \beta)^ i = \alpha^ i + P_i$ avec $\| P_i \|_{k , r} < 1$.
On v\'erifie que la suite $\left\{ \sum_{i=0}^n (\alpha + \beta)^i \right\}_{n \in \N}$ n'est pas de Cauchy dans l'algèbre $\F_{k , r}(U)$ sous l'hypoth\`ese $\| \alpha \|_{k , r} = 1$.
En effet, la différence de deux termes consécutifs est donn\'ee par $\sum_{i=0}^{n+1} (\alpha^i + P_i) - \sum_{i=0}^n (\alpha^i + P_i)  = \alpha^{n+1} + P_{n +1}$.
Puisque $\| P_{n+1} \|_{k , r} < 1$ et $\| \alpha^{n+1} \|_{k , r} = 1$, on déduit de la proposition \ref{prop4.1.23} (la norme est le maximum des normes des coefficients) que $\| \alpha^{n +1} + P_{n+1} \|_{k , r} = 1$.
Ainsi, la somme des puissances de $\alpha + \beta$ diverge dans l'algèbre $\F_{k , r}(U)$. Lorsque $\| \alpha \|_{k , r} > 1$, on vérifie de m\^eme que la somme définissant $P^{-1}$ diverge dans l'algèbre de Banach $(\F_{k , r}(U) , \| \cdot \|_{k , r})$.
Il en résulte que si $P^{-1} \in \F_{k , r}(V)$, alors $\| \alpha \|_{k , r} < 1$. L'implication directe de la proposition est donc démontrée.
\end{proof}

\begin{remark}
Les points 1 et 3 de la proposition ne dependent pas de l'ouvert $V$ choisi.
Autrement dit, si $P = \sum_{n = 0}^\infty a_n \cdot (\varpi^k \partial)^n \in \Dkq(U)$ vérifie $\Nb(P) = \nb(P) = d$ et $|a_n| \cdot p^{(k-r)(d-n)} < |a_d|$ pour tout $n \in \{ 0 , \dots , d-1 \}$, alors $P$ est inversible dans $\F_{k , r}(V)$ pour tout ouvert $V \subset U$ ne contenant aucun z\'ero de $a_d$.
\end{remark}

\section{Microlocalisation de $\Di$}\label{section4}

On introduit maintenant le microlocalisé $\Fi(U)$ de $\Di(U)$ vérifiant la condition d'inversibilité souhaitée : un opérateur différentiel $P$ de $\Di(U)$ est inversible dans $\Fi(U)$ si et seulement si $P$ est un opérateur fini dont le coefficient dominant est inversible.

\subsection{Les microlocalisés $\Fir$}\label{section4.1}

On consid\`ere dans cette partie les faisceaux $\Fir := \varprojlim_{ k\geq r} \F_{k , r}$ pour $r \geq 1$ donn\'e puis on détermine les opérateurs différentiels de $\Di(U)$ inversibles dans $\Fir(U)$.
La limite projective est prise pour les morphismes de faisceaux de $K$-algèbres $\F_{k+1 ,r} \to \F_{k,r}$ induis localement par les inclusions $\varepsilon_{k,r} : \F_{k+1 , r}(U)  \hookrightarrow \Fkr$.
Contrairement à ce qu'il se passe pour $\F_{k , r}(U)$, les éléments de $\Di(U) \cap (\Fi(U))^\times$ seront des opérateurs finis dont les coefficients dominants devront toujours être inversibles.
Il faudra de plus imposer certaines inégalités entre les normes des coefficients des opérateurs inversibles.
Le niveau de congruence $k$ sera toujours suppos\'e sup\'erieur ou égal \`a $r \geq 1$.
La proposition suivante résume les principales propriétés du faisceau $\Fir$.

\begin{prop}
Les sections du faisceau $\Fir := \varprojlim_{ k\geq r} \F_{k , r}$ sur les ouverts $U$ de $\mathcal{U}$ sont des algèbres de Frechet-Stein dont la topologie est induite par les normes $\| \cdot\|_{k,r}$ des $K$-algèbres de Banach $\Fkr$. Par ailleurs,
\[ \Fir(U) =\displaystyle \left\{ P + \sum_{n = 1}^\infty  a_{-n} \cdot (\varpi^r\partial)^{-n} :  \, P \in \Di(U) , ~ a_{-n} \in \O_{\X, \Q} (U),~~  \lim_{n \to \infty} |a_{-n}|  =0 \right\} .\]
\end{prop}
\begin{proof}
Pour rappel, tout élément $S$ de $\Fkr$ s'écrit uniquement sous la forme
\[ S =  \sum_{n = 0}^\infty a_n \cdot (\varpi^k \partial)^n + \sum_{n = 1}^{\infty} a_{-n} \cdot (\varpi^r \partial)^{-n} ~~\mathrm{avec} ~~ a_n \underset{n \to \pm \infty}{\longrightarrow} 0 .\]
Les morphismes d'algèbres $\varepsilon_{k,r} : \F_{k+1 , r}(U)  \hookrightarrow \Fkr$ sont injectifs et vérifient pour toute section $S$ de $\F_{k+1 , r}(U) $ l'inégalité $\| \varepsilon_{k,r}(S) \|_{k,r} \leq \| S \|_{k+1 , r}$.
Par ailleurs, $\Fir(U) = \varprojlim_{k \geq r} \Fkr = \bigcap_{k \geq r} \Fkr$ est une $K$-algèbre contenant l'algèbre de Fréchet-Stein $\Di(U)$.
Cette inclusion correspond au morphisme continu $\Di(U) \hookrightarrow \F_{\infty , r}(U)$ induit par passage à la limite projective sur $k$ dans les diagrammes commutatifs de $K$-algèbres topologiques
\[ \xymatrix@R=3pc @C=3pc{ \widehat{\mathcal{D}}^{(0)}_{\X , k+1 , \Q}(U) \ar@{^{(}->}[r]\ar@{^{(}->}[d] & \Dkq(U) \ar@{^{(}->}[d] \\ \F_{k+1 , r}(U) \ar@{^{(}->}[r] & \F_{k , r}(U) \, .} \]
La proposition \ref{prop2.28} implique que $\Fir(U)$ est une algèbre de Fréchet-Stein.
La description des éléments de $\Fir(U)$ découle de tout cela.
\end{proof}

\subsubsection*{Opérateurs de $\Di(U)$ inversibles dans le microlocalisé $\Fir(U)$}

Soit $P = \sum_{n=0}^d a_n \cdot \partial^n$ un opérateur différentiel fini d'ordre $d$ de $\Di(U)$.
Pour $k$ suffisamment grand, $\Nb(P) = \nb(P) = d$ d'après le lemme \ref{lemmedeg}.
Ainsi, $P$ est inversible dans le localisé $\Ek$ si et seulement si son coefficient dominant $a_d$ est inversible dans la $K$-algèbre affinoïde $\O_{\X , \Q}(U)$.
On déduit de la proposition \ref{prop2.22} que $P$ est inversible dans le microlocalisé $\Fkr$ si de plus pour tout $n\in \{ 0 , \dots , d-1 \}$, $|a_n|  < |a_d| \cdot p^{r(d-n)}$.
Cette condition ne dépendant pas de $k$, on peut expliciter les opérateurs différentiels finis de $\Di(U)$ inversibles dans le microlocalisé $\Fir(U)$ : un opérateur fini $P = \sum_{n=0}^d a_n \cdot \partial^n$ d'ordre $d$ appartient à $\Fir(U)^\times$ si et seulement si
\begin{enumerate}
\item
$a_d \in \O_{\X , \Q}(U)^\times$,
\item
$\forall n\in \{ 0 , \dots , d-1 \}, ~~  |a_n|  < |a_d| \cdot p^{r(d-n)}$.
\end{enumerate}
Cette condition est vérifiée par exemple pour les opérateurs dominants : $|a_n | \leq |a_d|$ pour tout entier $n \in \{ 0 , \dots , d\} $.
Réciproquement, on va démontrer que tout opérateur différentiel de $\Di(U)$ inversible dans le microlocalisé $\Fir(U)$ est de cette forme.
Soit maintenant $P = \sum_{n =0}^\infty a_n \cdot \partial^n$ un opérateur différentiel infini de $\Di(U)$. Il n'est pas vrai en général qu'à partir d'un certain niveau de congruence $k$, l'égalité $\Nb(P) = \nb(P)$ soit vérifiée.
Autrement dit, $P$ n'est pas forcément inversible dans le localisé $\Ek$ pour $k$ suffisamment grand dès que son coefficient d'ordre $\Nb(P)$ est inversible.
\begin{example}\label{ex3.6}~
\begin{enumerate}
\item
Soit $P = \prod_{n \geq 1} (1 - \varpi^n \partial) \in \Di(U)$. Alors $\Nb(P) = k$ et $\nb(P) = k-1$ pour tout $k \geq 1$.
Autrement dit, $P$ n'est jamais inversible dans $\Ek$ dès que $k \geq 1$.
\item
Soit $P = \sum_{n \geq 0} \varpi^{n^2} \cdot \partial^n \in \Di(U)$. On a $\Nb(P) = \nb(P) = k/2$ si $k$ est paire. Sinon, $\Nb(P) = \frac{k+1}{2}$ et $\nb(P) = \frac{k-1}{2}$.
Ainsi, $P$ est inversible dans l'algèbre $\Ek$ si et seulement si le niveau de congruence $k$ est pair.
\end{enumerate}
\end{example}

Comme nous disposons de morphismes de transition injectifs $\F_{k+1 ,r}(U) \hookrightarrow \Fkr$, si $P$ est inversible dans le microlocalisé $\Fkr$, alors $P$ est inversible dans $\F_{\ell ,r}(U)$ pour tout entier $\ell \in \{ 0 , \dots , k \}$.
Ainsi, ou bien $P$ est inversible dans tous les microlocalisés $\Fkr$, ou bien $P$ n'est plus inversible dans les microlocalisés $\Fkr$ à partir d'un certain niveau de congruence $k$.
Puisque $\Fkr$ est une sous-algèbre de $\Ek$, les deux opérateurs différentiels de l'exemple \ref{ex3.6} ne sont pas inversibles dans le microlocalisé $\Fkr$ quelque soit le niveau de congruence $k$.
Pour r\'esumer, si $P =  \sum_{n\in \N} a_n \cdot  \partial^n$ est un opérateur de $\Di(U)$ inversible dans le microlocalisé $\Fir(U)$, alors $P$ est inversible dans toutes les algèbres $\Fkr$.
R\'eciproquement, si $P$ est inversible dans chaque alg\`ebre $\Fkr$, alors $P$ est dans $\Fir(U)^\times = \cap_{k \geq r} \Fkr^\times$.
On s'intéresse donc aux opérateurs différentiels $P = \sum_{n=0}^\infty a_n \cdot \partial^n$ de $\Di(U)$ vérifiant pour un niveau de congruence $k$ la condition
\begin{equation}\label{eq1}
 d_k = \Nb(f) = \nb(f) ~~ \mathrm{et}~~
\forall n\in \{ 0 , \dots , d_k-1 \}, ~~  |a_n|  < |a_{d_k}| \cdot p^{r(d_k-n)}
\end{equation}
On démontre dans ce qui suit que si $P$ vérifie cette condition pour une infinité d'entiers naturels $k$, alors $P$ est un opérateur différentiel fini.
Les éléments de $\Di(U)$ inversibles dans le microlocalisé $\Fir(U)$ seront donc des opérateurs différentiels finis.

\vspace{0.4cm}

On note $U = \Spf A'$ et $A = A' \otimes_\V K$ la $K$-agèbre affinoïde associ\'ee. La norme spectrale $| \cdot |$ sur $A$ est multiplicative. En effet, la courbe formelle $\X$ est connexe et lisse, donc $A$ est intègre.
On fixe une clôture algébrique $\overline{K}$ de $K$. Cette dernière est munie d'une valuation, extension de la valuation de $K$, que l'on note encore $v$.
Soit $P = \sum_{n\in \N} a_n \cdot \partial^n$ un opérateur différentiel de $\Di(U)$. Les coefficients $a_n$ de $P$ sont des éléments de $A$ vérifiant la condition de surconvergence : $\forall \mu \in \R, ~~ \lim_{n \to \infty} |a_n \cdot \varpi^{- \mu n}| = 0$.
Pour $\mu \in v(\overline{K})$, on pose
\begin{enumerate}
\item[$\bullet$]
$|P|_\mu := \max_{ \ell\in \N} |a_\ell \cdot \varpi^{-\mu\ell} |$,
\item[$\bullet$]
$\Nm(P) := \max \left\{ n \in \N : |a_n \cdot \varpi^{-\mu n} | = |P|_\mu \right\}$,
\item[$\bullet$]
$\nm(P) := \min \left\{ n \in \N : |a_n \cdot  \varpi^{-\mu n} | = |P|_\mu \right\}$.
\end{enumerate}
Lorsque $\mu = k \in \N$, on retrouve les fonctions définies précédemment.
La norme spectrale $| \cdot |$ est à valeurs dans $| \overline{K} |$.
Pour tout élément $a$ de l'algèbre affinoïde $\O_{\X , \Q}(U)$, on note $v(a)$ l'élément $\mu$ de $v(\overline{K})$ pour lequel $|a| = |\varpi|^{-\mu}$.

\begin{definition}
Soit $P = \sum_{n\in \N} a_n \cdot \partial^n \in \Di(U)$. On définit son polygone de Newton comme étant l'enveloppe convexe des points $A_n =(n , v(a_n))$ de $\R^2$.
\end{definition}

Si $P$ est un opérateur différentiel fini, alors son polygone de Newton n'a qu'un nombre fini de pentes.
Dans le cas contraire, la suite des pentes (rangées dans l'ordre des indices des coefficients) est strictement croissante et diverge vers $+\infty$.
La proposition suivante montre que les pentes du polygone de Newton sont étroitement liées aux fonctions $\overline{N}$ et $\overline{n}$ de $P$.

\begin{prop}\label{prop3.7}
Soit $P = \sum_{n\in \N} a_n \cdot \partial^n \in \Di(U)$ et $\mu \in v(\overline{K})$.
\begin{enumerate}
\item
Si $\mu$ n'est pas une pente du polygone de Newton de $P$, alors $\Nm(P) = \nm(P)$.
\item
Si $\mu$ est une pente du polygone de Newton de $P$, alors $\nm(P) < \Nm(P)$.
\end{enumerate}
\end{prop}
\begin{proof}
On suppose tout d'abord que $\mu$ n'est pas une pente du polygone de Newton de $P$. La suite $(|a_n| \cdot p ^{\mu n})_n$ converge vers zéro.
La condition $\Nm(P) = \nm(P)$ signifie exactement qu'elle admet un unique élément maximal. 
On suppose qu'il existe deux indices $i < j$ tels que
\begin{equation}\label{eq4.5}
|P|_\mu = \max_{n\in\N} \{ |a_n |  \cdot p ^{\mu n} \}= |a_i| \cdot p ^{\mu i} = |a_j| \cdot p ^{\mu j}
\end{equation}
On note $\alpha$ la valuation commune des éléments $a_i \cdot \varpi^{-\mu i}$ et $a_j \cdot \varpi^{-\mu j}$.
Les points $A_i$ et $A_j$ appartiennent à la même droite d'équation $Y = \mu \cdot X + \alpha$.
L'égalité \ref{eq4.5} signifie que $\alpha$ minimalise la valuation des termes $a_n \cdot \varpi^{-\mu n}$.
La minimalité de $\alpha$ implique que $[A_i , A_j]$ est un côté du polygone de Newton de $P$. Mais ce segment à pour pente $\mu$ !
Cela contredit l'hypothèse de départ. La suite $(|a_n| \cdot p ^{\mu n})_n$ a donc un unique élément maximal. Autrement dit, $\Nm(P) = \nm(P)$.
Soit maintenant $\mu$ une pente du polygone de Newton de $P$. Alors le polygone de Newton admet un segment $[ A_i , A_j]$ de pente $\mu$ pour deux indices $i$ et $j$ distincts.
Les points $A_i$ et $A_j$ appartiennent à une même droite d'équation $Y = \mu \cdot X + \beta$. Cela se traduit par $|a_i| \cdot p ^{\mu i} = |a_j| \cdot p ^{\mu j} = |P|_\mu$ au niveau des normes.
 En particulier, $\nm(P) < \Nm(P)$.
\end{proof}

Comme on peut le voir dans la proposition suivante, la condition \ref{eq1} pour un niveau de congruence $k$ donn\'e a d'importantes conséquences sur les valeurs possibles des pentes du polygone de Newton.

\begin{prop}\label{cor3.8}
Soit $P = \sum_{n\in \N} a_n \cdot \partial^n \in \Di(U)$. Si $P$ vérifie la condition \ref{eq1} pour un certain niveau de congruence $k \geq r$, alors le polygone de Newton de $P$ n'admet pas de pente appartenant au segment $[r , k]$.
\end{prop}
\begin{proof}
Soit $\mu$ une pente du polygone de Newton de $P$ appartenant à l'intervalle $[r , k]$.
On sait d'après la proposition \ref{prop3.7} qu'il existe deux indices distincts $n < m$ tels que $ |P|_\mu = |a_n| \cdot p^{\mu n} = |a_m| \cdot p^{\mu m}$.
On suppose $m$ maximal. L'hypothèse \ref{eq1} implique que $\Nb(P) = \nb(P)$. On note $d_k$ cet entier.
Le lemme \ref{lemmedeg} nous assure que $m \leq d_k$. Par maximalité de $m$, on a $|a_m| \cdot p^{\mu m} > |a_{d_k}| \cdot p^{\mu d_k}$. On en déduit que
\[ |a_n| = p^{\mu(m-n)} \cdot |a_m| > p^{\mu(m-n)} \cdot p^{\mu (d_k-m)} \cdot |a_{d_k}| = p^{\mu (d_k-n)} \cdot |a_{d_k}| . \]
On rappelle que $n < m \leq d_k$. Le fait que $\mu$ soit supérieur ou égal à $r$ contredit la seconde partie de l'hypothèse \ref{eq1} : $\forall n\in \{ 0 , \dots , d_k-1 \}, ~~  |a_n|  < |a_{d_k}| \cdot p^{r(d_k-n)}$.
Ainsi, aucune pente du polygone de Newton de $P$ n'appartient au segment $[r , k]$.
\end{proof}

On en déduit qu'un \'el\'ement $P$ de $\Di(U)$ est un opérateur différentiel fini dès que $P$ vérifie l'hypothèse \ref{eq1} pour une infinit\'e de niveaux de congruences $k \geq r$.

\begin{cor}\label{cor3.9}
Soit $P = \sum_{n\in \N} a_n \cdot \partial^n \in \Di(U)$. Si $P$ vérifie la condition \ref{eq1} pour tout entier naturel $k \geq r$, alors $P$ est un opérateur différentiel fini.
\end{cor}
\begin{proof}
On suppose que $P = \sum_{n=0}^\infty a_n \cdot \partial^n$ est un opérateur infini. Alors son polygone de Newton admet une infinité de pentes $\lambda_\ell$ et la suite $(\lambda_\ell)_\ell$ est strictement croissante de limite $+\infty$.
Mais la proposition \ref{cor3.8} implique que pour tout entier $k \geq r$, il n'y aucune pente dans le segment $[r , k]$.
Le polygone de Newton de $P$ ne peut donc pas admettre de pente dans l'intervalle $[r , + \infty [$.
La suite des pentes $(\lambda_\ell)_\ell$ est donc finie et $P$ est un opérateur fini.
\end{proof}

Soit $P = \sum_{n = 0}^\infty a_n \cdot \partial^n$ un opérateur diff\'erentiel de $\Di(U)$.
Si $P$ est inversible dans le microlocalisé $\Fir(U)$, alors $P$ est inversible dans $\F_{k,r}(U)$ pour tout entier $k \geq r$.
En particulier, $P$ vérifie la condition \ref{eq1} pour tout entier naturel $k$.
Le corollaire \ref{cor3.9} permet donc de décrire les éléments de $\Di(U)$ inversibles au voisinage de $x$ dans le microlocalisé $\Fir$.

\begin{theorem}\label{prop4.11}
Soit $P \in \Di(U)$ et $V \subset U$ un ouvert. Alors $P$ est inversible dans le microlocalisé $\Fir(V)$ si et seulement si
\begin{enumerate}
\item
$P = \sum_{n=0}^d a_n \cdot \partial^n$ dans $\Di(U)$,
\item
$a_d$ est inversible dans $\O_{\X , \Q}(V)$,
\item
$\forall n\in \{ 0 , \dots , d-1 \}, ~~  |a_n|  < |a_d| \cdot p^{r(d-n)}$.
\end{enumerate}
\end{theorem}

Le critère du théorème \ref{prop4.11} ne correspond pas exactement à la condition d'inversibilité souhaitée.
En effet, on désire que tout opérateur différentiel fini de $\Di(U)$ soit inversible dans le microlocalisé dés que son coefficient dominant est inversible.
La condition 3 du theorème \ref{prop4.11} contient cependant beaucoup d'opérateurs : les coefficients $a_i$ peuvent être bien plus grands que $a_d$ pour la norme spectrale.
On peut les autoriser à être de plus en plus grands en augmentant $r$.
En considérant tous les entiers $r \geq 1$, la condition 3 donne exactement tous les opérateurs finis dont le coefficient dominant est inversible.

\subsection{L'alg\`ebre microlocalisée $\Fi(U)$}\label{section4.2}

Soit toujours $U$ un ouvert affine de $\X$ sur lequel on dispose d'une coordonnée locale $t$.
Pour rappel, $\Fir(U)$ est une algèbre de Frechet-Stein dont tout élément $S$ s'écrit uniquement sous la forme $S = P + Q$ avec $P \in \Di(U)$ et $Q = \sum_{n = 1}^\infty a_{-n} \cdot (\varpi^r \partial)^{-n}$ vérifiant $|a_n| \to 0$ lorsque $n \to - \infty$.
Regardons les liens entre les algèbres $\F_{k,r}(U)$ pour un niveau de congruence $k$ fixé et pour $r$ variable.
Soit $r < r'$ deux entiers strictement positifs et $k \geq r'$.
Tout élément $S = P + Q$ de $\F_{k,r}(U)$, o\`u $P = \sum_{n = 0}^\infty a_n \cdot (\varpi^k \partial)^n$ et $Q = \sum_{n = 1}^{\infty} a_{-n} \cdot (\varpi^r \partial)^{-n}$, appartient à $\F_{k,r'}(U)$.
En effet, puisque $P \in \Dkq(U) \subset \F_{k,r'}(U)$, il suffit de vérifier que $Q \in \F_{k , r'}(U)$. On a
\[ Q = \sum_{n < 0} a_n \cdot (\varpi^r \partial)^n = \sum_{n < 0} a_n \varpi^{-(r'-r) n} \cdot (\varpi^{r'} \partial)^n . \]
Comme $|\varpi| ^{-(r'-r) n}  \underset{n \to - \infty}{\longrightarrow} 0$, on a bien $Q \in \F_{k , r'}(U)$.
Il est clair que $\|P\|_{k,r} = \| P \|_{k , r'} = \max_{n \in \N} \{ |a_n| \}$.
Par ailleurs, $\| Q \|_{k,r} = \max_{n<0} \{ |a_n\}$ et $\| Q\|_{k , r'} = \max_{ n<0} \left\{ |a_n| \cdot p^{-(r'-r)n} \right\}$.
Ainsi, il en découle que $\|Q\|_{k , r'} \leq \|Q\|_{k,r}$ et que $\|S\|_{k , r'} \leq \|S\|_{k,r}$. Autrement dit, l'inclusion d'algèbres $\F_{k ,r}(U) \hookrightarrow \F_{k,r'}(U)$ est 1-lipschtizienne. Elle est donc continue.
Cependant, la topologie de $\F_{k,r'}(U)$ induite sur $\F_{k ,r}(U)$ pour $r' > r$ par l'inclusion $\F_{k ,r}(U) \hookrightarrow \F_{k,r'}(U)$ ne coïncide pas avec la topologie de $\F_{k ,r}(U)$.
En effet, les normes $\| \cdot \|_{k , r}$ et $\| \cdot \|_{k , r'}$ ne définissent pas les mêmes séries convergentes pour les éléments de la forme $\sum_{n < 0} a_n \cdot (\varpi^r \partial)^n$.
En conclusion, $\F_{k ,r}(U)$ est une sous-algèbre de $\F_{k,r'}(U)$ dès que $k \geq r' \geq r$ et la topologie de $\F_{k ,r}(U)$ est plus fine que celle de $\F_{k ,r'}(U)$.
Les inclusions $\F_{k ,r}(U) \hookrightarrow \F_{k,r'}(U)$ commutent avec les morphismes de transition $\F_{k+1 , r}(U) \hookrightarrow \F_{k,r}(U)$ et le diagramme suivant est commutatif :
\[ \xymatrix@R=3pc @C=3pc{ \F_{k+1 , r}(U) \ar@{^{(}->}[r]\ar@{^{(}->}[d] & \F_{k , r}(U) \ar@{^{(}->}[d] \\ \F_{k+1 , r'}(U) \ar@{^{(}->}[r] & \F_{k , r'}(U) . } \]
En passant à la limite projective sur $k$ dans la catégorie des $K$-algèbres topologiques, on obtient un morphisme injectif continu de $K$-algèbres $\F_{\infty , r}(U) \hookrightarrow \F_{\infty , r'}(U)$.

\begin{definition}
On introduit la $K$-alg\`ebre $\displaystyle \Fi(U) := \varinjlim_{r \geq 1} \F_{\infty , r}(U) = \bigcup_{r \geq 1} \F_{\infty , r}(U)$.
\end{definition}

Par d\'efinition, si $P \in \Fi(U)$, alors $P \in \F_{\infty , r}(U)$ pour un certain $r \in \N^*$.
En particulier, $P$ s'écrit uniquement sous la forme $P = \sum_{n\in \Z}a_n \cdot \partial^n$ avec $|a_n| \cdot \lambda^n \underset{n \to \infty}{\longrightarrow} 0$ pour tout réel $\lambda > 0$ et $|a_n| \cdot p^{nr} \underset{n \to -\infty}{\longrightarrow} 0$.
La première condition correspond à $\sum_{n = 0}^\infty a_n \cdot \partial^n \in \Di(U)$.
On obtient la description suivante de l'algèbre $\Fi(U)$ :
\[ \Fi(U) = \left\{ P + \sum_{n = 1}^\infty a_{-n} \cdot \partial^{-n}, ~~ P \in \Di(U) ,~~  \exists R > 0 ~~ \mathrm{tq} ~~ |a_n| \cdot R^n \underset{n \to -\infty}{\longrightarrow} 0  \right\} .\]
On munit la $K$-algèbre $\Fi(U)$ de la topologie la plus fine pour laquelle les inclusions $\Fir(U) \hookrightarrow \Fi(U)$ sont continues.
Il s'agit de la topologie localement convexe limite inductive.
Puisque l'algèbre $\Fi(U)$ est l'union croissante des algèbres $\Fir(U)$, on a $\Fi(U)^\times = \bigcup_{r \geq 1} \F_{\infty , r}(U)^\times$.
Le résultat suivant se déduit du théorème \ref{prop4.11}. En effet, l'union sur $r \in \N^*$ des opérateurs finis de $\Di(U)$ vérifiant le point 3 de cet théorème est l'ensemble des opérateurs différentiels finis de $\Di(U)$.

\begin{theorem}
Soit $V \subset U$ un ouvert et $P \in \Di(U)$.
Alors $P$ est inversible dans le microlocalisé $\Fi(V)$ si et seulement si
\begin{enumerate}
\item
$P = \sum_{n=0}^d a_n \cdot \partial^n$ dans $\Di(U)$,
\item
$a_d$ est inversible dans $\O_{\X , \Q}(V)$.
\end{enumerate}
\end{theorem}

\section{Variété caractéristique des $\Di$-modules coadmissibles}\label{section5}

On note $\pi : T^*\X \to \X$ le fibré cotangent de $\X$ et $s : \X \to T^*\X$ sa section nulle.
La surjection canonique $\pi$ est un morphisme ouvert.
Le niveau de congruence $k$ est toujours supposé supérieur ou égal à un.
Les $K$-algèbres de Banach $\Dkq(U)$ sont alors quasi-abéliennes et l'on dispose des microlocalis\'es d\'efinis dans les deux sections pr\'ec\'edentes.

\subsection{$\Fkrs$-module cohérent associé à un $\Dkq$-module cohérent}\label{section5.1}

On fixe dans cette partie les entiers $k \geq r \geq 1$.

\subsubsection*{Les faisceaux $\Fkrs$}

On associe au faisceau $\F_{k , r}$ un faisceau $\Fkrs$ sur le fibré cotangent $T^*\X$ de $\X$ de la manière suivante.
On rappelle que pour tout ouvert $V$ de $T^*\X$, son image $\pi(V)$ dans $\X$ est un ouvert puisque le morphisme $\pi : T^*\X \to \X$ est ouvert.

\begin{definition}
Pour tout ouvert $V$ de $T^* \X$, on pose
\[
\Fkrs(V) := 
\begin{cases}
\F_{k , r}( \pi(V)) \hspace{0.83cm} \mathrm{si}~~V \cap s(\X) = \emptyset  \\
\Dkq(\pi(V)) ~~~ \mathrm{si}~~V \cap s(\X) \neq \emptyset .
\end{cases}
\]
\end{definition}

La $K$-algèbre $\Fkrs(V)$ est munie d'une norme de Banach notée $\| \cdot \|_{k , r}$.
Si $V \cap s(\X) = \emptyset$, il s'agit de la norme sous-multiplicative $\| \cdot \|_{k , r}$ de $\F_{k , r}(\pi(V))$.
Sinon pour $V \cap s(\X) \neq \emptyset$, cette dernière est simplement la norme multiplicative $| \cdot |_k$ de $\Dkq(\pi(V))$.
Il est clair que $\Fkrs$ est un préfaisceau sur $T^*\X$. En effet, soient $W \subset V$ deux ouverts de $T^*\X$.
Si $V$ et $W$ n'intersectent pas la section nulle $s(\X)$, ou s'ils l'intersectent tous les deux, alors les morphismes de restrictions sont respectivement ceux du faisceau $\F_{k , r}$ ou ceux de $\Dkq$.
On suppose maintenant que $W \cap s(\X) = \emptyset$ mais que $V \cap s(\X) \neq \emptyset$.
Par définition,
\[ \Fkrs(W) = \F_{k , r}( \pi(W))~~ \mathrm{et} ~~ \Fkrs(V) = \Dkq(\pi(V)) . \]
L'inclusion $W \subset V$ implique que $\pi(W) \subset \pi(V)$.
Le morphisme de restriction d'algèbres $\Fkrs(V) \to \Fkrs(W)$ est simplement le morphisme composé
\[ \Dkq(\pi(V)) \to \Dkq(\pi(W )) \hookrightarrow \F_{k , r}(\pi(W)) . \]

\begin{lemma}
Le préfaisceau $\F_{k , r}$ est un faisceau sur le fibré cotangent $T^*\X$.
\end{lemma}
\begin{proof}
Puisque $\F_{k , r}$ et $\Dkq$ sont des faisceaux, il suffit en pratique de vérifier la propriété suivante.
Soit $V$ un ouvert du fibré cotangent $T^*\X$ et $(U_i)$ un recouvrement ouvert de $V$.
Si $S$ est une section de $\Fkrs(V)$ telle que $S_{|U_i} \in \Dkq(\pi(U_i))$ pour un certain $i$, alors $S_{|U_j} \in \Dkq(\pi(U_j))$ pour tout $j$.
Cette propriété implique que les conditions d'unicité et de recollement de $\Fkrs$ sont données exactement par celles des faisceaux $\Dkq$ et $\F_{k , r}$.
Quitte à réduire les ouverts $U_i$, on peut supposer les ouverts $\pi(U_j)$ affines munis de coordonnées locales.
On peut aussi écrire $S_{|U_j}$ sous la forme
\[ S_{|U_j} = P + \sum_{ n =1}^\infty a_{-n} \cdot (\varpi^r \partial)^{-n} \in \F_{k , r}(\pi(U_j)), ~~ P \in \Dkq(\pi(U_j)) . \]
En effet, $S_{|U_j}$ est par définition de $\Fkrs$ une section de $\Dkq(\pi(U_j))$ ou de $\F_{ k, r}(\pi(U_j))$ avec $\Dkq(\pi(U_j)) \subset \F_{ k, r}(\pi(U_j))$.
Comme les fibres de $\pi : T^*\X \to \X$ sont irréductibles, on a $\pi(U_i \cap U_j) = \pi(U_i) \cap \pi(U_j)$ et la restriction $S_{| U_i \cap U_j}$ coincide avec les restrictions respectives de $S_{|U_i} \in \Dkq(\pi(U_i))$ et de $S_{|U_j} \in \F_{k , r}(\pi(U_j))$ par les morphismes de restriction $\pi(U_i) \to \pi(U_i) \cap \pi(U_j)$ et $\pi(U_j) \to \pi(U_i) \cap \pi(U_j)$.
Il en découle que
\[ S_{| U_i \cap U_j} = P + \sum_{ n =1}^\infty a_{-n} \cdot (\varpi^r \partial)^{-n} \in \Dkq(\pi(U_i) \cap \pi(U_j)) . \]
Par unicité de l'écriture d'un élément de $\F_{k , r}(\pi(U_i) \cap \pi(U_j))$ sous la forme d'une série en les puissances de la dérivation $\partial$, on en déduit que $a_{-n} = 0$ pour tout $n \in \N^*$.
Autrement dit, $S_{|U_j} \in \Dkq(\pi(U_j))$.
\end{proof}

\subsubsection*{Les modules $\Mkrs$}

Soit $\pi^{-1} (\Dkq) \to \Fkrs$ le morphisme de faisceaux de $K$-algèbres donn\'e sur les ouverts $V$ de $T^*\X$ par les inclusions d'algèbres
\[
\pi^{-1} (\Dkq)(V) = \Dkq(\pi(V)) \hookrightarrow \Fkrs(V) = 
\begin{cases}
\F_{k , r}( \pi(V)) \hspace{0.82cm} \mathrm{si}~~V \cap s(\X) = \emptyset  \\
\Dkq(\pi(V)) ~~~ \mathrm{si}~~V \cap s(\X) \neq \emptyset .
\end{cases}
\]
Ce morphisme induit une structure de $\pi^{-1} (\Dkq)$-module pour le faisceau $\Fkrs$.

\begin{definition}
Soit $k \in \N^*$ et $r \in \{ 1 , \dots , k \}$.
On associe à tout $\Dkq$-module cohérent $\M_k$ les modules cohérents $\Mkr := \F_{k , r} \otimes_{\Dkq} \M_k$ et $\Mkrs := \Fkrs \otimes_{\pi^{-1} (\Dkq)} \pi^{-1} (\M_k)$.
\end{definition}

Soit $(x , \xi)$ un point du fibré cotangent $T^*\X$ ; il appartient \`a la section nulle de $T^*\X$ si et seulement si $\xi = 0$.
On note $\supp (\Mkrs)$ le support du module $\Mkrs$ en tant que faisceau sur le fibré cotangeant $T^*\X$.
On dispose des isomorphismes suivants :
\begin{align*}
(\Mkrs)_{(x , \xi)} & \simeq (\F_{k , r})_{(x , \xi)} \otimes_{(\Dkq)_{(x , \xi)}} (\M_k)_{(x , \xi)} \\ 
& \simeq \begin{cases}
(\Mkr)_x \simeq (\F_{k , r})_x \otimes_{(\Dkq)_x} (\M_k)_x  \hspace{1.2cm} \mathrm{si}~~\xi \neq 0  \\
(\M_k)_x \simeq (\Dkq)_x \otimes_{(\Dkq)_x} (\M_k)_x \hspace{1cm} \mathrm{si}~~\xi = 0 \, .
\end{cases}
\end{align*}

\begin{lemma}
Le support $\supp (\Mkrs)$ est une partie fermée de $T^*\X$.
\end{lemma}
\begin{proof}
On rappelle que le fibré cotangent $T^*\X$ est quasi-compact et que le $\Fkrs$-module $\Mkrs$ est cohérent.
Soit $V_1 , \dots , V_n$ un recouvrement ouvert affine fini de $T^*\X$ tel que ${\Mkrs}_{|V_i}$ soit un ${\Fkrs}_{|V_i}$-module de type fini pour tout entier $i \in \{ 1 , \dots , n \}$.
Il existe des sections $e_1 , \dots , e_s$ de $\Mkrs(V_i)$ telles que ${\Mkrs}_{|V_i} = {\Fkrs}_{|V_i} \cdot e_1 + \dots + {\Fkrs}_{|V_i} \cdot e_s$.
Alors $\supp(\Mkrs) \cap V_i = \bigcup_{j=1}^s \supp(e_j)$.
En effet, il est clair que $\cup_{j=1}^s \supp(e_j) \subset \supp(\Mkrs) \cap V_i$.
De plus, si $(e_1)_x = \dots = (e_s)_x = 0$ pour un certain $x\in V_i$, alors $e_1 = \dots = e_s = 0$ sur un voisinage de $x$ dans $V_i$ et toute combinaison linéaire des $e_i$ est nulle sur ce voisinage.
L'autre inclusion des supports en découle.
Comme le support de toute section de $\Mkrs$ est fermé d'après \cite[\href{https://stacks.math.columbia.edu/tag/01AU}{Tag 01AU}]{stacks-project}, $\supp(\Mkrs) \cap V_i $ est fermé car union finie de fermés.
Enfin, puisque les $V_i$ recouvrent $T^*\X$, le support $\supp (\Mkrs)$ est une partie fermée de $T^*\X$.
\end{proof}

\subsubsection*{Lien entre $\Car(\M_k)$ et $\supp( \Mkrs)$}

Soit $\M_k = \Dkq / \I_k$ avec $\I_k$ un idéal cohérent de $\Dkq$.
On note $\I_k^\circ$ l'ensemble des éléments de $\I_k$ de norme $| \cdot |_k$ inférieure ou égale à un : $\I_k^\circ(U) := \{ P \in \I_k(U) : |P|_k \leq 1 \}$.
C'est un idéal cohérent de $\Dk$ puisque les algèbres $\Dk(U)$ sont noetheriennes et le faisceau $\Dk$ est cohérent.
Comme $\I_k^\circ \subset \I_k$ et $\M_k = \Dkq / \I_k$ est sans $\varpi$-torsion, le module $\Dk / \I_k^\circ$ est aussi sans $\varpi$-torsion ; c'est donc un modèle entier de $\M_k$.
Identifions $U$ à un ouvert affine de la fibre spéciale $X$ de $\X$ et munissons $T^*U$ du système de coordonnées locales $(x , \xi)$ associé à la coordonnée étale de $U$.
\`A tout opérateur $P = \sum_{n = 0}^\infty a_n \cdot (\varpi^k \partial)^n \in \Dk(U)$ est associ\'e un élément du fibré cotangent $T^*U$ donné par $\sigma(P) = \sigma_k(P) = \bar{a}_{\Nb(P)} \cdot \xi^{\Nb(P)}$
où $\bar{a}_{\Nb(P)} \in \O_X(U)$ est la réduction modulo $\varpi$ de $a_{\Nb(P)}$. On a $\sigma(P) = \sigma(\bar{P})$, où $\sigma(\bar{P})$ est le symbole principal de $\bar{P} \in (\I_k^\circ \otimes_\V \kappa)(U)$ d\'efini dans la partie \ref{section2.1}.
On rappelle que
\[ \Car (\M_k) \cap T^* U= \left\{ ( x , \xi) \in T^* U : ~ \sigma(P)(x,\xi) = 0 ~~~ \forall P \in \I_k^\circ(U) \right\}. \]

\begin{lemma}\label{propideal}
Soit $\I$ un idéal cohérent de $\Dkq(U)$.
Si $\Ek \cdot \I = \E_k(U)$, alors $\I \cap \Ek^\times \neq \emptyset$.
\end{lemma}
\begin{proof}
On reprend les notations de la proposition \ref{propnoeth} vis à vis des gradués.
Pour tout entier $m \in \Z$, on pose $\I(m) = \Dkq(U)(m) \cap \I$. On définit ainsi une filtration croissante exhaustive de l'idéal $\I$.
Le gradué associé $\gr \I$ est un idéal de $\gr \Dkq(U)$.
On note $\gamma : \I \to \gr \I$ la restriction à $\I$ de l'application $\gamma : \Dkq(U) \to \gr \Dkq(U)$.
On observe alors que $\gr \I = \gr (\Dkq(U)) \cdot \gamma(\I)$. De même, $\gr (\Ek) \cdot \I) = \gr (\Ek) \cdot \gamma(\I)$ pour la filtration $(\Ek) \cdot \I)(m) = \Ek(m) \cap (\Ek) \cdot \I)$.
On rappelle que $\gr \Dkq(U) \simeq \gr(\O_{\X , \Q}(U)) [ \zeta_k ]$ et que $\gr \Ek \simeq \gr(\O_{\X , \Q}(U)) [ \zeta_k , \zeta_k^{-1}]$.
L'hypothèse $\Ek \cdot \I = \E_k(U)$ implique $\gr (\Ek\cdot \I) = \gr (\Ek) \cdot \gamma(\I) = \gr \Ek$.
Autrement dit, il existe des opérateurs $P_1 , \dots , P_s \in \I$ et $Q_1 , \dots , Q_s \in \gr \Ek$ tels que $Q_1 \cdot \gamma(P_1) + \dots + Q_s \cdot \gamma(P_s) = 1$.
Puisque $\gr \Ek \simeq \gr(\O_{\X , \Q}(U)) [ \zeta_k , \zeta_k^{-1}]$, on peut trouver un entier naturel $n$ pour lequel
\[ \zeta_k^n \cdot Q_1 , \dots , \zeta_k^n \cdot Q_s \in \gr \Dkq(U) \simeq \gr(\O_{\X , \Q}(U)) [ \zeta_k ] . \]
On note $R_i = \zeta_k^n \cdot Q_i \in \gr \Dkq(U)$. Alors
\[ \zeta_k^n = R_1 \cdot \gamma(P_1) + \dots + R_s \cdot \gamma(P_s) \in \gr \I = \gr (\Dkq(U)) \cdot \gamma(\I) . \]
Puisque $\zeta_k^n$ est un monôme de norme un, $\zeta_k^n \in \gr_0 \I = \I(0) / \I(-1)$.
En particulier, $\zeta_k^n$ se relève en un élément $P$ de $\I$ vérifiant $\gamma(P) = \zeta_k^n$.
Cela signifie que $\Nb(P) = n_k(P) = n$ et que le coefficient d'indice $n$ de $P$ est inversible.
D'après la proposition \ref{prop2.11}, $P$ est inversible dans $\Ek$.
\end{proof}

\begin{prop}\label{prop5.5}
Pour tout $\Dkq$-module cohérent $\M_k$ et pour tout entier $r \in \{1 , \dots , k \}$, on dispose d'une inclusion d'espaces topologiques $\Car (\M_k) \subset \supp (\Mkrs)$.
\end{prop}
\begin{proof}
On peut supposer la courbe formelle $\X$ est affine munie d'une coordonnée locale et se ramener au cas où $\M_k = \Dkq / \I$ pour un idéal cohérent $\I$ de $\Dkq$.
Alors $\tilde{\M}_{ k , r} = \F_{k , r} \otimes_{\Dkq} \M_k \simeq \F_{k , r} / \F_{k , r} \cdot \I$, où $\F_{k , r} \cdot \I$ est l'idéal à gauche de $\F_{k , r}$ engendré par $\I$.
On commence par établir que $\Car (\M_k)$ et $\supp (\Mkrs)$ coïncident au niveau de la section nulle du fibré cotangent $T^*\X$.
On rappelle que $(\Mkrs)_{( x , 0)} = (\M_k)_x$ et $(\Mkrs)_{( x , \xi)} = (\tilde{\M}_{k , r})_x$ pour $\xi \neq 0$ avec $\tilde{\M}_{k , r} = \F_{k , r} \otimes_{\Dkq} \M_k$.
L'idéal cohérent $\I^\circ$ de $\Dk$ est constitué des éléments de $\I$ de norme $| \cdot |_k$ inférieure où égale à un.
La réduction $I = \I^\circ \otimes_\V \kappa$ modulo $\varpi$ de $\I^\circ$ est un idéal cohérent de $\Dks = \Dk \otimes_\V \kappa$.
Soit $P_1 , \dots , P_s$ des éléments de l'idéal $\I^\circ(U)$ tels que $\sigma(P_1) , \dots , \sigma(P_s)$ engendrent le gradué $\gr I$.
Alors $P_1 , \dots , P_s$ engendrent l'idéal $\I$ et
\begin{align*}
 \Car (\M_k) & = \Car (\Dks / I) = \left\{ (x , \xi) \in T^*X : ~ \sigma(P)(
 x , \xi) =  0 ~~\forall P \in I \right\} \\
 & = \left\{ (x , \xi) \in T^*X : ~ \sigma(P_1)(x , \xi) = \dots = \sigma(P_s)(x , \xi) = 0  \right\} .
\end{align*}
On écrit $\sigma(P_i) = \sigma(\bar{P}_i) = a_i \cdot \xi^{d_i} \in \gr I(X)$. La section nulle est une composante irréductible de la variété caractéristique $\Car (\M_k)$ si et seulement si les ordres $d_1 , \dots , d_s$ des opérateurs $P_1 , \dots , P_s$ sont tous supérieurs ou égaux à un.
Puisque l'idéal $\gr I$ est engendré par les monômes $\sigma(P_1) , \dots , \sigma(P_s)$, l'idéal $\gr I$ ne contient pas de terme d'ordre zéro en $\xi$.
On en déduit que l'idéal $\I$ ne contient pas d'opérateur $P$ vérifiant $\Nb(P) = 0$.
Dans le cas contraire, $\sigma(P)$ serait un élément d'ordre $\Nb(P) = 0$ en $\xi$ de l'idéal $\gr I$.
Les points du support $\supp (\Mkrs)$ interceptant la section nulle sont donnés par le support de $\M_k$.
Puisque tout élément $P$ de l'idéal $\I(\X)$ vérifie $\Nb(P) \geq 1$, aucun élément de $\I$ n'est inversible dans $(\Dkq)_x$ pour chaque point $x$ de $\X$.
Ainsi, $(\Mkrs)_{(x, 0)} = (\Dkq / \I)_x \neq 0$ en tout point $x$ de $\X$. Autrement dit, $\supp (\M_k) = \X$ et le support $\supp (\Mkrs)$ contient la section nulle du fibré cotangent.
Réciproquement, le support du module $\Dkq / \I$ n'est pas $\X$ tout entier si et seulement si l'idéal $\I(\X)$ contient un opérateur $P$ vérifiant $\Nb(P) = 0$.
Dans ce cas, le support du module $\Dkq / \I$ consiste en un nombre fini de points contenus dans l'ensemble des zéros du coefficient constant $a_0$ de $P$.
Par ailleurs, $\sigma(P) = \bar{a}_0$ est un élément du gradué $\gr I(U)$. La variété caractéristique $\Car(\Dks / I)$ ne contient donc pas la section nulle.
De m\^eme, on vérifie, que la variété caractéristique $\Car (\M_k)$ intercepte la section nulle exactement en les points du support du module $\Dkq / \I$, autrement dit parmi les zéros du coefficient constant $a_0$ de $P$.

Regardons maintenant ce qu'il se passe en dehors de la section nulle et fixons un point $(x , \xi) \in \Car(\M_k)$ situé en dehors de cette dernière.
On rappelle qu'un point $(x , \xi) \in T^*\X$ appartient à la variété caractéristique $\Car(\M_k)$ si et seulement si pour tout opérateur $P$ de $\I(\X)$ de norme un, $\sigma(P)( x , \xi) = 0$.
On en déduit qu'aucun élément $P$ de l'idéal $\I(\X)$ n'est inversible dans $(\E_k)_x$.
En effet, le coefficient dominant $a$ de $P$ s'annule en $x$ puisque le point $(x , \xi)$ n'appartient pas à la section nulle du fibré cotangent : $\xi \neq 0$ et $\sigma(P) = \bar{a}(x) \cdot \xi^{\Nb(P)} = 0$ implique $a(x) = 0$.
D'après le lemme \ref{propideal}, l'idéal $\E_k \cdot \I$ de $\E_k$ engendré par l'idéal $\I$ ne contient pas d'élément inversible au voisinage de $x$.
Autrement dit, $\E_k(W) \cdot \I(W) \subsetneq \E_k(W)$ pour tout ouvert affine $W$ de $\X$.
Il en découle que l'idéal $\F_{k , r} \cdot \I$ de $\F_{k , r}$ ne contient aucun élément inversible dans $(\F_{k , r})_x$.
Dans le cas contraire, il existerait un voisinage ouvert $W$ de $x$ pour lequel $\F_{k , r}(W) \cdot \I(W) = \F_{k , r}(W)$. Alors
\[ \E_k(W) \cdot \F_{k , r}(W) \cdot \I(W) = \E_k(W) \cdot \I(W) = \E_k(W) \cdot \F_{k , r}(W) = \E_k(W) .\]
Ainsi, $\E_k(W) \cdot \I(W) = \E_k(W)$ et $\I(W)$ contiendrait un opérateur inversible au voisinage de $x$ d'après le lemme \ref{propideal}.
Il en découle qu'aucun élément de l'idéal $\F_{k , r} \cdot \I$ n'est inversible dans le microlocalisé $(\F_{k , r})_x$.
On en déduit que $(\F_{k , r} / \F_{k , r} \cdot \I)_x \neq 0$ et que $( x , \xi) \in \supp (\Mkrs)$.
On a donc démontré l'inclusion $\Car (\M_k) \subset \supp (\Mkrs)$ en dehors de la section nulle.
\end{proof}

\begin{remark}
Pour un $\Dkq$-module cohérent $\M_k$ et $r \in \{ 1 , \dots , k\}$, nous venons de démontrer que la variété caractéristique $\Car(\M_k)$ et le support $\supp(\Mkrs)$ coincident au niveau de la section nulle.
Cependant, l'inclusion $\Car(\M_k) \subset \supp(\Mkrs)$ peut être stricte en dehors de la section nulle.
Cela provient du fait qu'il est plus \og difficile \fg \, d'être inversible dans le microlocalisé $\F_{k, r}$ que d'inverser les symboles principaux des opérateurs.
Pour exemple, soit $P$ un opérateur différentiel infini de $\Di(U)$. Les variétés caractéristiques $\Car(\Dkq / P)$ sont toutes de dimension un, mais $\supp(\Mkrs) = T^*\X$ pour $r$ et $k$ suffisamment grands d'après \ref{ex6.3}.
\end{remark}

\subsubsection*{Inégalité de Bernstein}

Nous disposons de l'inégalité de Bernstein pour les $\Dkq$-modules cohérents démontrée dans \cite{hallopeau1}, proposition 3.13 : un $\Dkq$-module cohérent $\M_k$ est non nul si et seulement si $\dim \Car (\M_k) \geq 1$.
Il d\'ecoule de la proposition \ref{prop5.5} que si $\dim (\supp \Mkrs) \leq 1$ pour un certain $r \in \{1 , \dots , k \}$, alors $\M_k$ est un $\Dkq$-module holonome.
Par ailleurs, cette proposition implique aussi que si $\dim (\Car\M_k) \geq 1$, alors $\dim (\supp \Mkrs) \geq 1$.
Pour r\'esumer, un $\Dkq$-module cohérent $\M_k$ est non nul si et seulement si pour un certain $r$, $\dim(\supp \Mkrs) \geq 1$. Dans ce cas, cette inégalité est vérifiée pour tout entier $r \in \{ 1 , \dots , k \}$.
Puisque la courbe formelle $\X$ est connexe et lisse, elle est irréductible et le fibré cotangent $T^*\X$ est aussi irréductible.
Comme $\dim(T^*\X) =2$, si le support du module $\Mkrs$ est de dimension deux,
alors $\supp (\Mkrs) = T^*\X$.
Lorsque $\dim (\supp \Mkrs) = 1$, la proposition suivante montre que le support du module $\Mkrs$ est une partie conique de $T^*\X$ constituée de composantes irréductibles verticales et potentiellement d'une composante irréductible horizontale d'équation $\xi = 0$.

\begin{prop}[Inégalité de Bernstein]\label{prop5.7}
Soit $\M_k$ un $\Dkq$-module cohérent non nul. Alors $\supp (\Mkrs)$ est égal  $T^*\X$ ou bien équidimensionnel de dimension 1.
\end{prop}
\begin{proof}
On peut supposer la courbe formelle $\X$ affine munie d'une coordonnée locale.
Soit $\M_k$ un $\Dkq$-module cohérent tel que $\dim (\supp \Mkrs) = 1$.
Par la proposition \ref{prop3.22}, $\M_k$ est holonome et $\M_k = \Dkq / \I_k$ pour un idéal cohérent $\I_k$ de $\Dkq$.
On a $\Mkr = \F_{k , r} / \F_{k , r} \cdot \I_k$, où $\F_{k , r} \cdot \I_k$ est l'idéal de $\F_{k , r}$ engendré par $\I_k$.
Il a été vu dans la proposition \ref{prop5.5} que le support $\supp (\Mkrs)$ et la variété caractéristique $\Car(\M_k)$ coincident au niveau de la section nulle et que $\Car(\M_k) \subset \supp (\Mkrs)$.
Puisque les composantes irréductibles de $\Car(\M_k)$ sont toutes de dimension un, le support $\supp (\Mkrs)$ n'a pas de point isolé sur la section nulle.
Il suffit d'élucider ce qu'il se passe en dehors de la section nulle.
Soit $(x , \xi) \in T^* \X  \backslash s(\X)$. On rappelle que $(\Mkrs)_{(x , \xi)} = (\Mkr)_x$.
On en déduit que si $(\Mkr)_x \neq 0$, alors le support du module $\Mkrs$ contient la droite verticale passant par $x$ sauf potentiellement pour le point situé sur la section nulle.
Si le support du module $\Mkrs$ contient la section nulle, il n'y a rien à montrer.
Dans ce cas, le support $\supp (\Mkrs)$ consiste en une composante irréductible horizontale d'équation $\xi = 0$ et en un nombre fini de composantes verticales.
On suppose maintenant que la section nulle n'est pas une composante irréductible du support $\supp (\Mkrs)$.
Il faut montrer que si $(\Mkrs)_{(x , \xi)} = (\Mkr)_x \neq 0$ pour tout $\xi \neq 0$, alors $(\Mkrs)_{(x , 0)} = (\M_k)_x \neq 0$.
Autrement dit, il faut prouver qu'aucun opérateur $P$ de l'idéal $\I_k(U)$, pour $U$ un voisinage ouvert affine de $x$ dans $\X$, n'est inversible dans $(\Dkq)_x$.
Si $\Nb(P) \geq 1$, alors la proposition \ref{prop2.2.3} montre que $P$ n'est pas inversible dans $(\Dkq)_x$.
Néanmoins, il a été vu dans la preuve de la proposition \ref{prop5.5} que l'idéal $\I_k(U)$ contient des opérateurs différentiels $P$ vérifiants $\Nb(P) = 0$ lorsque la section nulle n'est pas une composante irréductible de la variété caractéristique $\Car(\Dkq / \I_k)$. Soit $P$ un tel élément et $(x , \xi)$ un point du fibré cotangent tel que $\xi \neq 0$.
L'hypothèse $(\Mkrs)_{(x , \xi)} = (\Mkr)_x \neq 0$ implique que le coefficient constant $a_0$ de $P$ s'annule en $x$.
Autrement, $P$ serait inversible dans le microlocalisé $(\F_{k , r})_x$ d'après la proposition \ref{prop2.22}.
En effet, $a_0$ serait inversible au voisinage de $x$ puisque ne s'annulant pas en $x$.
Comme $\Nb(P) = 0$, ceci impliquerait que $(\Mkrs)_{(x , \xi)} = (\Mkr)_x = 0$.
Ainsi, aucun opérateur différentiel $P$ de l'idéal $\I_k(U)$ vérifiant $\Nb(P) = 0$ n'est inversible dans $(\Dkq)_x$ lorsque $(\Mkrs)_{(x , \xi)} = (\Mkr)_x \neq 0$.
On en déduit que $(\M_k)_x \neq 0$. 
Le support du module $\Mkrs$ contient donc la droite verticale d'abscisse $x$.
\end{proof}

\subsection{$\Firs$-module coadmissible associé à un $\Di$-module coadmissible}\label{section5.2}

On associe maintenant à tout $\Di$-module coadmissible $\M$ un $\Firs$-module coadmissible $\Mrs$ limite projective sur $k$ des $\Fkrs$-modules cohérents $\Mkrs$, où $\Firs$ est la limite des faisceaux de $K$-algèbres $\Fkrs$.
L'inégalité de Bernstein pour le support $\supp (\Mrs)$ découle directement des inégalités de Bernstein pour les $\Fkrs$-modules cohérents $\Mkrs$.

\subsubsection*{Les faisceaux $\Firs$}

On considère les morphismes de transition $\F_{k+1 , r}^* \to \Fkrs$ induis localement sur les ouverts $V$ de $T^*\X$ par les inclusions $\F_{k+1 , r}( \pi(V))  \hookrightarrow \F_{k , r}( \pi(V))$ lorsque $V \cap s(\X) = \emptyset$ et par $\widehat{\mathcal{D}}^{(0)}_{\X, k+1, \Q}(\pi(V)) \hookrightarrow \Dkq(\pi(V))$ dans le cas où $V \cap s(\X) \neq \emptyset$.
On déduit de la proposition \ref{prop2.28} et de \cite[corollaire 2.2.20]{huyghe} que ces morphismes sont plats à gauche et à droite.
On vérifie alors que $\Firs := \varprojlim_{k \geq r} \Fkrs$ est un faisceau sur le fibré cotangent $T^*\X$ défini localement par les algèbres de Fréchet-Stein $\Fir(U)$ et $\Di(U)$ :
\[
\Firs(V) = 
\begin{cases}
\Fir( \pi(V)) \hspace{0.515cm} \mathrm{si}~~V \cap s(\X) = \emptyset  \\
\Di(\pi(V)) ~~~ \mathrm{si}~~V \cap s(\X) \neq \emptyset .
\end{cases}
\]

\subsubsection*{Les $\Fir$-modules $\Mr$}

On se donne  un $\Di$-module coadmissible $\M = \varprojlim_k \M_k$. Pour tous entiers $k \geq r \geq 1$, rappelons que $\Mkr$ est le $\F_{k,r}$-module cohérent $\Mkr  = \F_{k,r} \otimes_{\Dkq} \M_k$.
On note $g_k$ le morphisme de transition $\M_{k+1} \to \M_k$ et $i_{k,r} : \F_{k+1,r} \to \F_{k,r}$ le morphisme donné localement par l'inclusion de $\F_{k+1 , r}(U)$ dans $\F_{k,r}(U)$.
On considère le morphisme de transition $\F_{k+1 , r}$-linéaire $g_{k,r}=i_{k,r} \otimes g_k : \tilde{\M}_{k+1 , r} \to \Mkr$. On pose $\Mr := \varprojlim_{k \geq r} \Mkr$.

\begin{prop}\label{prop4.3}
On a $\Mkr \simeq \F_{k , r} \otimes_{\F_{k+1 , r}} \tilde{\M}_{k+1 , r}$. Autrement dit, $\Mr$ est un $\Fir$-module coadmissible.
\end{prop}
\begin{proof}
L'algèbres $\Fkr$ est naturellement un $\widehat{\mathcal{D}}^{(0)}_{\X, k+1 , \Q}(U)$-module via la composition des morphismes d'algèbres $\widehat{\mathcal{D}}^{(0)}_{\X, k+1 , \Q}(U) \hookrightarrow \Dku \hookrightarrow \Fkr$.
Ainsi, $\F_{k , r}$ est un $\widehat{\mathcal{D}}^{(0)}_{\X, k+1 , \Q}$-module.
On dispose de l'isomorphisme suivant de $\widehat{\mathcal{D}}^{(0)}_{\X, k+1 , \Q}$-modules :
\begin{align*}
\F_{k,r} \otimes_{\F_{k+1 , r}} \tilde{\M}_{k+1 , r}  & = \F_{k,r} \otimes_{\F_{k+1 , r}} \F_{k+1 , r} \otimes_{\widehat{\mathcal{D}}^{(0)}_{\X, k+1 , \Q}} \M_{k+1} \\
& \simeq \F_{k,r} \otimes_{\widehat{\mathcal{D}}^{(0)}_{\X, k+1 , \Q}} \M_{k+1} .
\end{align*}
La structure de $\F_{k , r}$-modules prolonge celle de $\widehat{\mathcal{D}}^{(0)}_{\X, k+1 , \Q}$-modules. Cet isomorphisme est donc un isomorphisme de $\F_{k,r}$-modules.
Comme $\Dkq \otimes_{\widehat{\mathcal{D}}^{(0)}_{\X, k+1 , \Q}} \M_{k+1} \simeq \M_k$, on obtient
\begin{align*}
\F_{k,r} \otimes_{\widehat{\mathcal{D}}^{(0)}_{\X, k+1 , \Q}} \M_{k+1} & \simeq  \F_{k,r} \otimes_{\Dkq}   \Dkq \otimes_{\widehat{\mathcal{D}}^{(0)}_{\X, k+1 , \Q}} \M_{k+1} \\
& \simeq \F_{k,r} \otimes_{\Dkq} \M_k = \Mkr
\end{align*}
en tant que $\Dkq$-modules. Puisque la structure de $\F_{k,r}$-modules étend celle de $\Dkq$-modules, c'est aussi un isomorphisme de $\F_{k,r}$-modules.
On obtient donc l'isomorphisme souhaité de $\F_{k,r}$-modules : $\F_{k,r} \otimes_{\F_{k+1 , r}} \tilde{\M}_{k+1 , r}  \simeq \Mkr$.
\end{proof}

Soit $\Mr = \varprojlim_{k \geq r} \Mkr$ un $\Fir$-module coadmissible.
On rappelle que les supports $\supp (\Mkr)$ sont des parties fermées de $\X$.
L'isomorphisme $\Mkr \simeq \F_{k , r} \otimes_{\F_{k+1 , r}} \tilde{\M}_{k+1 , r}$ induit un isomorphisme $(\Mkr)_x \simeq (\F_{k , r})_x \otimes_{(\F_{k+1 , r})_x} (\tilde{\M}_{k+1 , r})_x$ pour tout $x \in \X$.
Autrement dit, la suite des supports $(\supp (\Mkr))_k$ est croissante.
On définit le support de $\Mr$ par
\[ \supp (\Mr) := \overline{ \bigcup_{k \geq r} \supp (\Mkr)} . \]

\begin{remark}
Le $\Fir$-module $\Mr$ n'étant pas cohérent, il n'est pas clair que son support usuel en tant que faisceau sur $\X$, not\'e $S(\Mr)$ pour cette remarque, soit fermé.
On dispose seulement d'inclusions :
\[ \bigcup_{k \geq r} \supp (\Mkr) \subset S(\Mr) \subset \supp (\Mr) =\overline{ \bigcup_{k \geq r} \supp (\Mkr)} . \]
En effet, comme $\Mkr \simeq \F_{k , r} \otimes_{\F_{\infty , r}} \Mr$, on a $\supp (\Mkr)\subset S (\Mr)$ pour tout $k \geq r$.
Ainsi, $\bigcup_{k \geq r} \supp (\Mkr) \subset S (\Mr)$.
L'autre inclusion $S(\Mr) \subset \overline{ \bigcup_{k \geq r} \supp (\Mkr)}$ s'obtient de la manière suivante.
Si $x$ n'appartient pas à $\overline{ \bigcup_{k \geq r} \supp (\Mkr)}$, alors il existe un ouvert $U$ de $\X$ contenant $x$ pour lequel $(\Mkr)_{|U} = 0$ indépendamment de $k$.
La limite projective $\M_{|U} = \varprojlim_{k \geq 1} (\Mkr)_{|U} = 0$ est donc nulle et $x \notin S (\Mr)$.
\end{remark}

\begin{example}\label{ex5.6}~
On suppose que la courbe formelle $\X$ est affine munie d'une coordonnée locale.
Soit $P \in \Di(\X)$ et $\M = \Di / P$. On a $\Mkr = \F_{k,r} /P$ et $\Mr = \Fir / P$.

\begin{enumerate}
\item
Si $P$ est un opérateur infini, alors $P$ n'est jamais inversible dans le microlocalisé $\Fir(\X)$ d'après le théorème \ref{prop4.11}.
Ceci implique que $P$ n'est pas inversible dans $\F_{k,r}(\X)$ pour $k$ suffisamment grand et que $\supp (\F_{k,r} / P) = \X$. On en déduit que $\supp (\Mr) = \X$.
Prenons par exemple $P = \prod_{n \geq 1} (1 - a_n \cdot \varpi^n \cdot \partial) \in \Di(\X)$ avec $a_n \in \O_\X(\X)$ de norme un.
Alors $P$ s'écrit
\[ P = 1 + \sum_{n = 1}^\infty \beta_n \cdot \varpi^{n(n+1)/2} \cdot \partial^n  = 1 + \sum_{n = 1}^\infty \beta_n \cdot \varpi^{n((n+1)/2 - k)} \cdot (\varpi^k\partial)^n\]
avec $\beta_n = \varepsilon_n \cdot a_1 \dots a_n + \alpha_n$ où $\varepsilon_n \in \V^\times$ et $|\alpha_n | < 1$. On a donc $|\beta_n| = 1$.
On observe que $\Nb(P) = \nb(P) = k$. De plus, pour que $P$ soit inversible dans le microlocalisé $\F_{k,r}(\X)$, il faut avoir :
\[ \forall n \in \{ 0 , \dots , k-1 \}, ~  | \varpi|^{n((n+1)/2 - k)} \cdot p^{(k-r)(k-n)} < | \varpi|^{k((k+1)/2 - k)} . \]
Cette condition se réécrit
\[\forall n \in \{ 0 , \dots , k-1 \}, ~ \frac{k(k+1)}{2} - \frac{n(n+1)}{2} + r(n-k) < 0 . \]
Pour $n = 0$, on obtient $ \frac{k(k+1)}{2}  - rk < 0$. Ce n'est possible que pour $ \frac{k+1}{2}  < r$.
Ainsi, dès que $k$ est supérieur ou égal à $2 r- 1$, l'opérateur $P$ n'est jamais inversible dans le microlocalisé $\F_{k,r}(\X)$ et  $\supp (\F_{k,r} / P) = \X$.

\item
On suppose maintenant que $P = \sum_{n=0}^d a_n \cdot \partial^n$ est un opérateur fini d'ordre $d$.
Pour $k$ suffisamment grand, disons $ k \geq k_0$, $\Nb(P) = d$.
Pour $r$ assez grand, disons $r \geq r_0$, et pour $k \geq k_0' = \max\{k_0 , r_0\}$, $P$ vérifie la condition 3 de la proposition \ref{prop2.22} : 
\[ \forall n \in \{ 0 , \dots , d-1 \}, ~~ |a_n| \cdot p^{r(d-n)} < |a_d| . \]
Ainsi, pour tout entier $k \geq k_0'$ et pour tout entier $r \geq r_0$, $P$ est inversible dans le microlocalisé $(\F_{k,r})_x$ si et seulement si son coefficient dominant $a_d$ est inversible au voisinage de $x$.
Dans ce cas, $\supp (\Mkr) = V(a_d)$. On en déduit que $\supp (\Mr) = V(a_d)$ pour tout entier $r \geq r_0$.
Pour les autres valeurs de $r$, $P$ n'est jamais inversible. Donc $\supp (\Mkr) = \X$ pour tout $k$ dès que $r < r_0$ en prenant $r_0$ minimal.
Ainsi, $\supp (\Mr) = \X$ pour $r < r_0$.
\end{enumerate}
\end{example}

\subsubsection*{Les $\Firs$-modules $\Mrs$}

Soit $\M = \varprojlim_k \M_k$ un $\Di$-module coadmissible et $k \geq r \geq1$.
On rappelle que $\Mkrs = \Fkrs \otimes_{\pi^{-1} (\Dkq)} \pi^{-1} (\M_k)$ où $\pi : T^*\X \to \X$ désigne toujours la projection canonique.
Les morphismes de transition $\widehat{\mathcal{D}}^{(0)}_{\X, k+1 , \Q} \to \Dkq$, $\F_{k+1 , r}^* \to \F_{k , r}^*$ et $\M_{k+1} \to \M_k$ induisent des morphismes de transition $\M_{k+1 , r}^* \to \Mkrs$.
On définit un $\Firs$-module $\Mrs$ à partir de $\M$ par :
\[ \Mrs: = \varprojlim_{k \geq r} \Mkrs . \]
Pour tout point $(x, \xi)$ de $T^*\X$, on vérifie que
$
(\Mrs)_{(x,\xi)} \simeq 
\begin{cases}
(\Mr)_x ~~~~~ \mathrm{si}~~ \xi \neq 0 \\
(\M)_x \hspace{1.2cm} \mathrm{si}~~ \xi = 0 .
\end{cases}
$

\begin{prop}
Le $\Firs$-module $\Mrs$ associé à tout $\Di$-module coadmissible $\M$ est encore coadmissible.
\end{prop}
\begin{proof}
On note $\M = \varprojlim_k \M_k$. On dispose de l'isomorphisme de $\Fkrs$-modules
\begin{align*}
\Fkrs \otimes_{\F_{k+1 , r}^*} \M_{k+1 , r}^* & = \Fkrs \otimes_{\F_{k+1 , r}^*} \F_{k+1 , r}^* \otimes_{\pi^{-1}(\widehat{\mathcal{D}}^{(0)}_{\X, k+1 , \Q})} \pi^{-1}(\M_{k+1}) \\
& \simeq \Fkrs \otimes_{\pi^{-1}(\widehat{\mathcal{D}}^{(0)}_{\X, k+1 , \Q})} \pi^{-1}(\M_{k+1}) .
\end{align*}
Le morphisme de transition $\M_{k+1} \to \M_k$ induit un morphisme de transition $ \pi^{-1}(\M_{k+1}) \to \pi^{-1}(\M_k)$ et donc un morphisme $\Fkrs$-linéaire
\[ \Fkrs \otimes_{\F_{k+1 , r}^*} \M_{k+1 , r}^* \simeq \Fkrs \otimes_{\pi^{-1}(\widehat{\mathcal{D}}^{(0)}_{\X, k+1 , \Q})} \pi^{-1}(\M_{k+1}) \to \Mkrs = \Fkrs \otimes_{\pi^{-1}(\Dkq)} \pi^{-1}(\M_k) .\]
On vérifie ensuite que ce morphisme est un isomorphisme sur les germes. En effet, comme les modules $\M$ et $\Mr$ sont coadmissibles, on obtient
\begin{align*}
& (\Fkrs \otimes_{\F_{k+1 , r}^*} \M_{k+1 , r}^*)_{(x , \xi)}  \simeq \\
& \hspace{2cm}
\begin{cases}
(\F_{k, r})_x \otimes_{(\F_{k+1 , r})_x} (\tilde{\M}_{k+1 , r})_x \simeq (\tilde{\M}_{k , r})_x \simeq (\Mkrs)_{(x , \xi)} \hspace{1.1cm} \mathrm{si}~~ \xi \neq 0 \\
(\Dkq)_x \otimes_{(\widehat{\mathcal{D}}^{(0)}_{\X, k+1 , \Q})_x} (\M_{k+1})_x \simeq (\M_k)_x \simeq (\Mkrs)_{(x , \xi)} \hspace{0.8cm} \mathrm{si}~~ \xi = 0 .
\end{cases}
\end{align*}
On en déduit que $\Fkrs \otimes_{\F_{k+1 , r}^*} \M_{k+1 , r}^* \simeq \Mkrs$. Autrement dit, le module $\Mrs$ est coadmissible.
\end{proof}

\begin{remark}
Soient $\M$ et $\Nn$ deux $\Di$-modules coadmissibles. On suppose qu'il existe un morphisme de $\Di$-modules coadmissibles $\M \to \Nn $ induisant un isomorphisme $\Mrs \simeq \mathcal{\Nn}_r^*$ pour un certain entier $r \geq 1$. Alors $\M \simeq \Nn$.
En effet, l'hypothèse $\Mrs \simeq \mathcal{\Nn}_r^*$ implique que $(\Mrs)_{(x , 0)} \simeq \M_x \simeq (\mathcal{\Nn}_r^*)_{(x,0)} \simeq \Nn_x$ pour tout point $x$ de $\X$.
Le morphisme initial $\M \to \Nn$ est donc un isomorphisme de $\Di$-modules coadmissibles.
\end{remark}

Il a été vu que la suite des supports $(\supp \Mkr)_k$ est croissante pour $r\geq 1$ fixé.
De m\^eme, la suite $(\supp \M_k)_k$ est aussi croissante.
On en déduit, toujours pour $r$ fixé, que la suite des supports des modules $\Mkrs$ est croissante.
Le support du $\Firs$-module coadmissible $\Mrs$ est définie comme pour $\Mr$ par la fermeture de l'union croissante des supports des $\Fkrs$-modules cohérents $\Mkrs$ : 
\[ \supp \Mrs := \overline{ \bigcup_{k \geq r} \supp \Mkrs} \subset T^*\X . \]
Lorsque $\M$ est un $\Di$-module coadmissible non nul, les supports $\supp (\Mrs)$ sont des parties fermées coniques de $T^*\X$.
On en déduit que le support $\supp (\Mrs)$ est une aussi une partie fermée conique du fibré cotangent $T^*\X$.

\begin{example}\label{ex6.3}
On suppose que la courbe formelle $\X$ est affine munie d'une coordonnée étale. On rappelle que $T^*X$ et $T^*\X$ ont le même espace topologique.
On note $(x , \xi)$ les coordonnées locales de $T^*\X$ associées à la coordonnée étale donnée de $\X$.
Soit $P \in \Di(\X)$. On considère le $\Di$-module coadmissible $\M = \Di / P = \varprojlim_k \M_k$ avec $\M_k = \Dkq / P$.
On a $\Mkr = \F_{k,r} / P$, $(\Mkrs)_{(x , \xi)} \simeq (\Mkr)_x$ si $\xi \neq 0$ et $(\Mkrs)_{(x , \xi)} \simeq (\M_k)_x$ si $\xi = 0$.

\begin{enumerate}
\item
Cas d'un opérateur fini $P = \sum_{n=0}^d a_n \cdot \partial^n$ d'ordre $d$.
Pour $k$ assez grand, $\Nb(P) = d$. A partir de ce rang là et pour $r$ suffisamment grand, le support $\supp (\Mkrs)$ du module $\Mkrs$ coincide avec la variété caractéristique du module $\Dkq / P$.
En effet, pour $r$ et $k$ assez grand, le support du module $\Mkrs$ est donné par
\begin{center}
\begin{tikzpicture}

\draw[thick][->] (-2,0) -- (8,0);
\draw (8.2,0) node[right] {$x$};
\draw [thick][->] (0,-1) -- (0,3.5);
\draw (0,3.7) node[above] {$\xi$};
\draw (0,0) node[below right] {$0$};

\draw[red][thick] (-1.5,0) -- (7,0);

\draw[red][thick] (-1, -1) -- (-1,3) ;
\draw (-1,0) node[below right]{$x_1$} ;

\draw[red][thick] (1.8, -1) -- (1.8,3) ;
\draw (1.8,0) node[below right]{$x_2$} ;

\draw[red][thick] (3.5 , -1) -- (3.5,3) ;
\draw (3.5,0) node[below right]{$x_i$} ;

\draw[red][thick] (6, -1) -- (6,3) ;
\draw (6,0) node[below right]{$x_s$} ;

\draw (2.5,-1.5) node[right]{\textcolor{red}{$\supp (\Mkrs)$}} ;
\draw (7,3) node{$T^*\X$} ;

\end{tikzpicture}
\end{center}
avec $x_1, \dots , x_s$ les zéros du coefficient dominant $a_d$ de $P$.
Ce support admet la section nulle comme composante irréductible si et seulement si $d >0$.
Soit $(x , \xi)$ un point du fibré cotangent $T^*\X$. Si $(x , \xi) \in s(\X)$, alors $(\Mkrs)_{( x , \xi)} = (\Dkq /P)_x$. Si $d > 0$, alors $P$ n'est jamais inversible dans $(\Dkq)_x$.
En effet, pour que $P$ soit inversible dans $(\Dkq)_x$, une condition nécessaire est d'avoir $\Nb(P) = 0$.
Or pour $k$ suffisament grand, $\Nb(P) = d > 0$.
On en déduit que $(\Mkrs)_{( x , 0)} = (\Dkq /P)_x \neq 0$. Autrement dit, le support du module $\Mkrs$ contient la section nulle.
Les variétés caractéristiques $\Car(\Dkq / P)$ la contiennent aussi. Si $d = 0$, alors $P = a$ est une fonction de $\O_{\X , \Q}(\X)$.
Le support du module $\Dkq / a$ est l'ensemble des zéros de $a$ (en dehors de ces zéros, $a$ est inversible dans l'algèbre affinoïde $\O_{\X , \Q}(\X)$).
Dans ce cas, le support de $\Mkrs$ coupe la section nulle seulement en les zéros $x_1 , \dots , x_s$ de $a$, tout comme les variétés caractéristiques $\Car(\Dkq / P)$ pour $k$ suffisamment grand.
Si maintenant $(x , \xi) \in T^*\X \backslash s(\X)$, alors $(\Mkrs)_{( x , \xi)} = (\F_{k , r} /P)_x$.
Il a déjà été vu dans l'exemple \ref{ex5.6} que pour $r$ et $k$ assez grands, $(\F_{k , r} /P)_x \neq 0$ si et seulement si $x$ est un zéro du coefficient dominant $a_d$ de $P$.
Le support du module $\Mkrs$ contient ainsi les droites verticales d'équations $ x = x_i$, où $x_1 , \dots , x_s$ sont les zéros de $a_d$, auxquelles on a retiré les points de la section nulle.
Mais ces points de la section nulle sont compris dans le cas traité précédemment.
Pour résumer, les supports $\supp (\Mkrs)$ ne dépendent plus de $k$ pour $k$ suffisamment grand et $\supp (\Mkrs) = \Car(\M_k)$ en tant qu'espace topologique.
On en déduit que pour $r$ suffisamment grand, le support du module $\Mrs$ est donné par l'ensemble des points $(x , \xi) \in T^*\X$ vérifiant $a_d(x) \cdot \xi^d = 0$.
\item
Cas d'un opérateur différentiel infini $P$.
Pour $(x , 0) \in s(\X)$, on a $(\Mkrs)_{( x , 0)} = (\Dkq /P)_x$. La suite $(\Nb(P))_k$ diverge vers $+ \infty$ et pour $k$ suffisamment grand, $P$ n'est pas inversible dans $(\Dkq)_x$.
Ainsi, $(\Dkq /P)_x \neq 0$ et le support du module $\Mkrs$ contient la section nulle pour $k$ assez grand.
Enfin, pour tout point $(x , \xi) \in T^*\X \backslash s(\X)$ et pour $k$ suffisamment grand, $(\Mkrs)_{( x , \xi)} = (\F_{k , r} /P)_x \neq 0$ d'après l'exemple \ref{ex5.6}.
On en déduit que $\supp (\Mkrs) = T^*\X$ pour $k$ suffisamment grand.
Donc si $P$ est infini, alors $\supp (\Mrs) = T^*\X$ quelque soit l'entier $r \in \N^*$.
\end{enumerate}
La proposition suivante résume ce qui est dit dans cet exemple.
\end{example}

\begin{prop}\label{prop5.13}
Supposons la courbe formelle $\X$ affine munie d'une coordonnée locale. Soit $P \in \Di(\X)$ et $\M = \Di / P$.
Si $P$ est un opérateur infini, alors $\supp(\Mrs) = T^*\X$. Sinon, il existe un rang $r_0$ suffisamment grand tel que pour tous entiers $k \geq r \geq r_0$, $\supp(\Mrs) = \supp(\Mkrs) = \Car(\Dkq / P)$.
\end{prop}

\subsubsection*{Inégalité de Bernstein}

Soit  $\M = \varprojlim_k \M_k$ un $\Di$-module coadmissible et $r \geq 1$.
Rappelons que le support de $\Mrs$ est défini par $\supp \Mrs = \overline{\bigcup_{k \geq r} \supp \Mkrs }$.
On déduit de la proposition \ref{prop5.7} l'inégalité de Bernstein pour le support du $\Firs$-module coadmissible $\Mrs$.
En effet, si $\dim (\supp \Mrs) \geq 1$, alors $\dim (\supp \Mkrs) \geq 1$ pour $k$ suffisamment grand.
Autrement dit, $\M$ est non nul si et seulement si $\dim(\supp\Mrs) \geq 1$.
Le fibré cotangent $T^*\X$ de $\X$ est noethérien, les supports des $\Fkrs$-modules cohérents $\Mkrs$ n'ont qu'un nombre fini de composantes irréductibles.

\begin{prop}[Inégalité de Bernstein]\label{cor5.12}
Soit $\M$ un $\Di$-module coadmissible non nul.
Alors $\supp (\Mrs)$ est égal \`a $T^*\X$ ou bien équidimensionnel de dimension 1.
Dans le second cas, $\supp(\Mrs) = \supp (\Mkrs)$ pour $k$ suffisamment grand.
\end{prop}
\begin{proof}
Soit $\M = \varprojlim_k \M_k$ un $\Di$-module coadmissible non nul.
Si la dimension du support $\supp \Mrs$ est deux, alors $\supp \Mrs = T^*\X$ puisque le fibré cotangent $T^*\X$ est irréductible.
On suppose maintenant que $\dim (\supp \Mrs) = 1$. On a vu que $\dim (\supp \Mkrs) = 1$ pour $k$ suffisamment grand.
On ne considère plus que ces niveaux de congruences $k$.
Le support $\supp \Mrs$ n'a qu'un nombre fini de composantes irréductibles. Elles sont toutes de dimension un.
En effet, si une composante irréductible de $\supp \Mrs$ est réduite à un point, alors ce point appartient à l'un des supports $\supp \Mkrs$ pour $k$ suffisamment grand.
Mais les composantes irréductibles des supports $\supp \Mkrs$ sont toutes de dimension un d'après la proposition \ref{prop5.7}. 
Comme $\supp \Mkrs \subset \supp \Mrs$, le support $\supp \Mrs$ contient cette composante de dimension un.
Ainsi, ce point ne peut être une composante irréductible de $\supp \Mrs$.
Les composantes irréductibles des supports $\supp \Mkrs$ et du support $\supp \Mrs$ sont donc toutes de dimension un.
Pour une raison de dimension, les composantes irréductibles des supports $\supp \Mkrs$ sont des composantes irréductibles du support $\supp \Mrs$.
On en déduit que $\supp \Mrs = \supp \Mkrs$ à partir d'un certain niveau de congruence $k \geq r$.
En effet, puisque le support $\supp \Mrs$ n'a qu'à nombre fini de composantes irréductibles, les supports sont nécessairement égaux à partir d'un niveau de congruence $k$ suffisamment grand. 
\end{proof}

Si $\dim (\supp \Mrs) = 1$, alors $\dim (\supp \Mkrs) = 1$ pour $k$ suffisamment grand.
La proposition \ref{prop5.5} implique que les $\Dkq$-modules cohérents $\M_k$ sont holonomes pour $k \geq r$.
La réciproque est néanmoins fausse : les $\Dkq$-modules cohérents $\M_k$ peuvent être holonomes pour tout niveau de congruence $k$ (ie $\dim (\Car \M_k) \leq 1$) tout en ayant $\dim(\supp \Mrs) = 2$.
C'est le cas par exemple pour $\M = \Di / P$ dès que $P$ est un opérateur différentiel infini de $\Di(U)$. En effet, il a été vu dans l'exemple \ref{ex6.3} qu'alors $\supp \Mrs = T^*\X$.
Cependant, d'après la proposition \ref{prop3.22}, les $\Dkq$-modules cohérents $\Dkq / P$ sont holonomes.

\subsection{Variété caractéristique}\label{section5.3}

Soit $\M = \varprojlim_k \M_k$ un $\Di$-module coadmissible. Pour tout entier $k \geq r$, on rappelle que 
$\Mkr  = \F_{k,r} \otimes_{\Dkq} \M_k$ et que $\Mr = \varprojlim_{k \geq r} \Mkr $ est un $\Fir$-module coadmissible.
Pour $k \geq r+1$, on note $j_{k,r} : \F_{k,r} \to \F_{k , r+1}$ le morphisme induit localement par les inclusions d'alg\`ebres  $\F_{k,r}(U) \subset \F_{k , r+1}(U)$. Les morphismes $\F_{k,r}$-linéaires
\[ j_{k,r} \otimes \id : \Mkr = \F_{k,r} \otimes_{\Dkq} \M_k \to \tilde{\M}_{k ,r+1} = \F_{k,r+1} \otimes_{\Dkq} \M_k\]
commutent avec les morphismes de transition, ie le diagramme suivant commute :
\[ \xymatrix@R=3pc @C=5pc{ \Mkr \ar[r]^{j_{k,r} \otimes \id} \ar[d]_{g_{k,r}} & \tilde{\M}_{k , r+1} \ar[d]^{g_{k,r+1}} \\ \tilde{\M}_{k+1 , r} \ar[r]_{j_{k+1,r} \otimes \id} & \tilde{\M}_{k+1 , r+1} \, . } \]
En passant à la limite projective sur $k \geq r+1$, on obtient un morphisme continu $\Fir$-linéaire $j_r : \Mr \to \tilde{\M}_{\infty , r+1}$. On a
\begin{align*}
\F_{k , r+1} \otimes_{\F_{k,r}} \Mkr & = \F_{k , r+1} \otimes_{\F_{k,r}} \F_{k,r} \otimes_{\Dkq} \M_k \\
& \simeq \F_{k , r+1} \otimes_{\Dkq} \M_k \\
& \simeq \tilde{\M}_{k , r+1}
\end{align*}
en tant que $\Dkq$-modules. Puisque la structure de $\F_{k , r+1}$-modules étend celle de $\Dkq$-modules, on obtient un isomorphisme $\F_{k , r+1} \otimes_{\F_{k,r}} \Mkr \simeq \tilde{\M}_{k , r+1}$ de $\F_{k , r+1}$-modules.
On en déduit que $\supp (\tilde{\M}_{k , r+1}) \subset \supp (\Mkr)$ et que $\supp (\tilde{\M}_{\infty ,  r+1}) \subset \supp (\Mr)$.
De manière analogue, on définit pour les $\Firs$-modules coadmissibles $\Mrs$ des morphismes de transition induits par les morphismes $\Fkrs \to \F_{k , r+1}^*$ pour chaque niveau de congruence $k \geq r +1$.
Pour $(x , \xi) \in T^*\X$, le morphisme $ (\Mrs)_{(x , \xi)} \to (\M_{r+1}^*)_{(x , \xi)}$ correspond aux morphismes
\[
\begin{cases}
(\Mr)_x \longrightarrow (\tilde{\M}_{\infty , r+1})_x \hspace{1cm} \mathrm{si}~~ \xi \neq 0 \\
(\M)_x \overset{\id}{\longrightarrow} (\M)_x \hspace{2.49cm} \mathrm{si}~~ \xi = 0.
\end{cases}
\]
Les inclusions $\supp (\tilde{\M}_{\infty ,  r+1}) \subset \supp (\Mr)$ pour $r \geq 1$ induisent des inclusions des supports $\supp (\M^*_{r+1}) \subset  \supp (\Mrs)$.

\begin{definition}
La variété caractéristique d'un $\Di$-module coadmissible $\M$ est d\'efinie comme l'intersection de la suite décroissante des supports $\supp(\Mrs)$ : 
\[ \Car(\M) := \bigcap_{r \geq 1} \supp (\Mrs) \subset T^*\X . \]
\end{definition}

Puisque les supports $\supp(\Mrs)$ sont ferm\'es dans $T^*\X$, la vari\'et\'e caractéristique est une partie fermée du fibré cotangent $T^*\X$ de $\X$.

\begin{example}\label{ex5.5}
Reprenons l'exemple \ref{ex6.3} : la courbe formelle $\X$ est affine et $\M = \Di / P$ pour un opérateur $P \in \Di(\X)$.
Rappelons que $\Mkr = \F_{k,r} /P$ et que $\Mr = \Fir / P$.

\begin{enumerate}
\item
Si $P$ est un opérateur infini, alors $\supp (\Mr) = \X$ pour tout $r \geq 1$ d'après l'exemple \ref{ex6.3}. Ainsi, $\Car (\M) = T^*\X$.
\item
Si $P = \sum_{n=0}^d a_n \cdot \partial^n$ est un opérateur fini d'ordre $d$, alors il existe un rang $r_0 \geq 1$ tel que $\supp (\Mr) = V(a_d)$ pour tout $r \geq r_0$ et $\supp (\Mr) = \X$ pour tout $r < r_0$.
Pour $r \geq r_0$, ces supports sont l'ensemble des abscisses des composantes irréductibles des variétés caractéristiques $\Car\left( \Dkq / P \right)$ pour $k$ assez grand.
L'exemple \ref{ex6.3} montre que pour $r$ suffisamment grand, le support de $\Mrs$ coïncide avec la variété caractéristique du module cohérent $\Dkq / P$ pour $k \geq r$ suffisamment grand.
Ainsi, la vari\'et\'e caract\'eristique $\Car(\M) = \cap_{ r \geq 1 } \supp (\Mrs)$ coincide les variétés caractéristiques $\Car\left( \Dkq / P \right)$ pour $k$ suffisamment grand :
\[ \Car (\M) = \left\{ (x , \xi) \in T^*\X ~ : ~ a_d(x) \cdot \xi^d = 0 \right\} = \Car(\M_k) . \]
\end{enumerate}
\end{example}

On termine cette partie par la preuve de l'inégalité de Bernstein : un $\Di$-module coadmissible $\M$ est non nul si et seulement si $\dim (\Car(\M)) \geq 1$.

\begin{prop}\label{bernsteincoad}[Inégalité de Bernstein]
Soit $\M = \varprojlim_k \M_k$ un $\Di$-module coadmissible non nul.
Il existe un rang $r \in \N^*$ à partir duquel $\Car(\M) = \supp (\Mrs)$.
De plus, $\Car(\M)$ est égal \`a $T^*\X$ ou bien équidimensionnel de dimension 1.
Dans le second cas, $\Car(\M) = \supp (\Mkrs)$ pour $k$ suffisamment grand (dépendant de $r$).
\end{prop}
\begin{proof}
Si $\dim(\supp \Mrs) =2$ pour tout $r \geq 1$, alors $\supp (\Mrs) = T^*\X$ et donc $\Car (\M) = T^*\X$.
Sinon, $\dim(\supp \Mrs) = 1$ à partir d'un rang $r_0 \geq 1$ et $\dim \Car(\M) \leq 1$.
On ne considère maintenant que les rangs $r$ plus grands que $r_0$.
Puisque $\supp \M_{r+1}^* \subset \supp \Mrs$ et comme les composantes irréductibles des supports $\supp \Mrs$ sont toutes de dimension un, les composantes irréductibles de $\supp \M_{r+1}^*$ sont des composantes irréductibles de $\supp \Mrs$.
On suppose que $\dim \left( \bigcap_{r \geq 1} \supp \Mrs \right) = 0$. Soit $L$ une composante irréductible de $\supp \M_{r_0}^*$ ; elle est de dimension un. 
Si $L$ est une composante irréductible de tous les supports $\supp \Mrs$, alors $L$ est aussi une composante irréductible de $\Car(\M) = \bigcap_{r \geq 1} \supp \Mrs$.
Cela contredit l'hypothèse $\dim \Car(\M) = 0$. On en déduit que $L$ n'est plus une composante irréductible des supports $\supp \Mrs$ pour $r$ assez grand.
Enfin, comme le support $\supp \M_{r_0}^*$ n'a qu'un nombre fini de composantes irréductibles, aucune composante irréductible de $\supp \M_{r_0}^*$ n'est une composante irréductible des supports $\supp \Mrs$ pour un rang $r$ suffisamment grand.
Il en découle que $\dim (\supp \Mrs) = 0$. Alors $\M = 0$ d'après la proposition \ref{cor5.12}.
Puisque $\M \neq 0$, on en déduit que $\dim (\Car \M) = 1$ et que les composantes irréductibles de $\Car \M$ sont les composantes irréductibles communes à tous les supports $\supp \Mrs$.
Les supports des modules $\Mrs$ et la variété caractéristique $\Car(\M)$ n'ont qu'un nombre fini de composantes irréductibles.
On en déduit que $\Car(\M) = \supp \Mrs$ à partir d'un certain rang $r$.
Enfin, on sait toujours d'après la proposition \ref{cor5.12} que $\supp \Mrs = \supp \Mkrs$ pour un niveau de congruence $k$ suffisamment grand.
\end{proof}

\section{$\Di$-modules sous-holonomes}\label{section6}

Un $\Di$-module coadmissible $\M$ est appelé sous-holonome si sa variété caractéristique est de dimension inférieure ou égale à $\dim(\X) = 1$.
On démontre dans la partie \ref{section6.1} que $\M$ est sous-holonome si et seulement si $\M$ est génériquement un module à connexion intégrable.
Par ailleurs, un tel module est sous-holonome au sens de Ardakov-Bode-Wadsley.
Enfin on associe dans la section \ref{section6.2} un cycle caractéristique à tout $\Di$-module sous-holonome. On en déduit que ces modules sont de longueur finie.

\subsection{Définition et propriétés}\label{section6.1}

On associe à tout $\Di$-module coadmissible $\M$ un support \og infini \fg\, définit comme étant la partie fermée de $\X$ donnée par l'intersection des fermés $\supp(\Mr)$ :
\[ \suppi (\M) := \bigcap_{r \geq 1} \supp (\Mr) \subset \X. \]
La variété caractéristique $\Car(\M)$ contient la droite verticale d'abscisse $x \in \X$ si et seulement si $x \in \suppi(\M)$.
Ainsi, $\dim(\Car(\M)) \leq 1$ si et seulement si $\dim(\suppi(\M)) = 0$.
Dans ce cas, les abscisses des composantes verticales de la variété caractérisque $\Car(\M)$ sont exactement les points du support $\suppi(\M)$.

\begin{definition}
Un $\Di$-module coadmissible $\M$ est sous-holonome si $\dim (\Car(\M)) \leq 1$.
De manière équivalente, $\M$ est sous-holonome si et seulement si $\dim(\suppi(\M)) = 0$.
\end{definition}

Soit $\M = \varprojlim_k \M_k$ un $\Di$-module coadmissible sous-holonome. D'après l'inégalité de Bernstein, ou bien $\M = 0$ et $\Car(\M) = \emptyset$, ou bien $\M \neq 0$ et $\dim \Car(\M) = 1$.
Par ailleurs, les $\Dkq$-modules cohérents $\M_k$ sont holonomes pour un niveau de congruence $k$ suffisamment grand.
En effet, $\Car (\M_k) \subset \supp (\M_{k , r}^*) \subset \supp (\M^*_{\infty , r}) =  \Car(\M)$ pour $k \geq r$ assez grands. Il en découle que $\dim( \Car \M_k) \leq 1$.
Le résultat suivant découle de la proposition \ref{prop5.13}.

\begin{prop}
On suppose la courbe formelle $\X$ affine munie d'une coordonnée locale.
Soit $P$ un opérateur différentiel non nul de $\Di(\X)$.
Alors le $\Di$-module coadmissible $\Di / P$ est sous-holonome si et seulement si $P$ est un opérateur différentiel fini.
Dans ce cas, $\Car(\M) = \Car(\M_k)$ pour tout niveau de congruence $k$ suffisamment grand.
\end{prop}

On montrera dans la partie \ref{section6.2} que l'\'egalit\'e $\Car(\M) = \Car(\M_k)$ est v\'erifi\'ee plus g\'en\'eralement d\`es qu'un $\Di$-module coadmissible $\M$ est holonome.
La proposition suivante implique que la catégorie des $\Di$-modules sous-holonomes est abélienne.

\begin{prop}\label{prop6.1.5}
Soit $\xymatrix{ 0 \ar[r] & \M \ar[r] & \Nn \ar[r] & \L \ar[r] & 0 }$ une suite exacte de $\Di$-modules coadmissibles.
Alors $\Car(\Nn) = \Car(\M) \cup \Car(\L)$. En particulier, $\Nn$ est sous-holonome si et seulement si $\M$ et $\L$ le sont aussi.
\end{prop}
\begin{proof}
On note $\M = \varprojlim_k \M_k$, $\Nn = \varprojlim_k \Nn_k$ et $\L = \varprojlim_k \L_k$.
Par définition, $\M_k$, $\Nn_k$ et $\L_k$ sont des $\Dkq$-modules cohérents.
Par définition des modules coadmissibles, la suite exacte courte initiale induit une suite exacte courte $\xymatrix{ 0 \ar[r] & \M_k \ar[r] & \Nn_k \ar[r] & \L_k \ar[r] & 0 }$ de $\Dkq$-modules pour tout niveau de congruence $k \in \N$.
On rappelle que $\Mkr = \F_{k , r} \otimes_{\Dkq} \M_k $, $\tilde{\Nn}_{k , r} = \F_{k , r} \otimes_{\Dkq} \Nn_k $ et $\tilde{\L}_{k , r} = \F_{k , r} \otimes_{\Dkq} \L_k$.
Puisque le morphisme $\Dkq \to \F_{k , r}$ est plat d'après la proposition \ref{prop2.28}, on en déduit une suite exacte courte $\xymatrix{ 0 \ar[r] & \Mkr  \ar[r] &  \tilde{\Nn}_{k , r}  \ar[r] &  \tilde{\L}_{k , r}  \ar[r] & 0 }$ de $\F_{k , r}$-modules cohérents.
Pour $k \geq r$, on obtient en passant au fibre cotangent $T^*\X$ une suite exacte courte de $\Fkrs$-modules cohérents $\xymatrix{0 \ar[r] & \Mkrs  \ar[r] &  \Nn_{k , r}^*  \ar[r] &  \L_{k , r}^*  \ar[r] & 0 }$.
L'égalité suivante des supports de ces modules en découle :
\begin{equation}\label{eqsupp}
\supp (\Nn_{k , r}^*) = \supp (\Mkrs ) \, \cup \, \supp (\L_{k , r}^*) .
\end{equation}
D'après la proposition \ref{bernsteincoad}, il existe un rang $r \geq 1$ pour lequel $\Car(\M) = \supp (\Mrs)$, $\Car(\Nn) = \supp (\Nn_r^*)$ et $\Car(\L) = \supp (\L_r^*)$.
On suppose dans un premier temps que les modules $\M$ et $\L$ sont tous deux sous-holonomes.
Il existe alors un niveau de congruence $k_r \geq r$ tel que $\Car(\M) = \supp (\Mrs) = \supp (\Mkrs)$ et $\Car(\L) = \supp (\L_r^*) = \supp (\L_{k , r}^*)$ pour tout entier $k \geq k_r$.
La relation \ref{eqsupp} montre que pour tout $k \geq k_r$,
\[ \supp( \Nn_{k , r}^*) =  \supp (\Mkrs ) \, \cup \, \supp (\L_{k , r}^*) = \Car(\M) \, \cup \, \Car(\L) .\]
Puisque la suite des supports $(\Nn_{k , r}^*)_{k \geq r}$ est croissante, on en déduit que pour tout $r$ suffisamment grand,
\[ \supp \Nn_r^* = \overline{ \bigcup_{k \geq r} \supp \Nn_{k , r}^*} = \Car(\M) \, \cup \, \Car(\L) .\]
Comme la suite des supports $( \supp \Nn_r^*)_r$ est décroissantes, il en découle que
\[ \Car(\Nn) = \bigcap_{r \geq 1} \supp \Nn_r^* =\Car(\M) \, \cup \, \Car(\L) .\]
Autrement dit, $\Nn$ est sous-holonome dès que $\M$ et $\L$ le sont.
En utilisant toujours \ref{eqsupp}, on montre de façon analogue que si $\Nn$ est sous-holonome, alors l'égalité des supports tient encore.
Ainsi, on a démontré que $\Nn$ est sous-holonome si et seulement si $\M$ et $\L$ le sont aussi ; dans ce cas $\Car(\Nn) = \Car(\M) \cup \Car(\L)$.
Si les modules ne sont pas sous-holonomes, alors la proposition \ref{bernsteincoad} implique que $\Car(\Nn) = T^*\X =  \Car(\M) \cup \Car(\L)$.
Autrement dit, l'égalité des supports tient dans tous les cas possibles.
\end{proof}

\begin{example}
Un $\Di$-module coadmissible $\M$ est dit algébrique s'il existe un $\D^{(0)}_{\X , \Q}$-module cohérent $\M_{\text{alg}}$ tel que $\M \simeq \Di \otimes_{\D^{(0)}_{\X , \Q}} \M_{\text{alg}}$.
Tout $\Di$-module algébrique faiblement holonome $\M$ est sous-holonome.
En effet, on peut supposer $\X$ affine muni d'une cordonnée locale et $\M_{\text{alg}} \simeq \D^{(0)}_{\X , \Q} \cdot m$ monogène.
Puisque $\M$ est faiblement holonome, les $\Dkq$-modules cohérents $\M_k \simeq \Dkq \otimes_{\D^{(0)}_{\X , \Q}} \M_{\text{alg}} $ sont holonomes.
On en déduit que $\M_{\text{alg}} \not\simeq \D^{(0)}_{\X , \Q}$ car sinon $\M_k \simeq \Dkq$ ne serait pas holonome.
En particulier, $\M_{\text{alg}} \simeq \D^{(0)}_{\X , \Q} / \I$ pour un idéal cohérent non nul $\I$ de $\D^{(0)}_{\X , \Q}$.
Soit $P \in \I$ un opérateur différentiel non nul. On dispose d'une suite exacte $\D^{(0)}_{\X , \Q} / P \to \M_{\text{alg}} \to 0$.
Après tensorisation par $\Di$, on obtient une surjection $\Di / P \to \M \to 0$.
Puisque $P$ est un opérateur fini, le module $\Di / P$ est sous-holonome.  Il en est donc de même pour $\M$.
\end{example}

On identifie dans la suite la courbe formelle $\X$ avec la section nulle $s : \X \to T^*\X$ de son fibré cotangent $T^*\X$.

\begin{definition}
Un $\Di$-module coadmissible $\M$ est appel\'e un module à connexion intégrable si $\Car(\M) \subset \X$.
\end{definition}

On démontre, comme le nom choisit le laisse deviner, que les modules à connexion intégrables sont exactement les $\O_{\X , \Q}$-modules localement libres de rang fini.
La proposition suivante fait le lien entre les modules à connexion intégrable et les modules sous-holonomes.

\begin{prop}\label{prop7.2.4}
Un $\Di$-module coadmissible $\M$ est sous-holonome si et seulement si il existe un ouvert dense $U$ de $\X$ pour lequel $\M_{|U}$ est un module à connexion intégrable.
\end{prop}
\begin{proof}
On suppose dans un premier temps que $\M$ est sous-holonome. La variété caractéristique de $\M$ a un nombre fini de composantes irréductibles verticales.
Soit $U = \X \backslash \suppi(\M)$ l'ouvert de $\X$ obtenu en ôtant les abscisses des composantes verticales de $\Car(\M)$. 
On a $\Car(\M_{|U}) \subset U$. Autrement dit, $\M_{|U}$ est un module à connexion intégrable. 
Réciproquement, soit $U$ un ouvert non vide de $\X$ tel que $\M_{|U}$ soit un module à connexion intégrable.
Par définition, $\Car(\M_{|U}) \subset U$.
Si $\M$ n'est pas holonome, alors $\Car(\M) = T^*\X$ d'après la proposition \ref{bernsteincoad}.
On aurait alors $\Car(\M_{|U}) = T^*U$. Cela contredit l'hypothèse $\Car(\M_{|U}) \subset U$.
On en déduit que le module $\M$ est sous-holonome.
\end{proof}

Si $\M = \varprojlim_k \M_k$ est une connexion intégrable, alors $\Car(\M_k) \subset \Car(\M) \subset X$ pour tout niveau de congruence $k$ suffisamment grand.
La proposition \ref{prop3.22} implique que chaque module $\M_k$ est localement un $\O_{\X , \Q}$-module libre de rang fini.
On démontre maintenant que ces rangs ne dépendent pas de $k$.
On en déduit que le $\Di$-module coadmissible $\M$ est alors un $\O_{\X , \Q}$-module localement libre de rang fini.

\begin{lemma}\label{lemmerangfini}
Soit $\M = \varprojlim_k \M_k$ tel que $\Car(\M) = \X$.
Alors pour $k$ suffisamment grand, les $\M_k$ sont localement des $\O_{\X , \Q}$-modules libres de même rang fini.
\end{lemma}
\begin{proof}
On suppose la courbe formelle $\X$ affine munie d'une coordonnée locale.
Soit $r \in \N^*$ un rang pour lequel $\supp (\Mrs) = \Car(\M) = \X$.
On ne considère maintenant que les niveaux de congruences $k \geq r$ pour lesquels $\Car(\M_k) = X$.
Soit $x$ un point ferm\'e de $\X$. Quitte à étendre les scalaires par une extension finie de $K$, on peut supposer que $x$ est un point $\kappa$-rationnel.
D'après le lemme 3.23 de \cite{hallopeau1}, il existe un voisinage ouvert $U_k$ de $x$ dans $\X$ et un opérateur différentiel $P_k$ de $\Dkq(U_k)$ fini dominant unitaire tel que $\M_{k|U_k} \simeq \widehat{\mathcal{D}}^{(0)}_{U_k, k, \Q} / P_k$ soit un $\O_{U_k , \Q}$-module libre de rang l'ordre $d_k$ de $P_k$.
Il suffit de montrer que ce rang $d_k$ ne dépend pas de $k$.
Puisque $\M = \varprojlim_k \M_k$ est un $\Di$-module coadmissible, on a
\[ \widehat{\mathcal{D}}^{(0)}_{U_k\cap U_{k+1}, k, \Q} \otimes_{\widehat{\mathcal{D}}^{(0)}_{U_k\cap U_{k+1}, k+1 , \Q}} \widehat{\mathcal{D}}^{(0)}_{U_k\cap U_{k+1}, k+1 , \Q} / P_{k+1} \simeq \widehat{\mathcal{D}}^{(0)}_{U_k\cap U_{k+1}, k, \Q} / P_k .\]
Autrement dit, $\widehat{\mathcal{D}}^{(0)}_{U_k\cap U_{k+1}, k, \Q} / P_{k+1} \simeq \widehat{\mathcal{D}}^{(0)}_{U_k\cap U_{k+1}, k, \Q} / P_k$ et \cite[corollaire 2.16]{hallopeau1} implique que $\Nb(P_k) = \Nb(P_{k+1})$.
De plus, on sait par le lemme \ref{lemmedeg} que $\overline{N}_{k+1}(P_{k+1}) \geq \Nb(P_{k+1})$.
On en déduit que la suite $(\Nb(P_k))_k$ est croissante. Il faut donc prouver qu'elle est stationnaire.
On suppose par l'absurde qu'elle diverge vers l'infini.
Soit $k \geq r$ un entier tel que $d_k = \Nb(P_k) > d_{k-1} = \overline{N}_{k-1}(P_k) =\overline{N}_{k-1}(P_{k-1})$.
L'hypothèse
\[ \supp\left((\Mkrs)_{|\pi^{-1}(U_k)}\right) = \supp\left( (\Fkrs)_{|\pi^{-1}(U_k)} / P_k\right) \subset \Car\left((\M)_{|U_k}\right) = U_k \]
implique que l'opérateur $P_k$ est inversible dans le microlocalisé $\F_{k,r}(U_k)$.
Autrement le support du module $(\Mkrs)_{|\pi^{-1}(U_k)}$ contiendrait des composantes irréductibles verticales.
On écrit $P_k = \sum_{n=0}^{d_k} a_n \cdot \partial^n$, où $a_n \in \O_{\X , \Q}(U_k)$.
On rappelle que $P_k \in \F_{k,r}(U_k)^\times$ si et seulement si $\Nb(P_k) = \nb(P_k) = d_k$ et si pour tout $n\in \{ 0 , \dots , d_k-1 \}$, $|a_n|  < |a_{d_k}| \cdot p^{r(d_k-n)}$.
Pour $n = d_{k-1}$, on obtient
\[ |a_{d_{k-1}} | < |a_{d_k}| \cdot p^{r(d_k - d_{k-1})} . \]
La condition $d_{k-1} < d_k$ implique
\[ |a_{d_{k-1}} | > |a_{d_k}| \cdot p^{(k-1)(d_k - d_{k-1})} . \]
Comme la suite $(\Nb(P_k))_k$ n'est pas bornée, on peut choisir $k > r+1$.
Sous cette hypothèse, les deux inégalités précédentes sont incompatibles !
On en déduit par l'absurde que la suite $(d_k)_k$ est stationnaire.
\end{proof}

\begin{prop}\label{propconnex}
Un $\Di$-module coadmissible $\M$ est un module à connexion intégrable si et seulement si $\M$ est localement un $\O_{\X,\Q}$-module libre de rang fini.
\end{prop}
\begin{proof}
On peut se ramener au cas où $\X$ est affine muni d'une coordonnée locale.
On suppose tout d'abord que $\M$ est un $\O_{\X,\Q}$-module libre non nul de rang $n \in \N^*$.
Il existe des sections $e_1 , \dots , e_n$ de $\M(\X)$ telles que $\M \simeq \O_{\X , \Q} \cdot e_1 \oplus \dots \oplus \O_{\X , \Q} \cdot e_n$ en tant que $\O_{\X , \Q}$-module.
Les sections $e_1 , \dots , e_n$ engendrent aussi $\M$ en tant que $\Di$-module : $\M \simeq \Di \cdot e_1 + \dots + \Di \cdot e_n$.
On en déduit que $\Car(\M) \subset \bigcup_{i=1}^n \Car(\Di \cdot e_i)$. Il suffit de montrer que $\Car(\Di \cdot e_i) \subset \X$ (section nulle du fibr\'e cotangent) pour chaque indice $i \in \{ 1 , \dots , n \}$.
On peut donc se ramener au cas $n = 1$ : $\M$ est un $ \O_{\X , \Q}$-module libre de rang un engendré par une section $e$.
Puisque $\M \simeq  \O_{\X , \Q}\cdot e$ en tant que $ \O_{\X , \Q}$-module, la famille $\{ \partial^n \cdot e \}_{n \in \N}$ est liée sur $ \O_{\X , \Q}(\X)$.
Il existe donc un opérateur différentiel unitaire $P \in \Di(\X)$ tel que $P \cdot e = 0$.
Comme $e$ est un g\'en\'erateur de $\M$, il est possible de trouver un tel op\'erateur $P$ de degr\'e un, ie $P = \partial - f$.
Quoiqu'il en soit, $P$ appartient à l'idéal annulateur de $e$ et $\M$ est un $\Di$-module coadmissible quotient de $\Di / P$.
On en déduit que $\Car(\M) \subset \Car(\Di / P)$. Comme $P$ est un opérateur fini de coefficient dominant unitaire, on a  $\Car(\Di / P) = \X$.
Donc $\Car(\Di \cdot e) \subset \X$.
Réciproquement, on suppose que $\M = \varprojlim_k \M_k$ est un module non nul à connexion int\'egrable.
D'après le lemme \ref{lemmerangfini}, il existe un entier $n \in \N^*$ et un niveau de congruence $k_0 \in \N^*$ tels que $\M_k$ soit un $\O_{\X , \Q}$-module libre de rang $n$ pour tout niveau de congruence $k \geq k_0$.
Alors $\M$ est un $\O_{\X , \Q}$-module libre de rang fini $n$ d'après \cite[lemme 4.13]{hallopeau1}.
\end{proof}

Si $\M = \varprojlim_k \M_k$ est un $\Di$-module sous-holonome, alors il existe un ouvert dense $U$ de $\X$ tel que $\M_{|U}$ soit une à connexion intégrable.
La proposition \ref{propconnex} implique que $\M_{|U}$ est un $\O_{|U , \Q}$-module libre de rang fini.

\begin{definition}
La multiplicité horizontale $m_0(\M)$ d'un $\Di$-module sous-holonome $\M$ est définie comme le rang du $\O_{U , \Q}$-module localement libre $(\M)_{|U}$, o\`u $U = \X \backslash \suppi(\M)$.
\end{definition}

\begin{example}\label{ex6.10}
Soit $\M = \varprojlim_k \M_k$ un $\Di$-module coadmissible support\'e en un point ferm\'e $x \in \X$.
Alors $\M$ est sous-holonome et sa variété caractéristique consiste en une composante verticale d'abscisse $x$.
En effet, pour $k \geq r \geq 1$, rappelons que $\Mkr = \F_{k , r} \otimes_{\Dkq} \M_k$.
Puisque pour $k$ suffisamment grand, $\supp(\M_k) = \{x\}$, il en découle que $\supp(\Mkr) \subset \{x\}$.
Ainsi, $\suppi(\M) \subset \{x\}$ et $\M$ est sous-holonome.
Par ailleurs, $\M$ n'est pas g\'en\'eriquement une connexion car sinon son support contiendrait un ouvert non vide de $\X$.
Autrement dit, la variété caractéristique de $\M$ est réduite \`a des composantes irréductibles verticales d'abscisses les éléments de $\suppi(\M)$.
D'apr\`es l'in\'egalit\'e de Bernstein, $\dim \Car(\M) \geq 1$ et donc nécessairement $\suppi(\M) \neq \emptyset$.
On en d\'eduit que $\suppi(\M) = \{x\}$.

\end{example}

Soit $\M = \varprojlim_k \M_k$ un $\Di$-module coadmissible.
Si $\M$ est sous-holonome, alors les modules $\M_k$ sont holonomes pour tout niveau de congruence $k$.
En particulier, $\M$ est faiblement holonome d'après la proposition \ref{prop4.2.4}.
Son dual $\M^\vee$ est un module faiblement holonome. On montre maintenant que si $\M$ est holonome, alors son dual $\M^\vee$ reste holonome.

\begin{prop}
Soit $\M = \varprojlim_k \M_k$ un $\Di$-module coadmissible.
Alors $\M$ est sous-holonome si et seulement si $\M$ est faiblement holonome et si son dual $\M^\vee$ est sous-holonome.
\end{prop}
\begin{proof}
Il suffit de montrer que si $\M$ est sous-holonome, alors son dual $\M^\vee$ l'est aussi.
L'équivalence énoncée provient ensuite de l'isomorphisme de bidualité.
D'après la proposition \ref{prop7.2.4}, il existe un ouvert non vide $U$ de $\X$ tel que $\M_{|U}$ soit un module à connexion.
On peut donc supposer que $\X$ est affine et que $\M$ est un module à connexion de rang $n$.
Toujours d'après la proposition \ref{prop7.2.4}, on peut se ramener à montrer que $\M^\vee$ est un module à connexion intégrable.
Comme $\M = \varprojlim_k \M_k$ est un module à connexion intégrable, les modules $\M_k$ sont des modules à connexion intégrable de rang $n$ pour $k$ suffisamment grand.
Il reste à prouver que leurs duaux $\M_k^\vee = \Ext^1_{\Dkq}(\M_k , \Dkq) \otimes_{\O_{\X , \Q}} \omega_{\X , \Q}^{-1}$ sont aussi des modules à connexion de rang $n$.
D'après \cite[lemme 3.23]{hallopeau1} et quitte \`a r\'eduire $\X$, on peut se ramener au cas où $\M_k \simeq \Dkq / P$ avec $P$ un opérateur différentiel fini unitaire.
Il suffit de vérifier que $\M_k^\vee$ est aussi de la forme $\Dkq / P^t$ avec $P^t$ fini unitaire.
On considère la suite exacte courte de $\Dkq$-modules cohérents $\xymatrix{ 0 \ar[r] & \Dkq \ar[r]^{\times P} & \Dkq \ar[r] & \M_k \ar[r] & 0 }$.
Puisque $\Ext^1_{\Dkq}(\Dkq , \Dkq) = 0$ et en prenant le foncteur $\Hom_{\Dkq}( \bullet , \Dkq)$, on obtient la suite exacte courte suivante de $\Dkq$-modules à droites :
\[ \xymatrix{
0 \ar[r] & \Hom(\M_k ,\Dkq)  \ar[r] & \Hom(\Dkq ,\Dkq) \ar[r] & \Hom(\Dkq ,\Dkq)  \ar[d] & \\
& & & \Ext^1_{\Dkq}(\M_k , \Dkq) \ar[r] & 0 .
} \]
Comme les algèbres $\Dkq(U)$ sont intègres et $\M_k \simeq \Dkq / P$ avec $P$ non nul, on a $\Hom(\M_k , \Dkq) = 0$.
Par ailleurs, $\Hom(\Dkq ,\Dkq) \simeq \Dkq$ en tant que $\Dkq$-module à droite.
Le morphisme $\Hom(\Dkq ,\Dkq) \to \Hom(\Dkq ,\Dkq)$ ci dessus correspond via cette correspondance au morphisme $\Dkq \overset{P \times}{\longrightarrow} \Dkq$ multiplication à gauche par $P$.
La suite exacte précédente s'écrit donc
\[ \xymatrix{
0 \ar[r] & \Dkq \ar[r]^{P \times} & \Dkq \ar[r] & \Ext^1_{\Dkq}(\M_k , \Dkq) \ar[r] & 0
} . \]
On en déduit que $\Ext^1_{\Dkq}(\M_k , \Dkq) \simeq \Dkq / P\cdot \Dkq$ en tant que $\Dkq$-module à droite.
On écrit $P = \sum_{\ell = 0}^n a_\ell \cdot (\varpi^k \partial)^\ell \in \Dkq(\X)$.
On définit son adjoint formel $P^\vee$ par $P^\vee := \sum_{\ell = 0}^n (-1)^\ell \cdot (\varpi^k \partial)^\ell \cdot a_\ell$.
On vérifie alors que
\[ \M_k^\vee = \Ext_{\Dkq}^1(\M , \Dkq) \otimes_{\O_{\X , \Q}} \omega_{\X , \Q}^{-1} \simeq \Dkq / P^\vee . \]
Comme $P$ est unitaire d'ordre $n$, $P^\vee$ est aussi unitaire d'ordre $n$ au signe pr\`es. Ainsi, $\M_k^\vee$ est un module à connexion de rang $n$.
Enfin \cite[lemme 4.13]{hallopeau1} implique que
\[ \M^\vee = \varprojlim_k \left( \Ext^1_{\Dkq}(\M_k , \Dkq) \otimes_{\O_{\X , \Q}} \omega_{\X , \Q}^{-1} \right) \]
est un module à connexion intégrable de rang $n$.
\end{proof}

\subsection{Cycles caractérisques}\label{section6.2}

On associe dans cette partie finale des multiplicités aux composantes irréductibles verticales des variétés caractéristiques des $\Di$-modules coadmissibles sous-holonomes.
Elles sont définies comme dans la section 1.3 de l'article \cite{adriano} de Abe-Marmora.
Les constructions et les preuves sont similaires.
On suppose dans un premier un temps que la courbe formelle $\X$ est affine munie d'une coordonnée globale fixée $y \in \O_\X(\X)$. On note $\partial$ la dérivation associée.
On commence par introduire et par rappeler quelques notations.
\begin{enumerate}
\item
Tout élément $S$ de $\F_{k , r}(\X)$ s'écrit uniquement sous la forme $S = \sum_{n\in \Z} a_n\cdot \varpi^{[n , k , r]} \partial^n$ avec $[n , k , r] := n \cdot(k \cdot \mathds{1}_{n \geq 0} + r \cdot \mathds{1}_{n < 0})$.
Son ordre est l'entier naturel $\Nb(S) = \max\{ n \in \Z : \| S \|_{k , r} = | a_n | \}$.
Lorsque $\| S \|_{k , r} =1$, $\Nb(S)$ est l'ordre de l'\'el\'ement $\bar{S} = (S \mod \varpi)$.
\item
On note $\F_{k , r}^\circ$ le sous-faisceau de $\F_{k , r}$ composé des éléments $S \in \F_{k , r}$ de norme inférieure ou égale à un, ie les coefficients de $S$ sont dans $\O_\X$.
C'est un faisceau d'algèbres sur $\V$ tel que $\F_{k , r} \simeq K \otimes_\V \F_{k , r}^\circ$.
\item
On désigne par $\F_{k , r}^{\, -}$ le sous-faisceau d'algèbres de $\F_{k , r}$ composé des éléments de la forme $S = \sum_{n < 0} a_n \cdot (\varpi^r \partial)^n$.
\item
On introduit la sous-algèbre commutative $\vp$ de $\F_{k , r}^\circ(\X)$ composée des éléments dont les coefficients sont dans $\V$ et $\vqp := \vp \otimes K$. Ainsi,
\[ \vqp := \left\{ \sum_{n\in \Z} \alpha_n\cdot \varpi^{[n , k , r]} \partial^n \in \F_{k , r}(\X)~: ~ \alpha_n \in K, ~ \lim_{n \to \pm \infty} \alpha_n = 0  \right\} . \]
\end{enumerate}

On équipe la $K$-algèbre $\vqp$ de la norme sous-multiplicative $\| \cdot \|_{k,r}$ de $\F_{k , r}(\X)$ et de la filtration induite par cette norme.
Il est clair que le gradué associé $\gr(\vqp) \simeq \gr(K)[ \xi_k , \xi_r] / [\xi_k \cdot \xi_r]$ est noethérien.
Ainsi, \cite[proposition 1.1]{schneider} implique que l'anneau $\vqp$ est noethérien.
Soit $\M_{k , r}$ un $\F_{k , r}$-module cohérent qui est également de type fini en tant que $\vqp$-module.
Puisque $\vqp$ est cohérent et $\M_{k , r}$ est de type fini sur $\vqp$, c'est un module cohérent sur $\vqp$.
En conséquence, la norme sur $\M_{k , r}$ provenant d'une présentation en tant que $\vqp$-module est complète.
De plus, comme la norme sur $\vqp$ est sous-multiplicative, toutes ces normes sur $\M_{k , r}$ sont équivalentes.
On note $\M_{k , r}^{K\langle \partial \rangle}$ le module $\M_{k , r}$ considéré comme $\vqp$-module cohérent.
De même, le module $\M_{k , r}$ est complet pour toute norme provenant d'une présentation en tant que $\F_{k , r}$-module cohérent et ces normes sont toutes équivalentes.
On prouve dans le prochain lemme que l'application $\vqp$-linéaire bijective $\M_{k , r}^{K\langle \partial \rangle} \longrightarrow \M_{k , r}$ est un homéomorphisme.

\begin{lemma}
Soit $\M_{k , r}$ comme ci-dessus. Alors les topologies respectives de $\M_{k , r}$ en tant que $\F_{k , r}$-module et $\vqp$-module coincident.
\end{lemma}
\begin{proof}
Puisque le produit tensoriel $\F_{k , r} \otimes_{\vqp} \bullet$ est exact à droite, une présentation de $\M_{k , r}^{K\langle \partial \rangle}$ en tant que $\vqp$-module cohérent fournit une présentation de $\M_{k , r}$ en tant que $\F_{k , r}$-module cohérent après tensorisation.
Comme l'application injective $\vqp \hookrightarrow \F_{k , r}$ est continue, le diagramme commutatif
\[ \xymatrix{
(\F_{k , r})^m \ar[r] & (\F_{k , r})^n \ar[r] & \M_{k , r} \ar[r] & 0 \\
(\vqp)^m \ar[r] \ar@{^{(}->}[u] & (\vqp)^n \ar[r] \ar@{^{(}->}[u] & \M_{k , r}^{K\langle \partial \rangle}  \ar[r] \ar[u] & 0
} \]
implique que l'application bijective $\M_{k , r}^{K\langle \partial \rangle}  \longrightarrow \M_{k , r}$ est continue.
Il résulte du théorème de l'application ouverte entre espaces de Banach que l'application $\M_{k , r}^{K\langle \partial \rangle}  \longrightarrow \M_{k , r}$ est un homéomorphisme.
En particulier, les topologies de $\M_{k , r}$ considérées respectivement comme $\vqp$-module et $\F_{k , r}$-module sont \'equivalentes.
\end{proof}

\subsubsection*{Multiplicités verticales pour les $\Di$-modules sous-holonomes}

On d\'emontre maintenant que si $\M = \varprojlim_k \M_k$ est un $\Di$-module sous-holonome, alors pour tout point $x$ du support $\suppi(\M)$ et pour $k \geq r \geq 1$ assez grands, les $\F_{k , r}$-modules cohérents $\M_{k , r}$ sont des $\vqp$-modules libres de rang fini au voisinage de $x$ tels que ces rangs se stabilisent.
La multiplicité verticale de $\M$ en $x$ est alors ce rang commun.
On rappelle que la courbe formelle $\X$ est supposée affine munie d'une coordonnée locale $y$.
Soit $\M_k$ un $\Dkq$-module cohérent tel que $\dim\supp(\M_{k , r}) = 0$.
Comme $\Car(\M_k) \subset \supp(\M_{k , r})$, le module $\M_k$ est holonome et nous savons par la proposition \ref{prop3.22} que $\M_k \simeq \Dkq / \I_k$ pour un certain idéal à gauche non nul $\I_k$ de $\Dkq$.
Soit $P_1 , \dots , P_s$ une base de division de $\I_k$ en $x$ telle que définie dans \cite[section 2.4]{hallopeau1}.
Il s'agit en pratique d'une famille normalis\'ee g\'en\'eratrice optimale pour l'ordre $\Nb$ et pour les valuations des coefficients dominants.
On suppose que les opérateurs $P_i$ sont dans $\I_k(\X)$ quitte à réduire la courbe affine $\X$.
On note $\I_k^\circ$ le faisceau d'idéaux de $\Dk$ engendré par $P_1 , \dots , P_s$.
Alors le $\Dk$-module cohérent $\M_k^\circ := \Dk / \I_k^\circ$ est sans $\varpi$-torsion et $\M_k \simeq \M_k^\circ \otimes_\V K$.
La \og division\fg \, suivante est une adaptation du lemme 1.3.1 de \cite{adriano} dans le cas des microlocalisés $\F_{k , r}$.

\begin{lemma}\label{lemme4.12}
On suppose que $\supp(\M_{k , r}) = \{ x \}$. Soit $t$ un paramètre local de $\X$ en $x$ et $U = \X \backslash \{ x \}$.
Alors, il existe un entier $N \in \N$ et des éléments $R \in \F_{k , r}^\circ(U)$ et $S \in \F_{k , r}^{\,-}(U)$ tels que $(t^N - \varpi \cdot R - S) \in \I_k^\circ(\X)$.
\end{lemma}
\begin{proof}
Soit $P_1 , \dots , P_s$ une base de division de $\I_k^\circ$ comme précédemment.
L'idéal à gauche $\F_{k , r}^\circ \cdot \I_k^\circ$ est également engendré par les éléments $P_1 , \dots , P_s$. L'hypothèse $\supp(\M_{k , r}) = \{ x \}$ implique que $(\F_{k , r}^\circ)_{|U} \cdot (\I_k^\circ)_{|U} = (\F_{k , r}^\circ)_{|U}$.
En conséquence, il existe $Q_1 , \dots , Q_s \in \F_{k , r}^\circ(U)$ tels que $Q_1 \cdot P_1 + \dots + Q_s \cdot P_s = 1$.
Pour toute section $f \in \O_\X(U)$, il est possible de trouver un entier $n \in \N$ tel que $\bar{t}^n \cdot \bar{f} \in \O_\X(\X)$.
En d'autres termes, $t^n \cdot f \in \O_\X(\X) + \varpi \cdot \O_\X(U)$.
Écrivons $Q_i = \sum_{n \geq 0} a_n \cdot (\varpi^k \partial)^n + \sum_{n < 0} a_n \cdot (\varpi^r \partial)^n$ et notons $N_i = \max\{ v_{\O_{\X , x}}(a_n) :  n \in \{ 0 , \dots , \Nb(Q_i)\} \} \in \N$.
Alors
\[
t^{N_i} \cdot Q_i  =
\underbrace{\sum_{n > \Nb(Q_i)} t^{N_i}  a_n \cdot (\varpi^k \partial)^n}_{\in ~ \varpi \cdot \Dk(U)} + 
\underbrace{\sum_{n = 0}^{\Nb(Q_i)} t^{N_i}  a_n \cdot (\varpi^k \partial)^n}_{\in ~ \Dk(\X) + \varpi \cdot \Dk(U)} +
\underbrace{\sum_{n < 0} t^{N_i} a_n \cdot (\varpi^r \partial)^n}_{\in ~ \F_{k , r}^{\,-}(U)} .
\]
On note $N = \max\{ N_i : i \in \{ 1 , \dots , s \} \}$.
Puique que $Q_1 \cdot P_1 + \dots + Q_s \cdot P_s = 1$, il existe $Q_1', \dots , Q_s' \in \Dk(\X)$, $R \in \F_{k , r}^\circ(U)$ et $S \in \F_{k , r}^{\,-}(U)$ tels que $t^N = \sum_{i=1}^s Q_i' \cdot P_i + \varpi \cdot R + S$.
\end{proof}

\begin{prop}
Soit $\M_k$ un $\Dkq$-module cohérent tel que $\supp(\M_{k , r}) = \{x\}$.
Alors le $\F_{k , r}(\X)$-module cohérent $\M_{k , r}(\X)$ est de type fini sur $\vqp$.
\end{prop}
\begin{proof}
En gardant les notations précédentes, nous savons grâce au lemme \ref{lemme4.12} qu'il existe un entier $N \in \N$ accompagné d'un élément $S \in \varpi \cdot \F_{k , r}^\circ(U) +\F_{k , r}^{\,-}(U)$ tel que $t^N - S \in \I_k^\circ(\X)$.
Soit $M := \F_{k , r}(\X) / (t^N - S)$. En considérant la surjection $M \twoheadrightarrow \M_{k , r}(\X)$, il suffit de prouver que $M$ est de type fini sur $\vqp$ pour obtenir la proposition.
Puisque $M$ et $\vqp$ sont complets pour la topologie $\varpi$-adique tout en étant sans $\varpi$-torsion, on peut réduire la preuve à montrer que $M / \varpi$ est fini sur $\vqp / \varpi$.
De plus, $M / \varpi$ et $\vqp / \varpi$ sont complets pour la filtration induite par l'ordre des opérateurs différentiels.
Ainsi, il suffit de prouver que $\gr(M / \varpi)$ est fini sur $\gr(\vqp / \varpi)$, où $\gr$ désigne le gradué pour l'ordre.
Puisque $S \in \F_{k , r}^{\,-}(U)$, on a $\gr(M / \varpi) = \gr(\F_{X , k , r} / (\bar{t}^N))$, où $\F_{X , k , r} = \F_{k , r}^\circ \otimes_\V \kappa$.
Enfin, on obtient comme dans \cite[lemme 1.3.2]{adriano} que $\gr(M / \varpi)$ est fini sur $gr(\vqp / \varpi)$.
\end{proof}

On démontre enfin, en utilisant le même argument que dans \cite{adriano}, que si un $\F_{k , r}$-module cohérent $\M_{k , r}$ est tel que $\M_{k , r}(\X)$ soit de type fini en tant que $\vqp$-module, alors $\M_{k , r}(\X)$ est un $\vqp$-module libre de rang fini.
On rappelle que $\X$ est suppos\'ee \^etre une courbe affine mumie d'une coordonnée globale fixée $y \in \O_\X(\X)$ et que $\partial$ est la dérivation associée.
Pour des nombres réels $0 < a \leq b < \infty$, on considère l'anneau des fonctions analytiques en la variable $z$ :
\[ \B_z ([a,b]) := \left\{ \sum_{n \in \Z} \alpha_n \cdot z^n : ~ \alpha_i \in K, ~ \lim_{n \to + \infty} |\alpha_n| \cdot b^n = 0, ~ \lim_{n \to - \infty} |\alpha_n| \cdot a^n = 0  \right\}. \]
C'est un sous-anneau de l'anneau $\a_{z , K}(\{ a , b])$ défini dans \cite[partie 1.3.6]{adriano}.
En conséquence, on en déduit que le nombre de pentes des graphes des éléments de $\B_z ([a,b])$ est fini, et l'on dispose d'un théorème de préparation de Weierstrass pour $\B_z ([a,b])$.
Ceci implique que $\B_z ([a,b])$ est un anneau principal.
Dans les définitions de $\F_{k , r}$ et $\vqp$, $k \geq r \geq 1$.
En particulier,  $0 < |\varpi|^{-r} \leq |\varpi|^{-k}$. Le lemme suivant découle du fait que la $K$-algèbre $\vqp$ est commutative.

\begin{lemma}\label{lemme4.14}
Nous avons un isomorphisme de $K$-algèbres
\[ \vqp \overset{\simeq}{\longrightarrow} \B_z([ |\varpi|^{-r} , |\varpi|^{-k}]) \]
envoyant $\partial$ sur $z$.
En particulier, $\vqp$ est un anneau principal.
\end{lemma}

Soit $\M_{k , r}$ un $\F_{k , r}$-module cohérent tel que $\M_{k , r}(\X)$ soit un $\vqp$-module de type fini.
Via l'isomorphisme du lemme \ref{lemme4.14}, $\M_{k , r}(\X)$ est un $\B_z([ |\varpi|^{-r} , |\varpi|^{-k}])$-module de type fini.
On équipe ce module de la connexion donnée par $\nabla(m) := (-y m)\otimes dz$. Comme l'action de $z$ sur $\M_{k , r}(\X)$ coïncide avec celle de $\partial$, il s'agit d'une connexion.
En effet, on a  $[ \partial , y ] = \partial \cdot y - y \cdot \partial = 1$ et
\begin{align*}
\nabla(z\cdot m) & = \nabla(\partial \cdot m) = (-y \partial m)\otimes dz  = (1 - \partial y) m\otimes dz  = \partial m\otimes dz - \partial ym \otimes dz  \\
& =  \partial m\otimes dz + z \cdot( -ym \otimes dz) .
\end{align*}
Puisque $\B_z([ |\varpi|^{-r} , |\varpi|^{-k}])$ est un anneau principal et $\M_{k , r}(\X)$ est un $\B_z([ |\varpi|^{-r} , |\varpi|^{-k}])$-module fini, nous déduisons du lemme 1.3.10 de \cite{adriano} que $\M_{k , r}(\X)$ est un $\B_z([ |\varpi|^{-r} , |\varpi|^{-k}])$-module libre de rang fini.

\begin{cor}\label{cor4.15}
Soit $\M_k$ un $\Dkq$-module cohérent tel que $\supp(\M_{k , r}) = \{x\}$.
Alors $\M_{k , r}(\X)$ est un $\vqp$-module libre de rang fini noté $\text{rk}_x(\M_{k , r})$.
\end{cor}

Il reste à établir que les rangs que nous obtenons ainsi ne dépendent ni de $k$ ni de $r$ pour des entiers suffisamment grands.

\begin{lemma}\label{lemme4.16}
Nous avons des isomorphismes canoniques de bi-modules
\[ \F_{k , r} \simeq \vqp \, \widehat{\otimes}_{K\langle\partial\rangle^{(k+1,r)}} \, \F_{k+1 , r} \hspace{0.3cm} \text{et} \hspace{0.3cm}  \F_{k , r+1} \simeq K\langle\partial\rangle^{(k, r+1)} \, \widehat{\otimes}_{\vqp} \, \F_{k , r} . \]
Le produit tensoriel complété $\widehat{\otimes}$ est pris pour la topologie $\varpi$-adique.
\end{lemma}
\begin{proof}
On démontre seulement le premier isomorphisme de bi-modules $\F_{k , r} \simeq \vqp \, \widehat{\otimes}_{K\langle\partial\rangle^{(k+1,r)}} \F_{k+1 , r} $.
L'argument pour le second est similaire.
On rappelle que $\F_{k+1 , r}$ et $K\langle\partial\rangle^{(k+1,r)}$ sont des sous-faisceaux de $\F_{k , r}$.
Comme $\F_{k , r}$ est complet pour la topologie $\varpi$-adique, on obtient une application naturelle $\epsilon : \vqp \, \widehat{\otimes}_{K\langle\partial\rangle^{(k+1,r)}} \F_{k+1 , r} \to \F_{k , r}$ envoyant $P \otimes S$ sur $P \cdot S$.
Puisque le produit tensoriel commute avec les images directes, il suffit de montrer que pour tout ouvert affine $U$ de $\X$, l'application
\[ \epsilon : \vqp \, \widehat{\otimes}_{K\langle\partial\rangle^{(k+1,r)}} \F_{k+1 , r}(U) \to \F_{k , r}(U) \]
est un isomorphisme.
Soit
\[ S = \underbrace{\sum_{n = 0}^\infty a_n \cdot (\varpi^k \partial)^n}_{P^+ \in \Dkq(U)} + \underbrace{\sum_{n = 1}^\infty a_{-n} \cdot (\varpi^r \partial)^{-n}}_{P^- \in \F_{k , r}^{\,-}(U)} \in \F_{k , r}(U). \]
On a $P^- \in \F_{k , r}^{\,-}(U) = \F_{k+1 , r}^{\,-}(U)$. De plus,
$P^+ = \lim_{m \to \infty} P_m$ où
\[ P_m = \sum_{n = 0}^m a_n \cdot (\varpi^k \partial)^n = \epsilon \left(\sum_{n = 0}^m a_n \otimes (\varpi^k \partial)^n\right)  \]
avec $\sum_{n = 0}^m a_n \otimes (\varpi^k \partial)^n \in \vqp  \otimes_{K\langle\partial\rangle_\X^{(k+1,r)}} \F_{\X , k+1 , r}$.
On peut vérifier que la suite $\left(\sum_{n = 0}^m a_n \otimes (\varpi^k \partial)^n\right)_m$ converge pour la topologie $\varpi$-adique vers $\sum_{n = 0}^\infty a_n \otimes (\varpi^k \partial)^n$.
Ainsi,
\[ S = \epsilon\left(\sum_{n = 0}^\infty \left(a_n \otimes (\varpi^k \partial)^n\right) + 1 \otimes P^-\right) . \]
On obtient donc une application réciproque à droite bien définie pour $\epsilon$ :
\[ \varphi : \F_{k , r}(U) \to \vqp \, \widehat{\otimes}_{K\langle\partial\rangle^{(k+1,r)}} \F_{k+1 , r}(U) \]
\[ \sum_{n = 0}^\infty a_n \cdot (\varpi^k \partial)^n + \sum_{n = 1}^\infty a_{-n} \cdot (\varpi^r \partial)^{-n} \mapsto \sum_{n = 0}^\infty a_n \otimes (\varpi^k \partial)^n + 1 \otimes \sum_{n = 1}^\infty a_{-n} \cdot (\varpi^r \partial)^{-n} .  \]
Un calcul assez similaire montre que $\varphi \circ \epsilon = \id$.
\end{proof}

Soit $\M = \varprojlim_k \M_k$ un $\Di$-module coadmissible tel que $\suppi(\M) = \{x\}$.
Comme $\suppi(\M) =\bigcap_{r \geq 1} \supp (\Mr) $, il existe un entier $r_0 \geq 1$ tel que, pour tout $r \geq r_0$, $\suppi(\M_\X) = \supp(\Mr) = \{x\}$.
Puisque $\supp (\Mr) = \overline{\bigcup_{k \geq r} \supp (\M_{k , r})}$ et la suite des supports $(\supp (\M_{k , r}))_{k \geq r}$ est croissante, on en déduit pour tout niveau de congruence $k \geq r$ suffisamment grand que $\supp(\M_{k , r}) = \{ x \}$
Ainsi, d'après le corollaire \ref{cor4.15} nous savons que pour $k \geq r$ suffisamment grand, $\M_{k , r}$ est localement autour de $x$ un $\vqp$-module libre de rang $\text{rk}_x(\M_{k , r})$.

\begin{prop}
Soit $\M = \varprojlim_k \M_k$ un $\Di$-module coadmissible tel que $\suppi(\M) = \{x\}$ et $r \geq r_0$ comme ci-dessus.
\begin{enumerate}
\item
On dispose d'un isomorphisme $\M_{k , r} \simeq \vqp \otimes_{K\langle \partial \rangle^{k+1 , r}} \M_{k+1 , r}$
de $K\langle \partial \rangle^{k , r} $-modules. En particulier, pour $k$ suffisamment grand, $\text{rk}_x(\M_{k+1 , r}) = \text{rk}_x(\M_{k , r})$.
Ce rang est noté $\text{rk}_x(\Mr)$.
\item
Pour $k > r$, on a un isomorphisme $\M_{k, r+1} \simeq K\langle \partial \rangle^{k , r+1} \otimes_{\vqp} \M_{k , r}$
de $K\langle \partial \rangle^{k , r+1} $-modules. Il s'ensuit que $\text{rk}_x(\M_{k , r+1}) \leq \text{rk}_x(\M_{k , r})$.
\item
Pour $r \geq r_0$ assez grand, on a $\text{rk}_x(\Mr) = \text{rk}_x(\tilde{\M}_{\infty , r+1})$.
\end{enumerate}
\end{prop}
\begin{proof}
On déduit du lemme \ref{lemme4.16} les isomorphismes suivants :
\begin{align*}
\M_{k , r} & \simeq \F_{k , r} \otimes_{\F_{k+1 , r}} \M_{k+1 , r} \simeq \F_{k , r} \widehat{\otimes}_{\F{k+1 , r}} \,  \M_{k+1 , r} \\
& \simeq \vqp \widehat{\otimes}_{K\langle\partial\rangle^{(k+1,r)}} \F_{k+1 , r} \widehat{\otimes}_{\F{k+1 , r}} \,  \M_{k+1 , r} \\
& \simeq \vqp \widehat{\otimes}_{K\langle\partial\rangle^{(k+1,r)}} \,  \M_{k+1 , r} .
\end{align*}
Puisque $\M_{k+1 , r}$ est de type fini sur $K\langle\partial\rangle^{(k+1,r)}$, on a
\[ \vqp \widehat{\otimes}_{K\langle\partial\rangle^{(k+1,r)}} \M_{k+1 , r} \simeq \vqp  \otimes_{K\langle\partial\rangle^{(k+1,r)}} \M_{k+1 , r} . \]
On obtient le premier isomorphisme de la proposition. La preuve du second est analogue.
On prouve comme dans la proposition \ref{prop2.28} que $\vqp$ est plat sur $K\langle\partial\rangle^{(k+1,r)}$.
Ce fait, ainsi que l'isomorphisme de $\vqp$-modules $\M_{k , r} \simeq \vqp \otimes_{K\langle \partial \rangle^{(k+1 , r)}} \M_{k+1 , r}$, impliquent que $\text{rk}_x(\M_{k+1 , r}) = \text{rk}_x(\M_{k , r})$.
De plus, on déduit du second isomorphisme $\M_{k, r+1} \simeq K\langle \partial \rangle^{(k , r+1)} \otimes_{\vqp} \M_{k , r}$ que $\text{rk}_x(\M_{k , r+1}) \leq \text{rk}_x(\M_{k , r})$.
Ainsi, $\text{rk}_x(\Mr) \leq \text{rk}_x(\M_{\X , \infty , r+1})$.
Puisque $(\text{rk}_x(\Mr))_{r \geq r_0}$ est une suite décroissante d'entiers positifs, elle est stationnaire.
Cela prouve le troisième point de la proposition.
\end{proof}

Cette proposition nous dit que si $\M = \varprojlim_k \M_k$ est un $\Di$-module coadmissible tel que $\suppi(\M) = \{x\}$, alors les rangs $\text{rk}_x(\Mr)$ se stabilisent pour $r$ suffisamment grand.

\begin{definition}
Soit $\M = \varprojlim_k \M_k$ un $\Di$-module coadmissible tel que le support $\suppi(\M)$ soit réduit à un point fermé $\{x\}$.
La multiplicité $m_x(\M)$ de $\M$ en $x$ est définie comme étant le rang commun en $x$ des modules $\M_{k , r}$ pour $k \geq r$ suffisamment grands.
\end{definition}

Si $x$ est un point fermé de la courbe affine lisse $\X$ n'appartenant pas au support $\suppi(\M)$, on pose $m_x(\M_\X) := 0$.

\begin{lemma}\label{lemme4.20}
Soit $0 \to \Nn \to \M \to \L \to 0$ une suite exacte courte de $\Di$-modules coadmissibles telle que $\suppi(\M) = \{ x \}$.
Alors $\suppi(\Nn) \cup \suppi(\L) = \{ x \}$ et $m_x(\M) = m_x(\Nn) + m_x(\L)$.
\end{lemma}
\begin{proof}
Soit $r_0 \geq 1$ tel que $\supp(\Mr) = \{ x \}$ pour tout $r \geq r_0$.
Alors $\supp(\M_{k , r}) = \{ x \}$ pour tout $k \geq r$.
Puisque la suite $0 \to \Nn \to \M \to \L \to 0$ de modules coadmissibles est exacte, la suite de $\Dkq$-modules cohérents $0 \to \Nn_k \to \M_k \to \L_k \to 0$  est exacte pour tout niveau de congruence $k \geq 0$.
Il a d\'ej\`a \'et\'e vu que $\supp(\M_{k , r}) = \supp(\Nn_{k , r}) \cup \supp(\L_{k , r})$.
Il s'ensuit que $\suppi(\Nn) \cup \suppi(\L) = \{ x \}$.
En conséquence du corollaire \ref{cor4.15}, les modules $\F_{k , r}(\X)$-modules $\M_{k , r}(\X)$, $\Nn_{k , r}(\X)$ et $\L_{k , r}(\X)$ sont des $\vqp$-modules libres de rang fini.
Puisque la courbe formelle lisse $\X$ est supposée affine, le théorème A pour les $\Dkq$-modules cohérents (\cite[corollaire 2.2.15]{huyghe}) implique que $0 \to \Nn_{k}(\X) \to \M_k(\X) \to \L_k(\X) \to 0$ est une suite exacte courte de $\Dkq(\X)$-modules.
Par exactitude de $\F_{k , r}(\X)$ sur $\Dkq(\X)$ (proposition \ref{prop2.28}), on obtient une suite exacte $0 \to \Nn_{k , r}(\X) \to \M_{k , r}(\X) \to \L_{k, r}(\X) \to 0$ de $\F_{k , r}(\X)$-modules.
Cette suite est également une suite exacte de $\vqp$-modules en considérant $\M_{k , r}(\X)$, $\Nn_{k , r}(\X)$ et $\L_{k , r}(\X)$ comme des $\vqp$-modules.
Ainsi,
\[ \text{rk}_x(\M_{k , r}(\X)) = \text{rk}_x(\Nn_{k , r}(\X)) + \text{rk}_x(\L_{k , r}(\X)) . \]
Puisque pour tout $k >> r \geq r_0$, nous avons $m_x(\M) = \text{rk}_x(\M_{k , r}(\X))$, $m_x(\Nn) = \text{rk}_x(\Nn_{k , r}(\X))$ et $m_x(\L) = \text{rk}_x(\L_{k , r}(\X))$.
Nous obtenons donc l'additivité des multiplicités en $x$.
\end{proof}

\subsubsection*{Cycles caractéristiques}

Retournons au cas d'une courbe formelle lisse $\X$ non nécessairement affine. Soit $\M$ un $\Di$-module coadmissible sous-holonome.
Alors $\dim(\suppi(\M)) = 0$, et la variété caractéristique $\Car(\M)$ est formée de composantes irréductibles verticales et, potentiellement d'une composante horizontale égale à la section nulle de l'espace cotangent $T^*\X$.
Pour un point fermé $x$ du support $\suppi(\M)$, on note $U$ un ouvert affine de $\X$ pour lequel $U \cap \suppi(\M) = \{x\}$.
En travaillant sur $U$, on peut associer à $\M$ une multiplicité $m_x(\M) \in \N^*$ en $x$.
Les multiplicités que nous obtenons ainsi sont additives pour les suites exactes courtes de $\Di$-modules coadmissibles sous-holonomes.
Soit $C$ une composante irréductible verticale de la variété caractéristique $\Car(\M)$.
L'abscisse de $C$ est un point $x_C$ du support $\suppi(\M)$ ; on appelle $m_C(\M)$ la multiplicité $m_{x_C}(\M)$ de la composante irréductible $C$.
On rappelle que $m_0(\M)$ est la multiplicité de la composante horizontale de $\Car(\M)$.
C'est le rang du $\O_{\X, \Q}$-module localement libre $(\M)_{|\X \backslash \suppi(\M)}$.
Si $\M = 0$, alors toutes les multiplicités de $\M$ sont nulles.

\begin{definition}
Soit $\M$ un $\Di$-module coadmissible sous-holonome.
On note $\Irr(\M)$ l'ensemble des composantes irréductibles verticales de $\Car(\M)$. On définit le cycle caractéristique de $\M$ comme la somme formelle finie \`a coefficients entier naturels
\[
\CC(\M) := m_0(\M) \cdot \X + \sum_{C \in \Irr(\M)} m_C(\M) \cdot C,
\]
où l'on identifie toujours $\X$ avec la section nulle de $T^*\X$. Pour $\M = 0$, on a $\CC(\M) = 0$.
\end{definition}

\begin{example}\,
\begin{enumerate}
\item
Soit $P$ un point ferm\'e de $\X$, $u : P \hookrightarrow \X$ l'inclusion associ\'ee et $t$ un paramètre local en $P$.
Alors $u_+ \O_P := \varprojlim_k \left( \Dkq / \Dkq \cdot t \right)$ est sous holonome et sa variété caractéristique consiste en une composante irréductible verticale d'abscisse $P$ d'après l'exemple \ref{ex6.10}.
De plus, la multiplicité associée est égale \`a un.
\item
Soit $\X = \Spf \V\langle t \rangle$, $\mathbb{D} = \X_K = \Sp K \langle t \rangle$ le disque analytique rigide de rayon un et $\text{sp} : \mathbb{D} \to \X$ le morphisme de sp\'ecialisation.
On dispose sur $\mathbb{D}$ du faisceau $\wideparen{\D}_\mathbb{D}$ des opérateurs différentiels introduits dans \cite{ardakov} par Ardakov-Wadsley.
On note $\mathbb{D}^* = \mathbb{D} \backslash \{0\}$ le disque \'epoint\'e et $j : \mathbb{D}^* \hookrightarrow \mathbb{D}$ l'inclusion associée.
Alors $j_* \O_{\mathbb{D}^*}$ est un $\wideparen{\D}_\mathbb{D}$-module coadmissible faiblement holonome d'apr\`es \cite[section 10.5]{ABW2}.
En identifiant $\mathbb{D}$ avec l'espace de Zariski-Riemann associ\'e \`a $\X$ et en utilisant l'équivalence de cat\'egories \cite[proposition 3.1.12]{huyghe},
on obtient sur $\X$ un $\Di$-module coadmissible $ \M = \text{sp}_*\left(j_* \O_{\mathbb{D}^*}\right) \simeq \varprojlim_k \left(\Dkq / (t \cdot \partial_t +1 \right)$.
Ce dernier est sous-holonome.
En effet, $\supp(\F_{k , r} \otimes_{\Dkq} \Dkq / (t \cdot \partial_t +1)) = \{ t =0\}$ et donc $\suppi (\M) = \{ t =0\}$.
Par ailleurs, les $\F_{k , r} $-modules cohérents $\F_{k , r} \otimes_{\Dkq} \Dkq / (t \cdot \partial_t +1))$ sont libres de rang un sur $\vqp$ engendr\'es par $t$.
Ainsi, la multiplicité de $\M$ en z\'ero vaut un et sa variété caractéristique contient exactement une composante horizontale de multiplicité une puis une composante verticale passant par l'origine de multiplicité aussi égale \`a un.
On consid\`ere le $\Di$-module
\[ \wideparen{\O_{\X , \Q}[1/t]} = \left\{ \sum_{n=0}^\infty a_n \cdot t^{-n},~ \forall R>0, ~ |a_n| \cdot R^n \to 0 \right\} \]
donn\'e par l'action naturelle de la d\'erivation $\partial_t$ sur $\frac{1}{t}$.
En tant qu'algèbre de Fr\'echet, $\wideparen{\O_{\X , \Q}[1/t]} \simeq \varprojlim_k \widehat{\O_{\X , \Q}[\varpi^k / t]}$.
Les algèbres $\widehat{\O_{\X , \Q}[\varpi^k / t]}$ ont une structure naturelle de $\Dkq$-modules induisant la structure de $\Di$-module sur $\wideparen{\O_{\X , \Q}[1/t]}$.
Par ailleurs, on dispose d'injections $\Dkq$-lin\'eaires $\widehat{\O_{\X , \Q}[\varpi^k / t]} \hookrightarrow \Dkq / (t \cdot \partial_t +1)$ envoyant $(\frac{1}{t})^n$ sur $(-\partial_t)^n$ car $t\cdot \partial_t =-1$ dans $\Dkq / (t \cdot \partial_t +1)$.
L'algèbre $\widehat{\O_{\X , \Q}[\varpi^k / t]}$ admet comme sous-module la connexion $\O_{\X , \Q}$ de multiplicité une sans être égale a cette dernière : la somme des multiplicités de $\widehat{\O_{\X , \Q}[\varpi^k / t]}$ vaut donc au moins deux.
Puisque le module $\Dkq / (t \cdot \partial_t +1)$ a une multiplicité globale égale \`a deux, nécessairement $\widehat{\O_{\X , \Q}[\varpi^k / t]} \simeq \Dkq / (t \cdot \partial_t +1)$ en tant que $\Dkq$-modules.
Ces isomorphismes étant donn\'es par la flèche $(\frac{1}{t})^n \mapsto (-\partial_t)^n$ ne dependant pas de $k$, ils passent \`a la limite projective pour donner un isomorphisme de $\Di$-modules $\M \simeq \wideparen{\O_{\X , \Q}[1/t]}$.
En particulier, le $\Di$-module $\wideparen{\O_{\X , \Q}[1/t]}$ est coadmissible et sous-holonome.
\end{enumerate}
\end{example}

On obtient la proposition suivante grâce au lemme \ref{lemme4.20}.

\begin{prop}\label{prop4.22}
Soit $0 \to \Nn \to \M \to \L \to 0$ une suite exacte de $\Di$-modules coadmissibles sous-holonomes.
Alors $\CC(\M) = \CC(\Nn) + \CC(\L)$. De plus, un $\Di$-module coadmissible sous-holonome $\M$ est nul si et seulement si $\CC(\M) = 0$.
\end{prop}

Une conséquence directe de cette proposition est le résultat suivant ; voir la preuve de \cite[proposition 3.19]{hallopeau1} pour plus de d\'etails.

\begin{cor}\label{cor1.16}
Tout $\Di$-module sous-holonome coadmissible $\M$ est de longueur finie, inférieure ou égale à $m(\M) := m_0(\M) + \sum_{x \in \suppi(\M)} m_x(\M) \in \N$.
\end{cor}

On termine cet article par démontrer que la variété caractéristique $\Car(\M)$ d'un $\Di$-module sous-holonome $\M = \varprojlim_k \M_k$ coincide avec les variétés caractéristiques $\Car(\M_k)$ des $\Dkq$-modules holonomes $\M_k$ pour $k$ suffisamment grand.
En particulier, les variétés caractéristiques $\Car(\M_k)$ se stabilisent lorsque $k \to + \infty$.
Ce résultat est clairement faux dans le cas général. En effet, si $P \in \Di$ est un opérateur différentiel d'ordre infini, alors les modules $\M_k = \Dkq / P$ sont holonomes mais $\dim(\Car(\Di / P)) = T^*\X$.
De plus, les abscisses des composantes verticales des variétés caractéristiques $\Car(\M_k)$ sont données par les zéros du coefficient d'ordre $\Nb(P)$ de $P$.
Comme $\Nb(P) \to + \infty$ lorsque $k \to \infty$, il n'y a a priori aucun moyen de contrôler ces composantes verticales pour $k$ variable.

\begin{prop}
Un $\Di$-module coadmissible $\M = \varprojlim_k \M_k$ est sous-holonome si et seulement si il existe un niveau de congruence $k$ pour lequel le morphisme $\M \to \M_k$ est injectif et si le $\Dkq$-module cohérent $\M_k$ est holonome.
Dans ce cas, les morphismes $\M \to \M_k$ sont injectifs pour tout $k$ suffisamment grand.
\end{prop}
\begin{proof}
Supposons tout d'abord qu'il existe un niveau de congruence $k$ pour lequel le morphisme de transition $\M \to \M_k$ est injectif avec $\M_k$ holonome.
Soit alors $U$ un ouvert dense de $\X$ tel que $(\M_k)_{|U}$ soit une connexion intégrable, ie un $\O_{\X , \Q}$-module cohérent.
Il en découle que $\M_{|U}$ est aussi une connexion intégrable : $\M$ est donc sous-holonome.
Supposons maintenant que le module $\M$ soit sous-holonome. Alors les modules cohérents $\M_k$ sont holonomes.
Il reste à démontrer que les morphismes de transition $\Di$-linéaires $\M \to \M_k$ sont injectifs pour tout niveau de congruence $k$ suffisamment grand.
Soit $\Nn(k)$ le noyau du morphisme $\M \to \M_k$. Comme $\Nn(k)$ est localement un sous-module fermé de $\M$, c'est un $\Di$-module coadmissible.
Puisque $\M_k \simeq \Dkq \otimes_{\Di} \M$, on a $\Dkq \otimes_{\Di} \Nn(k) = 0$.
Par ailleurs, il est clair que $\Nn(k+1)$ est un sous-module coadmissible de $\Nn(k)$.
Le module $\M$ étant sous-holonome, les modules $\Nn(k)$ le sont aussi.
On déduit alors de la proposition \ref{prop4.22} que la suite des cycles caractéristique $(\CC(\Nn(k)))_k$ est décroissante.
Puisque que les coefficients des cycles caractéristiques sont des entiers positifs, cette suite est stationnaire.
Il en découle que la suite décroissante de $\Di$-modules coadmissibles $(\Nn(k))_k$ est stationnaire.
Autrement dit, il existe un niveau de congruence $k_0$ tel que pour tout $k \geq k_0$, $\Nn(k) = \Nn(k_0)$.
On rappelle que $\Dkq \otimes_{\Di} \Nn(k) = 0$. Ainsi, pour tout $k \geq k_0$ on a $\Dkq \otimes_{\Di} \Nn(k_0) = 0$.
Comme le $\Di$-module $\Nn(k_0)$ est coadmissible, il vient $\Nn(k_0) \simeq \varprojlim_{k \geq k_0} \Dkq \otimes_{\Di} \Nn(k_0) = 0$.
Les morphismes de transition $\M \to \M_k$ sont donc injectifs pour tout $k \geq k_0$.
\end{proof}

\begin{cor}
Soit $\M = \varprojlim_k \M_k$ un $\Di$-module sous-holonome.
Alors pour tout niveau de congruence $k$ suffisamment grand,
$\Car(\M) = \Car(\M_k)$ en tant qu'espaces topologiques.
\end{cor}
\begin{proof}
Il a déjà été vu que $\Car(\M_k) \subset \Car(\M)$ pour $k$ suffisamment grand.
On suppose le module $\M$ sous-holonome. La proposition précédente nous dit que pour les niveaux de congruence $k$ grands, $\M \hookrightarrow \M_k$.
En particulier, $\M_{|U}$ est une connexion dès que $(\M_k)_{|U}$ est une connexion pour un certain ouvert $U$ de $\X$.
Il en découle que le support $\suppi(\M)$ est contenu dans l'ensemble des abscisses des composantes verticales de $\Car(\M_k)$.
Autrement dit, $\Car(\M) \subset \Car(\M_k)$.
\end{proof}

\bibliographystyle{plain}
\bibliography{biblio.bib}

\end{document}